%% file: 0main.tex
\documentclass{article}

    \PassOptionsToPackage{numbers, compress}{natbib}

\usepackage[main,final]{neurips_2026}

\usepackage[utf8]{inputenc} 
\usepackage[T1]{fontenc}    
\usepackage{hyperref}       
\usepackage{url}            
\usepackage{booktabs}       
\usepackage{amsfonts}       
\usepackage{nicefrac}       
\usepackage{microtype}      
\usepackage{xcolor}         

\usepackage{algorithmic,algorithm} 

\usepackage{amsmath}
\usepackage{amssymb}
\usepackage{mathtools}
\usepackage{amsthm}
\usepackage[capitalize,noabbrev]{cleveref}

\usepackage{tabularx}
\usepackage{multirow}
\usepackage{subcaption}
\usepackage{wrapfig}

\theoremstyle{plain}
\newtheorem{theorem}{Theorem}[section]
\newtheorem{proposition}[theorem]{Proposition}
\newtheorem{lemma}[theorem]{Lemma}

\theoremstyle{definition}
\newtheorem{definition}[theorem]{Definition}
\newtheorem{assumption}[theorem]{Assumption}
\theoremstyle{remark}
\newtheorem{remark}[theorem]{Remark}

\usepackage[textsize=tiny]{todonotes}

\usepackage[normalem]{ulem} 

\newcommand{\bm}[1]{\boldsymbol{#1}}

\newcommand{\va}{{\mathbf{a}}}

\newcommand{\vg}{{\mathbf{g}}}

\newcommand{\vu}{{\mathbf{u}}}
\newcommand{\vv}{{\mathbf{v}}}
\newcommand{\vw}{{\mathbf{w}}}
\newcommand{\vx}{{\mathbf{x}}}
\newcommand{\vy}{{\mathbf{y}}}
\newcommand{\vz}{{\mathbf{z}}}

\newcommand{\vlam}{{\bm{\lambda}}}
\newcommand{\vtheta}{{\bm{\theta}}}

\newcommand{\RR}{\mathbb{R}} 

\newcommand{\dist}{\mathrm{dist}}    
\newcommand{\prox}{{\mathbf{prox}}} 
\newcommand{\dom}{{\mathrm{dom}}} 

\def\tx{{\hat x}}
\def\ty{{\hat y}}

\def\h{\bar h}

\DeclareMathOperator*{\argmin}{arg\,min} 
\DeclareMathOperator*{\argmax}{arg\,max} 

\newcommand{\weiran}[1]{\textcolor{red}{#1 (Weiran)}}

\newcommand{\bc}{\begin{center}}
\newcommand{\ec}{\end{center}}

\title{Penalty-Based First-Order Methods for Bilevel Optimization with Minimax and Constrained Lower-Level Problems}

\author{%
  Yiyang Shen\thanks{Equal contribution. Author ordering determined by a coin flip.}, Yutian He\footnotemark[1], Weiran Wang, Qihang Lin \\
  University of Iowa\\
  Iowa City, IA \\
  \texttt{\{yiyang-shen,yutian-he,weiran-wang,qihang-lin\}@uiowa.edu} \\
}

\begin{document}
\maketitle

\begin{abstract}
We study a class of bilevel optimization problems in which both the upper- and lower-level problems have minimax structures. This setting captures a broad range of emerging applications. Despite the extensive literature on bilevel optimization and minimax optimization separately, existing methods mainly focus on bilevel optimization with lower-level minimization problems, often under strong convexity assumptions, and are not directly applicable to the minimax lower-level setting considered here. To address this gap, we develop penalty-based first-order methods for bilevel minimax optimization without requiring strong convexity of the lower-level problem. In the deterministic setting, we establish that the proposed method finds a lifted $\epsilon$-KKT point with $\tilde{O}(\epsilon^{-4})$ oracle complexity. We further show that bilevel problems with convex constrained lower-level minimization can be reformulated as special cases of our framework via Lagrangian duality, leading to an $\tilde{O}(\epsilon^{-4})$ complexity bound that improves upon the existing $\tilde{O}(\epsilon^{-7})$ result when the lower-level Slater's condition holds. Finally, we extend our approach to the stochastic setting, where only stochastic gradient oracles are available, and prove that the proposed stochastic method finds a nearly lifted $\epsilon$-KKT point with $\tilde{O}(\epsilon^{-9})$ oracle complexity. 
\end{abstract}
{\let\thefootnote\relax\footnotetext{Code available at: \url{https://github.com/weiranwang19/minimax}}}

\input{1intro}

\input{2related}
\input{3methods}
\input{4exps}

\newpage
\bibliographystyle{plainnat}
\bibliography{opt}

\newpage
\appendix
\input{apdx}

\newpage
\input{checklist.tex}

\end{document}

%% file: 1intro.tex
\section{Introduction}\label{sec:intro}
In this paper, we consider the following bilevel optimization problem with minimax  upper- and lower-level problems:
\begin{align}\label{eq:bilevel}
    \min_{\vx_1,\vy_1,\vy_2}&~\left\{\max_{\vx_2} f(\vx_1, \vx_2, \vy_1, \vy_2)\Big |
    (\vy_1, \vy_2) \in \arg\min_{\vz_1} \max_{\vz_2} \tilde{f}(\vx_1, \vz_1, \vz_2)\right\},
\end{align}
where $(\vy_1, \vy_2) \in \arg\min\limits_{\vz_1} \max\limits_{\vz_2} \tilde{f}(\vx_1, \vz_1, \vz_2)$ means $(\vy_1, \vy_2)$ is a saddle-point of $ \tilde{f}(\vx_1, \cdot, \cdot)$, i.e.,
$$
\vy_1 \in \arg \min_{\vz_1}\tilde{f}(\vx_1, \vz_1, \vy_2) \quad\text{and}\quad \vy_2 \in \arg\max_{\vz_2} \tilde{f}(\vx_1, \vy_1, \vz_2).
$$
We assume  the saddle-point of $ \tilde{f}(\vx_1, \cdot, \cdot)$ always exists for any $\vx_1$ and assume
\begin{align}
    f(\vx_1, \vx_2, \vy_1, \vy_2)
    &= f_1(\vx_1, \vx_2, \vy_1, \vy_2) + f_2(\vx_1) - f_3(\vx_2),
    \label{eq:f_structure}\\
    \tilde{f}(\vx_1, \vy_1, \vy_2)
    &= \tilde{f}_1(\vx_1, \vy_1, \vy_2) + \tilde{f}_2(\vy_1) - \tilde{f}_3(\vy_2).
    \label{eq:tildef_structure}
\end{align}
where $f_2$, $f_3$, $\tilde{f}_{2}$, and $\tilde{f}_{3}$ are proper closed convex functions, $f_1$ is smooth and concave on $\vx_2$, and $\tilde{f}_1$ is smooth, convex in $\vy_1$ and concave in $\vy_2$.

Bilevel optimization has found numerous applications in machine learning, including but not limited to hyperparameter tuning~\citep{bennett2008bilevel,franceschi2018bilevel}, meta-learning~\citep{franceschi2018bilevel,bertinetto2018meta,rajeswaran2019meta}, reinforcement learning~\citep{hong2023two,yang2024bilevel,li2024learning}, and neural architecture search~\citep{liu2018darts}. 
On the other hand, minimax optimization problems also arise in a wide range of applications, such as robust learning~\citep{duchi2021learning}, distributionally robust optimization (DRO,~\citep{ben2013robust}), generative adversarial networks~\citep{goodfellow2014generative}, and adversarial training~\citep{madry2018towards}. 
Incorporating such minimax problems as the lower-level problem in bilevel optimization leads to a broad class of novel and practically relevant applications of~\eqref{eq:bilevel}. We will present the examples of~\eqref{eq:bilevel} from hyperparameter tuning for distributionally robust deep learning in~Section~\ref{sec:exp} with numerical results. Additional applications of \eqref{eq:bilevel} are given in Section~\ref{sec:applications}.

Despite the extensive literature on both bilevel optimization and minimax optimization, existing bilevel formulations typically assume that the lower-level problem is a single-objective minimization problem (often strongly convex), with or without constraints. 
The setting where the lower-level problem itself is a minimax problem has only been studied in a very few papers~\citep{zheng2026saddle,zhang2026stability,yaoovercoming}. Moreover, general first-order oracle complexity guarantees
for finding approximate KKT points of bilevel optimization
with minimax lower-level problems remain largely unexplored. To fill this gap, this work proposes both deterministic and stochastic first-order algorithms for \eqref{eq:bilevel} based on a value-function penalty reformulation. Such reformulation has been mainly studied for bilevel optimization with a lower-level minimization problem~\citep{shen2023penalty,yaoovercoming,yaoconstrained,jiang2024primal,kwon2024penalty,kwon2023fully,JMLR:v26:23-1104,lu2025solving, lu2024firstorder} and for bilevel optimization with a lower-level minimax problem~\citep{yaoovercoming,zheng2026saddle}.
We establish the (stochastic) gradient oracle complexity of the proposed methods for finding an $\epsilon$-KKT point of \eqref{eq:bilevel} (see Definition~\ref{def:epsilon_KKT}) in the deterministic case and for finding a nearly $\epsilon$-KKT point of \eqref{eq:bilevel} (see Definition~\ref{def:epsilon_KKT_nearly}) in the stochastic case.

Furthermore, we consider a bilevel optimization problem in which the lower-level problem is a convex minimization problem with convex nonlinear inequality constraints. In particular, we study the following problem
\begin{gather}\label{eq:blo_cons}
\min_{\vx_1, \vy_1 } \left\{f(\vx_1, \vy_1) \bigg| \vy_1 \in \arg \min_{\vz_1}  \left\{\bar{f}(\vx_1, \vz_1)\big|\bar{\vg}(\vx_1, \vz_1)\leq 0\right\}\right\},
\end{gather}
where $\bar{\vg}(\vx_1, \vy_1)=(\bar{g}_1(\vx_1, \vy_1),\dots,\bar{g}_l(\vx_1, \vy_1))\in\mathbb{R}^l$ and $\bar{g}_i(\vx_1, \vy_1)$ is smooth and convex in $\vy_1$ and satisfies Slater's condition for $i=1,\dots,l$. We also assume  
     \begin{align}
        \label{eq:f_structure_cons}
            f(\vx_1,\vy_1) =~ f_1(\vx_1,  \vy_1) + f_2(\vx_1),\qquad 
            \bar{f}(\vx_1, \vy_1) = ~\bar{f}_1(\vx_1, \vy_1) + \bar{f}_2(\vy_1) ,
        \end{align}
where  $f_2$ and $\bar{f}_{2}$ are proper closed convex functions, $f_1$ is smooth, and $\bar{f}_1$ is smooth and convex in $\vy_1$. Assuming 
Slater's condition holds, strong duality holds for the lower-level problem and any of its optimal Lagrangian multiplier must fall in a box $[0,B]^l$ for some constant $B$, we can reformulate \eqref{eq:blo_cons} as a special case of \eqref{eq:bilevel} as follows
\begin{gather}\label{eq:cons_minmax}
    \min_{\vx_1, \vy_1, \vy_2}\left\{f(\vx_1, \vy_1)\bigg|
   (\vy_1, \vy_2) \in \arg \min_{\vz_1} \max_{\vz_2 } \tilde{f}(\vx_1, \vz_1, \vz_2)\right\},
\end{gather}
where there is no variable $\vx_2$ in the upper-level problem and 
\begin{align}
\label{eq:tildef_cons}
    \tilde{f}(\vx_1, \vy_1, \vy_2) =  \bar{f}(\vx_1, \vy_1) + \vy_2^\top\bar{\vg}(\vx_1, \vy_1)-\mathbf{1}_{[0,B]^l}(\vy_2)
\end{align}
with 
$\mathbf{1}_{[0,B]^l}(\cdot)$ being the indicator function of the set $[0,B]^l$. It follows that $\tilde{f}$ in \eqref{eq:tildef_cons} satisfies \eqref{eq:f_structure} with
$\tilde{f}_1(\vx_1, \vy_1, \vy_2) = \vy_2^\top\bar{\vg}(\vx_1, \vy_1) + \bar{f}_1(\vx_1, \vy_1) $, $\tilde{f}_2(\vy_1) = \bar{f}_2( \vy_1)$, and $\tilde{f}_3 (\vy_2) = \mathbf{1}_{[0,B]^l}(\vy_2)$.
We are then able to apply the proposed first-order method to \eqref{eq:blo_cons} and show that our method finds an $\epsilon$-KKT point of \eqref{eq:blo_cons} (see Definition~\ref{def:cons_kkt}) with a lower oracle complexity than \citet{lu2025solving, lu2024firstorder} although their methods can be 
applied to a more general setting, including the case where the strong duality does not hold in the lower level problem.


\textbf{Summary of contributions.}
\begin{enumerate}
    \item We propose a deterministic first-order method for~\eqref{eq:bilevel} based on a penalty reformulation, and show that it finds an $\epsilon$-KKT point of~\eqref{eq:bilevel} with $\tilde{O}(\epsilon^{-4})$ oracle complexity. We further extend the proposed method to the stochastic setting and show that our method finds a nearly $\epsilon$-KKT point of~\eqref{eq:bilevel} with $\tilde{O}(\epsilon^{-9})$ stochastic oracle complexity. To the best of our knowledge, these are the first complexity results for deterministic and stochastic first-order methods for~\eqref{eq:bilevel}.

    \item We apply the proposed method to~\eqref{eq:blo_cons} after reformulating it as \eqref{eq:cons_minmax}, which is an instance of~\eqref{eq:bilevel}. 
    This reformulation holds due to strong duality and a uniform bound on the dual variables guaranteed by the assumed Slater's condition. The resulting method finds an $\epsilon$-KKT point of~\eqref{eq:blo_cons} with complexity $\tilde{O}(\epsilon^{-4})$, improving upon the best existing complexity bound of $\tilde{O}(\epsilon^{-7})$ for a merely convex case lower-level problem~\citep{lu2025solving} and $\tilde{O}(\epsilon^{-6})$ for a strongly convex lower-level problem~\citep{lu2024firstorder}. That said, we also note that the methods in~\citep{lu2024firstorder,lu2025solving} 
    utilize a primal penalty reformulation of the lower-level problem so they do not require strong duality in the lower level and can be applied in a more general setting than ours, e.g., when Slater's condition does not hold.

    \item We demonstrate the proposed algorithms on both synthetic and real-world datasets. On~\eqref{eq:blo_cons}, our method outperforms existing penalty-based methods in terms of convergence speed. We further consider a bilevel formulation for hyperparameter tuning~\eqref{eq:dro_nested} of a real-world distributionally robust optimization problem~\citep{sagawa2020distributionally}, which gives rise to an instance of~\eqref{eq:bilevel} that substantially improves over an existing DRO method in terms of test worst-group accuracy.
\end{enumerate}


%% file: 2related.tex
\section{Related works}
\label{sec:related}

Most 
first-order methods for bilevel optimization focus
on the smooth setting in which the lower-level problem is a strongly convex constrained or unconstrained minimization problem. 
Under this assumption (and constraint qualification in the constrained case), the lower-level solution is a single-valued and differentiable mapping of the upper-level solution, and the gradient of the outer objective can be computed through the chain rule, the implicit function theorem and Danskin's theorem. This leads to the class of hypergradient or implicit-gradient methods 
\citep{pedregosa2016hyperparameter,
ghadimi2018approximation,rajeswaran2019meta,grazzi2020iteration,ji2021bilevel,
yang2021provably,guo2021randomized,chen2021closing,tsaknakis2022implicit,huang2022enhanced,chen2022fast,chen2022single,li2022fully,khanduri2021near,xiao2022alternating,liu2022general,xu2023efficient,hu2023improved,hong2023two,kwon2023fully,kwon2024penalty,JMLR:v26:23-1104,jiang2024primal}.
Their assumptions exclude applications in which the lower-level problem
is merely convex and has multiple solutions. 

Several recent works have replaced the lower-level strong convexity
assumption to other assumptions, such as the Polyak-{\L}ojasiewicz condition, or consider an optimization framework without using a hypergradient, for example, the penalty-based, value-function-based, and optimality-condition-based  a~\citep{liu2021value,liu2023value,liu2021towards,arbel2022non,liu2023averaged,chen2024finding,sow2022primal,liu2022bome,kwon2024penalty,xiao2023alternating,huang2023momentum,huang2024optimal,huang2023adaptive,jiang2025beyond}. \citet{shen2023penalty,shen2025penalty} consider a generic penalty method when the lower-level optimal set allows a squared-distance bound function. These methods can be applied to bilevel optimization with merely convex or non-convex lower-level objectives, 
but not directly to~\eqref{eq:bilevel} whose lower-level problem
has a minimax structure. A more detailed comparison between our work and some penalty-based methods is given in Table~\ref{tab:penalty_comparison} in Appendix~\ref{sec:penalty_comparison}.


More recently, \citet{lu2024firstorder} proposed a first-order penalty framework that
solves bilevel optimization through a minimax reformulation using the algorithmic framework from~\citet{lu2023first}. Their method finds an $\epsilon$-KKT point of \eqref{eq:blo_cons}
with complexity $\tilde{O}(\epsilon^{-4})$ and $\tilde{O}(\epsilon^{-7})$ when the lower-level problem is unconstrained and constrained, respectively. 
By a similar penalty approach, \citet{lu2025solving} developed a sequential minimax optimization
method that reduces the complexity to $\tilde{O}(\epsilon^{-6})$ when the lower-level problem is constrained but has a strongly convex objective function. Our work is motivated by \citet{lu2024firstorder,lu2023first,lu2025solving}. In fact, our Algorithm~\ref{alg:first-order-new} in the deterministic case is exactly Algorithm 2 from \cite{lu2023first}. However, we study a  more general
bilevel minimax structure in~\eqref{eq:bilevel}, which covers lower-level constrained optimization as a special case if Slater's condition holds. In that case, we  improve the complexity in \citep{lu2024firstorder,lu2025solving} from $\tilde{O}(\epsilon^{-7})$ and $\tilde{O}(\epsilon^{-6})$ to $\tilde{O}(\epsilon^{-4})$. That said, because the methods in \citep{lu2024firstorder,lu2025solving} apply a penalty formulation to handle the lower-level constraints instead of using duality gap, they 
can be applied to a more general setting than ours, e.g., when the lower-level problem does not satisfy Slater's condition or does not even have strong duality.





\citet{ouattara2018duality,jiang2024primal,nie2025augmented} study bilevel problems in which the lower-level problem involves nonlinear constraints, and reformulate the bilevel problem as a single-level problem via Lagrangian duality in a way similar to our approach. 
\citet{ouattara2018duality} do not provide a complexity analysis, however. 
\citet{jiang2024primal} and \citet{nie2025augmented} consider a penalty formulation of the single-level problem, but require strong convexity of the lower-level problem to ensure uniqueness of the lower-level solution; this assumption is not needed in our work. 
\citet{yaoconstrained} propose a related reformulation based on a proximal Lagrangian value function and, similarly to our approach, impose a box constraint on the dual variable. 
However, their algorithm requires projection onto the constraint set, which can be computationally inefficient. 

\citet{yaoovercoming} study bilevel optimization with
convex, nonlinearly constrained lower-level minimization.
They introduce a doubly regularized Lagrangian gap function
that yields an equivalent single-level formulation with a
smooth constraint and develop a penalty-based first-order
method. In the same paper, they further
extend their gap-function reformulation to bilevel optimization
with a convex-concave minimax lower-level problem and propose
a corresponding first-order algorithm. However, the convergence
analysis of this extension is left for future work. 


\citet{zheng2026saddle} study bilevel reinforcement learning
with a regularized zero-sum Markov game at the lower level
and establish convergence guarantees for a penalty-based
first-order policy-gradient method. Although their analysis
does not require strong convexity-concavity of the full
lower-level objective, it exploits state-wise strongly
convex-concave regularization, uniqueness of the lower-level
equilibrium policy pair, and a Polyak-Lojasiewicz condition
on the penalized objective. In contrast, our framework
allows general convex-concave lower-level problems with
possibly nonunique saddle points, without requiring
these additional structural properties.

Finally, \citet{khanduri2023linearly,kornowski2024first,tsaknakis2022implicit} study bilevel problems in which the lower-level constraints are linear, whereas we allow general nonlinear convex constraints. 
\citet{li2023novel,allende2013solving} consider reformulations of~\eqref{eq:blo_cons} based on Wolfe duality and the KKT conditions of the lower-level problem, respectively, which can accommodate general lower-level constraints. 
\citet{ye2023difference} and \citet{gao2022value} adopt value-function-based
reformulations and solve~\eqref{eq:blo_cons} using techniques from
difference-of-convex optimization.
However, these works do not provide complexity guarantees for finding an $\epsilon$-KKT point. 
Recent works by \citet{khalafi2025regularized,dvurechensky2026stochastic} study stochastic bilevel variational inequality problems. 
Although minimax problems can be formulated as variational inequality problems, their models use the same decision variables at the upper and lower levels, which is substantially different from our setting. 

%% file: 3methods.tex
\section{Preliminaries and assumptions}
\label{sec:preliminaries}

A function or mapping $\phi$ on $\mathbb{R}^n$ is \emph{$L_{\phi}$-Lipschitz} if $\|\phi(\vx)-\phi(\vx')\| \leq L_{\phi} \|\vx-\vx'\|$ for any $\vx$ and $\vx'$ in $\mathbb{R}^n$, and \emph{$L_{\nabla\phi}$-smooth} if it is differentiable and $\|\nabla\phi(\vx)-\nabla\phi(\vx')\| \leq L_{\nabla\phi} \|\vx-\vx'\|$ for any $\vx$ and $\vx'$ in $\mathbb{R}^n$. For a proper closed convex function $\phi$ on $\mathbb{R}^n$, its  \emph{proximal mapping} 
$\prox_{\phi}(\cdot)$, is defined as 
\[
\prox_{\phi}(\vx) := \argmin\nolimits_{\vx'\in\mathbb{R}^n} \left\{ \frac{1}{2}\|\vx' - \vx\|^2 + \phi(\vx') \right\}.
\]
For a proper closed function $\phi$ on $\mathbb{R}^n$, its \emph{subdifferential} is defined as
\[
\partial \phi(\vx)
:=
\left\{
v\in \mathbb{R}^n:
\phi(\vx')
\ge
\phi(\vx)+\langle v,\vx'-\vx\rangle
+o(\|\vx'-\vx\|),\quad\text{as } \vx'\to \vx
\right\}.
\]
Any vector $v\in \partial \phi(\vx)$ is called a \emph{subgradient} of $\phi$ at $\vx$. If $\vz$ is a sub-vector of $\vx$,  we use $\partial_{\vz} \phi(\vx)$ to denote the (partial) subdifferential of $\phi$ with respect to $\vz$. 

We define the primal and dual objective functions of the lower-level problem in \eqref{eq:bilevel} as
\begin{align}
\label{eq:prima_dual_obj}
p(\vx_1,\vy_1):=\max_{\vz_2} \tilde{f}(\vx_1, \vy_1, \vz_2)
\quad\text{and}\quad 
d(\vx_1,\vy_2):=\min_{\vz_1} \tilde{f}(\vx_1, \vz_1, \vy_2),
\end{align}
respectively, which can be non-smooth. By weak duality, it holds that $p(\vx_1,\vy_1)\geq d(\vx_1,\vy_2)$ for any $\vx_1$, $\vy_1$ and $\vy_2$, and $(\vy_1, \vy_2)$ is the saddle-point of $ \tilde{f}(\vx_1, \cdot, \cdot)$ if and only if $p(\vx_1,\vy_1)= d(\vx_1,\vy_2)$. Therefore, bilevel optimization \eqref{eq:bilevel} is essentially a non-convex concave minimax problem with a non-convex non-smooth inequality constraint, i.e., 
\begin{align}\label{eq:bilevel_equality}
    \min_{\vx_1,\vy_1,\vy_2}\left\{\max_{\vx_2} f(\vx_1, \vx_2, \vy_1, \vy_2)\Big|p(\vx_1,\vy_1)-d(\vx_1,\vy_2)\leq0\right\}.
\end{align}

We make the following assumptions throughout the paper. 
\begin{assumption}\label{assump:dom}Problem \eqref{eq:bilevel} has the structures in \eqref{eq:f_structure} and \eqref{eq:tildef_structure} and the following statements hold.
    \begin{enumerate}
    \item $\mathcal{X}_1:=\dom f_2$, $\mathcal{X}_2:=\dom f_3$, $\mathcal{Y}_1:=\dom\tilde{f}_2$, and $\mathcal{Y}_2:=\dom\tilde{f}_3$ are compact.
        \item $f_1(\vx_1, \vx_2, \vy_1, \vy_2)$ is concave in $\vx_2$ and $L_{\nabla f_1}$-smooth on $\mathcal{X}_1 \times \mathcal{X}_2 \times \mathcal{Y}_1 \times \mathcal{Y}_2$. $\tilde{f}_1(\vx_1, \vy_1, \vy_2)$ is convex in $\vy_1$,  concave in $\vy_2$, and $L_{\nabla \tilde{f}_1}$-smooth on $\mathcal{X}_1  \times \mathcal{Y}_1 \times \mathcal{Y}_2$. 
        \item $f_2$, $f_3$, $\tilde{f}_{2}$, and $\tilde{f}_{3}$ are proper closed convex, and their proximal operators can be computed exactly.
    \end{enumerate}
\end{assumption}

We define the following quantities, which are all finite according to Assumption~\ref{assump:dom}.
\begin{align}
 \label{eq:D1}
 D_1 := & \max\{\|(\vx_1,\vy_1,\vy_2) - (\vx_1',\vy_1',\vy_2') \| \;|\; (\vx_1,\vy_1,\vy_2), (\vx_1',\vy_1',\vy_2') \in \mathcal{X}_1  \times \mathcal{Y}_1 \times \mathcal{Y}_2\},\\   \label{eq:D2}
 D_2 := &\max\{\|(\vx_2,\vz_1,\vz_2) - (\vx_2',\vz_1',\vz_2')\| \;|\; (\vx_2,\vz_1,\vz_2),(\vx_2',\vz_1',\vz_2') \in \mathcal{X}_2  \times \mathcal{Y}_1 \times \mathcal{Y}_2\},\\  
     f_{\text{low}}:=& \min \{f(\vx_1, \vx_2, \vy_1, \vy_2) \;|\; (\vx_1, \vx_2, \vy_1, \vy_2)\in \mathcal{X}_1 \times \mathcal{X}_2 \times \mathcal{Y}_1 \times \mathcal{Y}_2 \}\label{eq:flow} . 
\end{align}

We call $(\vx_1,\vy_1,\vy_2)$ a \emph{Karush–Kuhn–Tucker} (KKT) point of \eqref{eq:bilevel} if there exists $\rho\geq 0$ such that
\begin{align*}
&\mathbf{0}\in\partial_{(\vx_1,\vy_1,\vy_2)} \left[\max_{\vx_2} f(\vx_1, \vx_2, \vy_1, \vy_2)\right] +\rho\partial_{(\vx_1,\vy_1,\vy_2)}\left[p(\vx_1,\vy_1)-d(\vx_1,\vy_2)\right],\\
&p(\vx_1,\vy_1)-d(\vx_1,\vy_2)\leq0.
\end{align*}
Given that the maximization problems in $\max_{\vx_2} f(\vx_1, \vx_2, \vy_1, \vy_2)$ and in the definition of $p(\vx_1,\vy_1)-d(\vx_1,\vy_2)$ are concave, by Danskin's theorem~\citep{bertsekas1997nonlinear}, $(\vx_1,\vy_1,\vy_2)$ as a KKT point of \eqref{eq:bilevel} if there exist $\rho\geq 0$ and $(\vx_2,\vz_1,\vz_2)$ such that 
\begin{align*}
&\mathbf{0}\in\partial_{(\vx_1,\vy_1,\vy_2)} \left[ f(\vx_1, \vx_2, \vy_1, \vy_2) +\rho\big(\tilde{f}(\vx_1, \vy_1, \vz_2)-\tilde{f}(\vx_1, \vz_1, \vy_2)\big)\right],\\
&\mathbf{0}\in\partial_{\vx_2}\left[-f(\vx_1, \vx_2, \vy_1, \vy_2)\right],\quad \mathbf{0}\in\rho\partial_{\vz_1}\tilde{f}(\vx_1, \vz_1, \vy_2),\quad \mathbf{0}\in\rho\partial_{\vz_2}\left[-\tilde{f}(\vx_1, \vy_1, \vz_2)\right],  \\
&p(\vx_1,\vy_1)-d(\vx_1,\vy_2)\leq0.
\end{align*}
In general, a KKT point can only be approached as a limiting solution of an algorithm. Hence, we consider a lifted $\epsilon$-KKT point of \eqref{eq:bilevel} which satisfies all conditions of a KKT point up to an error of $\epsilon$. 
\begin{definition}
\label{def:epsilon_KKT}
$(\vx_1, \vy_1, \vy_2)$ is a lifted $\epsilon$-KKT point \eqref{eq:bilevel} if there exist $\rho\geq0$ and $(\vx_2, \vz_1, \vz_2)$ such that
\begin{align*}
&\text{dist}\left(\mathbf{0},\partial_{(\vx_1,\vy_1,\vy_2)} \left[ f(\vx_1, \vx_2, \vy_1, \vy_2)+\rho\big(\tilde{f}(\vx_1, \vy_1, \vz_2)-\tilde{f}(\vx_1, \vz_1, \vy_2)\big)\right]\right)\leq\epsilon,\\
&\text{dist}\left(\mathbf{0},\partial_{\vx_2}\left[-f(\vx_1, \vx_2, \vy_1, \vy_2)\right]\right)\leq\epsilon,\quad \text{dist}\left(\mathbf{0},\rho\partial_{\vz_1}\tilde{f}(\vx_1, \vz_1, \vy_2)\right)\leq\epsilon,  \\
&\text{dist}\left(\mathbf{0},\rho\partial_{\vz_2}\left[-\tilde{f}(\vx_1, \vy_1, \vz_2)\right]\right)\leq\epsilon,\quad p(\vx_1,\vy_1)-d(\vx_1,\vy_2)\leq\epsilon.
\end{align*}
\end{definition}
We propose a first-order method for finding a lifted $\epsilon$-KKT point of~\eqref{eq:bilevel}. The efficiency of our  method is characterized by its \emph{oracle complexity}, i.e., the total number of gradient evaluations of $f$ and $\tilde f$.

\section{A deterministic first-order method for \eqref{eq:bilevel} based on penalty formulation}
\label{sec:penalty}
Following the existing literature, e.g.,~\citep{shen2023penalty,kwon2023fully,JMLR:v26:23-1104,lu2025solving, lu2024firstorder}, we reformulate \eqref{eq:bilevel} and \eqref{eq:bilevel_equality} into an unconstrained penalty problem using a penalty parameter $\rho\geq0$:
\begin{align}
\label{eq:old_minmax}
     f^*:=&\min_{\vx_1, \vy_1, \vy_2} \left\{\max_{\vx_2}f(\vx_1, \vx_2, \vy_1, \vy_2) + \rho\left[p(\vx_1,\vy_1)-d(\vx_1,\vy_2)\right]\right\}\\ \label{eq:new_minmax}
     =&\min_{\vx_1, \vy_1, \vy_2}\max_{\vx_2, \vz_1, \vz_2}P_\rho(\vx_1,\vx_2, \vy_1,\vy_2, \vz_1, \vz_2)
\end{align}
where
\begin{align}
\label{eq:lagrange}
    P_\rho(\vx_1,\vx_2, \vy_1,\vy_2, \vz_1, \vz_2) :=&~ f(\vx_1, \vx_2, \vy_1, \vy_2) + \rho\big(\tilde{f}(\vx_1, \vy_1, \vz_2)-\tilde{f}(\vx_1 , \vz_1, \vy_2)\big)\\\label{eq:lagrange_split}
=&~P_1(\vx_1,\vx_2, \vy_1,\vy_2, \vz_1, \vz_2)+P_2(\vx_1,\vy_1,\vy_2)-P_3(\vx_2,\vz_1,\vz_2)\\\label{eq:P1}
P_1(\vx_1,\vx_2, \vy_1,\vy_2, \vz_1, \vz_2):=&~f_1(\vx_1, \vx_2, \vy_1, \vy_2) + \rho\big(\tilde{f}_1(\vx_1, \vy_1, \vz_2)-\tilde{f}_1(\vx_1 , \vz_1, \vy_2)\big),\\\label{eq:P2}
P_2(\vx_1,\vy_1,\vy_2):=&~f_2(\vx_1)+\rho\tilde{f}_2(\vy_1)+\rho\tilde{f}_3(\vy_2)\\\label{eq:P3}
P_3(\vx_2 , \vz_1, \vz_2):=&~f_3(\vx_2)+\rho\tilde{f}_2(\vz_1)+\rho\tilde{f}_3(\vz_2).
\end{align}
The strategy we adopt is to first find an $\epsilon$-primal-dual stationary point of \eqref{eq:new_minmax} defined below using a primal-dual method and show that such a point is also a lifted $O(\epsilon)$-KKT point of \eqref{eq:bilevel}.
\begin{definition}
\label{def:epsilon_stationary}
$(\vx_1, \vy_1, \vy_2,\vx_2, \vz_1, \vz_2)$   is an $\epsilon$-primal-dual stationary point of \eqref{eq:old_minmax} if
\begin{align*}
&\text{dist}\left(\mathbf{0},\partial_{(\vx_1,\vy_1,\vy_2)} \left[ f(\vx_1, \vx_2, \vy_1, \vy_2)+\rho\big(\tilde{f}(\vx_1, \vy_1, \vz_2)-\tilde{f}(\vx_1, \vz_1, \vy_2)\big)\right]\right)\leq\epsilon,\\
&\text{dist}\left(\mathbf{0},\partial_{(\vx_2,\vz_1,\vz_2)} \left[-f(\vx_1, \vx_2, \vy_1, \vy_2) +\rho\big(-\tilde{f}(\vx_1, \vy_1, \vz_2)+\tilde{f}(\vx_1, \vz_1, \vy_2)\big)\right]\right)\leq\epsilon.
\end{align*}
\end{definition}

The lemma below characterizes the relationship between Definition~\ref{def:epsilon_KKT}  and Definition~\ref{def:epsilon_stationary}. The proof can be found in Appendix~\ref{sec:lemma_connection}.

\begin{lemma}
\label{thm:connection}
Suppose $(\vx_1, \vy_1, \vy_2,\vx_2, \vz_1, \vz_2)$   is an $\epsilon$-primal-dual stationary point of \eqref{eq:old_minmax} and 
\begin{align}
\label{eq:feasibility_bound_condition}
\rho^{-1}\left(\max_{\vx_2', \vz_1', \vz_2'}P_{\rho}(\vx_1, \vx_2', \vy_1, \vy_2, \vz_1', \vz_2')- f_{\text{low}}\right)\leq O(\epsilon),
\end{align}
where $P_{\rho}$ and $f_{\text{low}}$ are defined in \eqref{eq:lagrange} and \eqref{eq:flow}.
Then $(\vx_1, \vy_1, \vy_2)$ is a lifted $O(\epsilon)$-KKT  point of~\eqref{eq:bilevel}.
\end{lemma}

Based on Lemma~\ref{thm:connection}, our approach will find a lifted $O(\epsilon)$-KKT point of \eqref{eq:bilevel} by computing an $\epsilon$-primal-dual stationary point $(\vx_1, \vy_1, \vy_2,\vx_2, \vz_1, \vz_2)$ of \eqref{eq:new_minmax} and, simultaneously, ensure it satisfies \eqref{eq:feasibility_bound_condition} by showing that $\max_{\vx_2', \vz_1', \vz_2'}P_{\rho}(\vx_1, \vx_2', \vy_1, \vy_2, \vz_1', \vz_2')=O(1)$ and  setting $\rho=O(\epsilon^{-1})$.

Under Assumption~\ref{assump:dom}, \eqref{eq:new_minmax} is a non-convex-concave (NCC) minimax problem. We directly apply the \emph{inexact proximal point method} by \citet[Algorithm 2]{lu2023first} to find its $\epsilon$-primal-dual stationary point. In this method, the solutions are updated by solving a sequence of minimax subproblems:
\small
\begin{gather}
\nonumber
\left(
(\vx_1^{(k+1)}, \vy_1^{(k+1)}, \vy_2^{(k+1)}),(\vx_2^{(k+1)}, \vz_1^{(k+1)}, \vz_2^{(k+1)})
\right)
\approx \arg\min_{\vx_1, \vy_1, \vy_2} \max_{\vx_2, \vz_1, \vz_2} \bigg\{ P_\rho(\vx_1,\vx_2, \vy_1,\vy_2, \vz_1, \vz_2) \\ \label{eq:minmax_prox}
+\frac{\rho_1}{2}\left\|(\vx_1, \vy_1, \vy_2)-(\vx_1^{(k)}, \vy_1^{(k)}, \vy_2^{(k)})\right\|^2
-\frac{\rho_2}{2}\left\|(\vx_2, \vz_1, \vz_2)-(\vx_2^{(0)}, \vz_1^{(0)}, \vz_2^{(0)})\right\|^2
\bigg\},
\end{gather}
\normalsize
where $\rho_1\geq0$ is a proximal parameter and $\rho_2\geq0$ is a smoothing parameter. If $\rho_1$ and $\rho_2$ are set appropriately, subproblem \eqref{eq:minmax_prox} is a strongly-convex-strongly-concave (SCSC) minimax problem, which can be solved by the optimal method from
\citep[Algorithm 4]{kovalev2022first} and \citep[Algorithm 1]{lu2023first}. We present this optimal method in Algorithm~\ref{mmax-alg1} in Section~\ref{sec:wcc_minmax} and denote it by \texttt{OptFOM}. Using Algorithm~\ref{mmax-alg1} as a subroutine for \eqref{eq:minmax_prox} leads to \citep[Algorithm 2]{lu2023first}, which is presented in the deterministic case in Algorithm~\ref{alg:first-order-new}. Its oracle complexity is presented in Theorem~\ref{thm:minmax} with the proof given in Appendix~\ref{sec:proof_thm1}.

\begin{algorithm}[t]
\caption{A (deterministic or stochastic) inexact proximal point method for~\eqref{eq:bilevel}}
\label{alg:first-order-new}
\begin{algorithmic}[1]
\STATE {\bfseries Input:} $K>0$, $\epsilon>0$, $(\vx_1^{(0)},\vy_1^{(0)},\vy_2^{(0)})\in\mathcal{X}_1 \times \mathcal{Y}_1 \times \mathcal{Y}_2$, $(\vx_2^{(0)},\vz_1^{(0)},\vz_2^{(0)})\in\mathcal{X}_2 \times \mathcal{Y}_1 \times \mathcal{Y}_2$
\STATE Set $\rho= 1/\epsilon,\quad \rho_1= ~2L_{\nabla f}+4\rho L_{\nabla \tilde{f}},\quad  \rho_2=  ~\epsilon/(2D_2)$ in \eqref{eq:minmax_prox}.
\STATE Set $\hat\epsilon=\epsilon^{1.5},\quad L_{\nabla P_1}=L_{\nabla f} + 2\rho L_{\nabla \tilde{f}},\quad\delta_P^2=3\delta_{f}^2 + 6\rho^2\delta_{\tilde{f}}^2$
\FOR{$k=0,1,2,\ldots$}
\STATE Set
\small
\begin{equation*}
\begin{aligned}
\h(x,y) \leftarrow&~P_1(\vx_1,\vx_2, \vy_1,\vy_2, \vz_1, \vz_2)\\
&~+\frac{\rho_1}{2}\left\|(\vx_1, \vy_1, \vy_2)-(\vx_1^{(k)}, \vy_1^{(k)}, \vy_2^{(k)})\right\|^2-\frac{\rho_2}{2}\left\|(\vx_2, \vz_1, \vz_2)-(\vx_2^{(0)}, \vz_1^{(0)}, \vz_2^{(0)})\right\|^2,\\
p(x)\leftarrow&~P_2(\vx_1,\vy_1,\vy_2),\quad q(y)\leftarrow~P_3(\vx_2,\vz_1, \vz_2),\\
    L_{\nabla\h} \leftarrow&~L_{\nabla P_1}+\max\{\rho_1,\rho_2\},\quad
    \sigma_x\leftarrow~L_{\nabla P_1},\quad
    \sigma_y\leftarrow~\epsilon/(2D_2),\quad \delta^2\leftarrow~\delta_P^2
\end{aligned}
\end{equation*}
\normalsize
\IF {\eqref{eq:bilevel} is deterministic (i.e., $\delta_f=\delta_{\tilde{f}}=0$)} 
\STATE  Apply \texttt{OptFOM} in Algorithm~\ref{mmax-alg1} as follows.
\small
\begin{align*}
&\left((\vx_1^{(k+1)}, \vy_1^{(k+1)}, \vy_2^{(k+1)}),(\vx_2^{(k+1)}, \vz_1^{(k+1)}, \vz_2^{(k+1)})\right)\\
&\leftarrow\texttt{OptFOM}\left(\epsilon_k:=\frac{\hat{\epsilon}}{k+1},(\vx_1^{(k)}, \vy_1^{(k)}, \vy_2^{(k)}),(\vx_2^{(k)}, \vz_1^{(k)}, \vz_2^{(k)}), \bar{h}, p,q,L_{\nabla \h},\sigma_x,\sigma_y\right)
\end{align*}
\normalsize
\IF{
$\left\|
(\vx_1^{(k+1)}, \vy_1^{(k+1)}, \vy_2^{(k+1)})
-
(\vx_1^{(k)}, \vy_1^{(k)}, \vy_2^{(k)})
\right\|
\leq
\epsilon/(4L_{\nabla P_1})$
}
\STATE Return \small$\left((\vx_1^{\epsilon}, \vy_1^{\epsilon}, \vy_2^{\epsilon}),(\vx_2^{\epsilon}, \vz_1^{\epsilon}, \vz_2^{\epsilon})\right)=\left((\vx_1^{(k+1)}, \vy_1^{(k+1)}, \vy_2^{(k+1)}),(\vx_2^{(k+1)}, \vz_1^{(k+1)}, \vz_2^{(k+1)})\right)$ \normalsize
\ENDIF
\ELSE
\STATE  Apply \texttt{SAPD} in Algorithm~\ref{alg:sapd} as follows.
\small
\begin{align*}
&\left((\vx_1^{(k+1)}, \vy_1^{(k+1)}, \vy_2^{(k+1)}),(\vx_2^{(k+1)}, \vz_1^{(k+1)}, \vz_2^{(k+1)})\right)\\
&\leftarrow\texttt{SAPD}\left(\hat{\epsilon},(\vx_1^{(k)}, \vy_1^{(k)}, \vy_2^{(k)}),(\vx_2^{(k)}, \vz_1^{(k)}, \vz_2^{(k)}), \bar{h}, p,q,L_{\nabla \h},\sigma_x,\sigma_y,\delta^2\right)
\end{align*}
\normalsize
\vspace{-0.3cm}
\IF{$k=K-1$}
\STATE Return \small$\left((\vx_1^{\epsilon}, \vy_1^{\epsilon}, \vy_2^{\epsilon}),(\vx_2^{\epsilon}, \vz_1^{\epsilon}, \vz_2^{\epsilon})\right)=\left((\vx_1^{(k')}, \vy_1^{(k')}, \vy_2^{(k')}),(\vx_2^{(k')}, \vz_1^{(k')}, \vz_2^{(k')})\right)$, 
\normalsize
where $k'$ is uniformly randomly sampled from $\{1,2,\dots,K\}$.
\ENDIF
\ENDIF
\ENDFOR
\end{algorithmic}
\end{algorithm}

\begin{theorem}\label{thm:minmax}
Suppose Assumption~\ref{assump:dom} holds and $(\vx_1^{(0)},\vy_1^{(0)}, \vy_2^{(0)})$ in Algorithm~\ref{alg:first-order-new} satisfies
\begin{align}
\label{eq:initial_condition}
p(\vx_1^{(0)}, \vy_1^{(0)})-d(\vx_1^{(0)} ,\vy_2^{(0)})\leq \epsilon.
\end{align}
   Algorithm~\ref{alg:first-order-new} outputs a lifted $O(\epsilon)$-KKT point $(\vx_1^{\epsilon}, \vy_1^{\epsilon}, \vy_2^{\epsilon})$ of~\eqref{eq:bilevel} with oracle complexity $\tilde O(\epsilon^{-4})$.
    \end{theorem}
\begin{remark}\label{remark:initial_condition}
To satisfy \eqref{eq:initial_condition}, one can simply choose any $\vx_1^{(0)}\in\mathcal{X}_1$ and then apply any optimal first-order method, e.g., the one by \citep{doi:10.1137/S1052623403425629}, to the convex-concave minimax problem
$\min_{\vy_1}\max_{\vy_2}\tilde{f}(\vx_1^{(0)}, \vy_1, \vy_2)$ to find an $\epsilon$-saddle point $(\vy_1^{(0)} ,\vy_2^{(0)})$ that satisfies \eqref{eq:initial_condition} in oracle complexity $O(1/\epsilon)$. This initialization procedure does not increase the complexity of Algorithm~\ref{alg:first-order-new}.
\end{remark}
Under additional strong convexity and strong concavity assumptions on the lower-level problem, the oracle complexity established in Theorem~\ref{thm:minmax} can be improved by appropriately adjusting the parameters in Algorithm~\ref{alg:first-order-new}. Specifically, under Assumptions~\ref{assump:dom} and~\ref{assump:dom_strong}, the oracle complexity improves to $\tilde{O}(\epsilon^{-3.5})$. Under Assumptions~\ref{assump:dom} and~\ref{assump:dom_strong_nox2}, it further improves to $\tilde{O}(\epsilon^{-3})$. These improved complexity bounds are established in Theorems~\ref{thm:minmax_strong} and~\ref{thm:minmax_strong_nox2}, respectively, in the appendix.

\section{Application on bilevel optimization with a constrained lower-level problem}
\label{sec:cons_minmax}

In this section, we discuss the application of  Algorithm~\ref{alg:first-order-new} to~\eqref{eq:blo_cons}. We first present assumptions on~\eqref{eq:blo_cons}. 
\begin{assumption}
\label{assume:constrained}
Problem \eqref{eq:blo_cons} has the structures in \eqref{eq:f_structure_cons} and the following statements hold.
    \begin{enumerate}
    \item $\mathcal{X}_1:=\dom f_2$ and $\mathcal{Y}_1:=\dom\bar{f}_2$ are compact.
        \item $f_1(\vx_1, \vy_1)$ is $L_{\nabla f_1}$-smooth on $\mathcal{X}_1  \times \mathcal{Y}_1$. $\bar{f}_1(\vx_1, \vy_1)$ is convex in $\vy_1$ and $L_{\nabla \bar{f}_1}$-smooth on $\mathcal{X}_1  \times \mathcal{Y}_1$. $\bar{f}(\vx_1, \vy_1)$ is $L_{\bar{f}}$-Lipschitz on $\mathcal{X}_1  \times \mathcal{Y}_1$. $\bar{\vg}(\vx_1, \vy_1)$ is convex in $\vy_1$ in each component, $L_{\nabla \bar{\vg}}$-smooth and and $L_{\bar{\vg}}$-Lipschitz on on $\mathcal{X}_1  \times \mathcal{Y}_1$. 
        \item $f_2$ and $\bar{f}_{2}$ are proper closed convex and their proximal operators can be computed exactly. 
       \item (Slater's condition) There exists $G>0$ such that, for any  $\vx_1\in\mathcal{X}_1$, there exists $\hat{\vy}_1\in\mathcal{Y}_1$ such that   $\bar{g}_i(\vx_1, \hat{\vy}_1)\leq -G$ for $i=1,\dots,l$.
    \end{enumerate}
\end{assumption}

We define the following quantities, which are all finite according to Assumption~\ref{assume:constrained}.
\small
\begin{align}
D_{\mathcal{Y}_1} := &\max\{\|\vy_1 - \vy_1'\| \;|\; \vy_1, \vy_1' \in \mathcal{Y}_1 \},\quad B:= \frac{2L_{\bar{f}}D_{\mathcal{Y}_1}}{G} \label{eq:cons_y1}.
\end{align}
\normalsize

The value of $B$ above is used in \eqref{eq:cons_minmax}. Our goal is to find a lifted $\epsilon$-KKT point of \eqref{eq:blo_cons} defined by \citep{lu2025solving,lu2024firstorder}.

\begin{definition}
\label{def:cons_kkt}
$(\vx_1, \vy_1)$ is a lifted $\epsilon$-KKT point of \eqref{eq:blo_cons} if there exists $(\vz_1, \rho, \vlam, \bar{\vlam})\in\mathcal{Y}_1\times\mathbb{R}_+\times \mathbb{R}_+^{l}\times \mathbb{R}_+^{l}$ such that
\small
    \begin{align*}
        &\text{dist}\left(\mathbf{0},\partial_{(\vx_1,\vy_1)} \left[f(\vx_1, \vy_1) +\rho\big(\bar{f}(\vx_1, \vy_1)-\bar{f}(\vx_1, \vz_1) - \bar{\vlam}^T\bar{g}(\vx_1, \vz_1)\big)+\vlam^T\bar{g}(\vx_1, \vy_1)\right]\right)\leq\epsilon, \\  
&\text{dist}\left(\mathbf{0}, \rho\partial_{\vz_1}\left[\bar{f}(\vx_1, \vz_1) + \bar{\vlam}^T\bar{g}(\vx_1, \vz_1)\right]  \right)\leq\epsilon,\quad\|[\bar{g}(\vx_1, \vz_1)]_+\|\leq \epsilon,\quad\left|\bar{\vlam}^T\bar{g}(\vx_1, \vz_1)\right|\leq\epsilon, \\
&\bar{f}(\vx_1, \vy_1)-\bar{f}^*(\vx_1)\leq \epsilon,\quad \|[\bar{g}(\vx_1, \vy_1)]_+\|\leq \epsilon,\quad\left|\vlam^T\bar{g}(\vx_1, \vy_1)\right|\leq\epsilon,
    \end{align*}
\normalsize
where $\bar{f}^*(\vx_1):=\min_{\vy_1}  \left\{\bar{f}(\vx_1, \vy_1)|\bar{\vg}(\vx_1, \vy_1)\leq 0\right\}$ for any $\vx_1\in\mathcal{X}_1$.
\end{definition}


It is well known that, under Slater's condition for the lower-level problem,
\eqref{eq:blo_cons} can be reformulated as \eqref{eq:cons_minmax},
which is a special case of \eqref{eq:bilevel}. This reformulation holds because of the lemma below whose proof is standard and given in Appendix~\ref{sec:proof_lemma1}.
\begin{lemma}
\label{thm:con_lagrangian}
Suppose Assumption~\ref{assume:constrained} holds. 
For any fixed $\vx_1\in\mathcal{X}_1$, a 
$\vy_1\in\mathcal{Y}_1$ is an optimal solution to the lower-level problem of \eqref{eq:blo_cons} if and only if there exists $\vy_2\in [0,B]^l$ such that $(\vy_1,\vy_2)$ is a saddle-point of 
$
\min_{\vz_1}\max_{\vz_2}\tilde{f}(\vx_1, \vz_1, \vz_2),
$
where $\tilde{f}(\vx_1, \vz_1, \vz_2)$ is defined in \eqref{eq:tildef_cons}.
\end{lemma}


According to Lemma~\ref{thm:con_lagrangian}, we can solve \eqref{eq:blo_cons} by solving \eqref{eq:cons_minmax} using Algorithm~\ref{alg:first-order-new}. 
Below we show that a lifted $\epsilon$-KKT point of  \eqref{eq:cons_minmax} is also a lifted $O(\epsilon)$-KKT point of  \eqref{eq:blo_cons}. The proof is given in Appendix~\ref{sec:proof_of_cons_minmax}.


\begin{lemma}\label{thm:e-kkt}
Suppose $\epsilon \leq \min\{\frac{\rho L_{\bar{f}}}{4},\frac{\rho G }{4}\}$.     A lifted $\epsilon$-KKT point for \eqref{eq:cons_minmax} under Definition~\ref{def:epsilon_KKT} is a lifted $O(\epsilon)$-KKT point of \eqref{eq:blo_cons} under Definition~\ref{def:cons_kkt}.
\end{lemma}
With this lemma, we can characterize the complexity for Algorithm~\ref{alg:first-order-new} to find a lifted $O(\epsilon)$-KKT point of \eqref{eq:blo_cons} in the following theorem whose proof is given in Appendix~\ref{sec:proof_thmcons}.

\begin{theorem}\label{thm:constrained}
Suppose $\epsilon \leq \min\{\frac{\rho L_{\bar{f}}}{4},\frac{\rho G }{4}\}$ and  Algorithm~\ref{alg:first-order-new} is applied to \eqref{eq:cons_minmax}. 
Suppose Assumption~\ref{assume:constrained} holds and $(\vx_1^{(0)},\vy_1^{(0)}, \vy_2^{(0)})$ in Algorithm~\ref{alg:first-order-new}  satisfies
\small
\begin{align}
\label{eq:initial_condition_constrained}
\max_{\vz_2 \in [0, B]} \left\{\bar{f}(\vx_1^{(0)}, \vy_1^{(0)}) + \vz_2^\top\bar{\vg}(\vx_1^{(0)}, \vy_1^{(0)})\right\}
-\min_{\vz_1}\left\{\bar{f}(\vx_1^{(0)}, \vz_1) + (\vy_2^{(0)})^\top\bar{\vg}(\vx_1^{(0)}, \vz_1)\right\} 
\leq \epsilon
\end{align}
\normalsize
Algorithm~\ref{alg:first-order-new} outputs a lifted $O(\epsilon)$-KKT point $(\vx_1^{\epsilon}, \vy_1^{\epsilon})$ of~\eqref{eq:blo_cons} in oracle complexity $\tilde O(\epsilon^{-4})$.
    \end{theorem}
\begin{remark}
The methods of \citet{lu2024firstorder,lu2025solving} achieve an
oracle complexity of $\tilde{O}(\epsilon^{-7})$ for finding a lifted
$O(\epsilon)$-KKT point of~\eqref{eq:blo_cons}. This bound can be
improved to $\tilde{O}(\epsilon^{-6})$ under strong convexity of the
lower-level problem~\citep{lu2025solving}.
The relatively high oracle complexity of these methods stems in part
from the use of two penalty parameters,
$\rho=\Theta(\epsilon^{-1})$ and $\mu=\Theta(\epsilon^{-2})$,
associated with the outer- and lower-level problems, respectively.
The resulting minimax penalty reformulation has a smoothness constant
of order $\Theta(\rho\mu)=\Theta(\epsilon^{-3})$.
In contrast, our reformulation~\eqref{eq:cons_minmax} does not
require the additional penalty parameter $\mu$, yielding a
smoothness constant of order $\Theta(\rho)=\Theta(\epsilon^{-1})$.
Nevertheless, our approach relies on the equivalence established
in Lemma~\ref{thm:con_lagrangian}.
If this equivalence fails, for instance, due to the lack of
strong duality or uniformly bounded optimal dual solutions
over $\vx_1\in\mathcal{X}$, our method is not  applicable.
By contrast, the methods of \citet{lu2024firstorder,lu2025solving}
do not rely on the reformulation~\eqref{eq:cons_minmax}
and can therefore be stilll applied.
\end{remark}

\section{Stochastic bilevel optimization with a minimax lower level problem}
\label{sec:stochastic_bilevel}
In this section, we still consider problem~\eqref{eq:bilevel} when gradients $f_1$ and $\tilde f_1$ can only be accessed through stochastic oracles. In particular, in addition to Assumption \ref{assump:dom}, we make the following assumptions.
\begin{assumption}
\label{assump:stoc}
There exist a random variable $\omega$, two constants $\delta_f$ and $\delta_{\tilde{f}}$, and stochastic gradient estimators  of $\nabla f_1(\vx_1, \vx_2, \vy_1, \vy_2)$ and $\nabla \tilde{f}_1(\vx_1,\vy_1, \vy_2)$, which are mappings of $(\vx_1, \vx_2, \vy_1, \vy_2)$ and $\omega$ denoted by $\hat{\nabla} f_1(\vx_1, \vx_2, \vy_1, \vy_2;\omega)$ and $\hat{\nabla} \tilde{f}_1(\vx_1,\vy_1, \vy_2;\omega)$, respectively, and satisfy
\small
 \begin{align*}
    & \mathbb{E}[\hat{\nabla}f_1(\vx_1, \vx_2, \vy_1, \vy_2; \omega)] = \nabla f_1(\vx_1, \vx_2, \vy_1, \vy_2), \quad
     \mathbb{E}[\hat{\nabla}\tilde{f}_1(\vx_1, \vy_1, \vy_2; \omega)] = \nabla \tilde{f}_1(\vx_1, \vy_1, \vy_2), \\
    & \mathbb{E}[\|\hat{\nabla}f_1(\vx_1, \vx_2, \vy_1, \vy_2; \omega) - \nabla f_1(\vx_1, \vx_2, \vy_1, \vy_2)\|^2] \leq \delta_f^2, \\
    & \mathbb{E}[\|\hat{\nabla}\tilde{f}_1(\vx_1, \vy_1, \vy_2; \omega)-\nabla \tilde{f}_1(\vx_1, \vy_1, \vy_2)\|^2] \leq \delta_{\tilde{f}}^2.
\end{align*}
\normalsize
\end{assumption}

A key design that allows Algorithm~\ref{alg:first-order-new} to find a lifted $\epsilon$-KKT point of \eqref{eq:bilevel} in the deterministic case is the proximal steps (see lines 26 and 27) in Algorithm~\ref{mmax-alg1}, which require deterministic gradients and thus are not implementable in the stochastic setting. Therefore, we consider a slightly different notion of a lifted $\epsilon$-KKT point in the stochastic case. 

\begin{definition}
\label{def:epsilon_KKT_nearly}
A solution $(\vx_1, \vy_1, \vy_2)$ is a nearly lifted $\epsilon$-KKT point of \eqref{eq:bilevel} if there exists $(\hat{\vx}_1, \hat{\vy}_1, \hat{\vy}_2)$ such that
$\|(\vx_1, \vy_1, \vy_2)-(\hat{\vx}_1, \hat{\vy}_1, \hat{\vy}_2)\|\leq \epsilon$ and $(\hat{\vx}_1, \hat{\vy}_1, \hat{\vy}_2)$ is a lifted $\epsilon$-KKT point \eqref{eq:bilevel}. 
\end{definition}

Under Assumptions~\ref{assump:dom} and \ref{assump:stoc}, \eqref{eq:new_minmax} is a stochastic NCC minimax problem, which can be solved by a stochastic inexact proximal point method such as the SAPD+ method by \citep{zhang2022sapd+}. SAPD+ updates solutions also in the way of \eqref{eq:minmax_prox} except that the subproblem \eqref{eq:minmax_prox} is also solved by the stochastic accelerated primal-dual (SAPD) method~\citep{zhang2024robust}. The SAPD+ method  applied to \eqref{eq:new_minmax} is described in Algorithm~\ref{alg:first-order-new} when the problem is stochastic (i.e., the condition of the if-statement in Line 6 is false) and SAPD is presented in Algorithm~\ref{alg:sapd} in Section~\ref{sec:SAPD} for a generic SCSC minimax problem. The complexity of Algorithm~\ref{alg:first-order-new} in the stochastic case is given below with proof provided in Appendix~\ref{sec:proof_thm2}.

\begin{theorem}\label{thm:minmax_stoch}
Suppose Assumptions~\ref{assump:dom} and \ref{assump:stoc} hold, $(\vx_1^{(0)},\vy_1^{(0)}, \vy_2^{(0)})$ in Algorithm~\ref{alg:first-order-new} satisfies
\begin{align}
\label{eq:initial_condition_stoch}
\mathbb{E}[p(\vx_1^{(0)}, \vy_1^{(0)})-d(\vx_1^{(0)} ,\vy_2^{(0)})]\leq \epsilon,
\end{align}
and $K$ in Algorithm~\ref{alg:first-order-new}  satisfies 
$
K= \left\lceil \left(\mathbb{E}\left[\max_{\vx_2}f(\vx_1^{(0)}, \vx_2, \vy_1^{(0)}, \vy_2^{(0)})\right]+1- f^* + \epsilon D_2/4\right)\epsilon^{-3}\right\rceil.
$
Algorithm~\ref{alg:first-order-new} terminates and outputs a nearly lifted $O(\epsilon)$-KKT point  of~\eqref{eq:bilevel} in expectation in stochastic oracle complexity $\tilde{O}(\epsilon^{-9})$.
\end{theorem}
Under additional strong convexity and strong concavity assumptions on the lower-level problem, the oracle complexity established in Theorem~\ref{thm:minmax_stoch} can be improved by appropriately adjusting the parameters in Algorithm~\ref{alg:first-order-new}. Specifically, under Assumptions~\ref{assump:dom} and~\ref{assump:dom_strong}, the oracle complexity improves to $\tilde{O}(\epsilon^{-8})$. Under Assumptions~\ref{assump:dom} and~\ref{assump:dom_strong_nox2}, it further improves to $\tilde{O}(\epsilon^{-7})$. These improved complexity bounds are established in Theorems~\ref{thm:minmax_stoch_concavity} and~\ref{thm:minmax_stoch_nox2}, respectively, in the appendix.
   

%% file: 4exps.tex
\section{Numerical experiments}\label{sec:exp}
In this section, we evaluate the numerical performance of the proposed methods. Additional experimental details are provided in Section~\ref{sec:imp_detail}.
\subsection{Constrained linear optimization}
We consider the following instance of \eqref{eq:blo_cons} from ~\citep{lu2024firstorder,lu2025solving}:
\begin{equation}
    \label{eq:clo}
    \min_{\vx,\vy}\; c^T\vx+d^T\vy+\mathbf{1}_{[-1,1]^n}(\vx)\quad\textbf{s.t.}\quad\vy\in\arg\min_\vz\big\{{\tilde d}^T\vz+\mathbf{1}_{[-1,1]^m}\vz\;|\;{\widetilde A}\vx+{\widetilde B}\vz-{\tilde b}\leq 0\big\},
\end{equation}
where $c\in\RR^n$, $d,\tilde d\in\RR^m$, $\tilde b\in\RR^l$, $\widetilde A\in\RR^{l\times n}$, and $\widetilde B\in\RR^{l\times m}$. Five instances are generated per tuple $(n,m,l)$ as described in Appendix~\ref{sec:clo_instances}, and we set the initial solution as $(\vx^{(0)},\vy^{(0)})=(0,0)$. 
\begin{figure}[t]
\centering
\begin{subfigure}{0.32\linewidth}
    \centering
    \includegraphics[width=\linewidth]{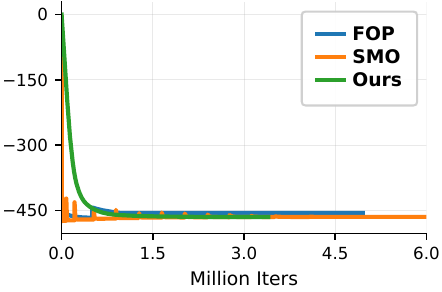}
    \caption{$f(\vx_1,\vy_1)$}
\end{subfigure}
\begin{subfigure}{0.33\linewidth}
    \centering
    \includegraphics[width=\linewidth]{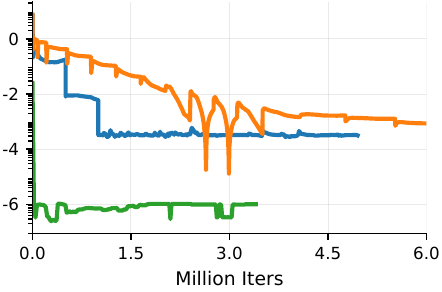}
    \caption{$\log_{10}(p(\vx_1,\vy_1)-d(\vx_1,\vy_2))$}
\end{subfigure}
\begin{subfigure}{0.32\linewidth}
    \centering
    \includegraphics[width=\linewidth]{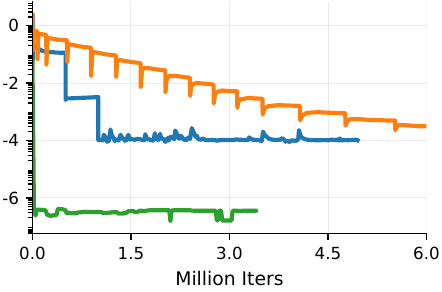}
    \caption{$\log_{10}(\|[\bar{\vg}(\vx_1, \vz_1)]_+\|)$}
\end{subfigure}
\caption{Performance for an instance with $(n,m,l)=(300,300,15)$ and tuned Lipschitz constants.}
\vspace{-1ex}
\label{fig:clbo}
\end{figure}

We compare our method against 
FOP~\citep{lu2024firstorder} and SMO~\citep{lu2025solving}. 
Overall, our method achieves competitive results against the two baselines when all methods use respective theoretical Lipschitz constants to determine step sizes, and our method performs significantly better after we tune the Lipschitz constants for each method. In Figure~\ref{fig:clbo}, we plot for one instance the values of $f(\vx_1,\vy_1)$, $p(\vx_1,\vy_1)-d(\vx_1,\vy_2)$, and $\|[\bar{\vg}(\vx_1, \vz_1)]_+\|$ over the number of gradient steps. Our method reduces the last two quantities faster than the others. In Table~\ref{tab:clo}, we report $f(\vx_1,\vy_1)$ at termination. 


\begin{table}[htbp]
    \centering
    \begin{tabular}{c||cc|cc|cc}
    \toprule
        $(n,m,\ell)$ & FOP & FOP-t & SMO & SMO-t & Ours & Ours-t \\
    \midrule
        $(100,100,5)$   & -141.67 & -141.69 & -145.78 & -145.63 & -143.14 & \textbf{-145.85} \\
        $(200,200,10)$  & -278.67 & -282.31 & -305.95 & -305.97 & -300.67 & \textbf{-310.84} \\
        $(300,300,15)$  & -439.92 & -448.36 & -455.91 & -461.45 & -438.58 & \textbf{-462.92} \\
    \bottomrule
    \end{tabular}
    \vspace{1em}
    \caption{Upper objective $f(\vx_1,\vy_1)$ at termination for~\eqref{eq:clo}. ``-t'' means using tuned a Lipschitz constant to construct step sizes, and it means using the theoretical Lipschitz constant without ``-t''.}
    \vspace{-1.5em}
    \label{tab:clo}
\end{table}

\subsection{Hyperparameter tuning for group distributionally robust optimization (Group DRO)}
Empirical risk minimization (ERM) assigns uniform weights to training samples,  resulting in potentially poor worst-group performance on an unbalanced dataset~\citep{sagawa2020distributionally}. Instead of creating balanced data, which are expensive~\citep{ren2018learning}, we propose a deep bilevel DRO framework by decomposing the model into a feature extractor $\vx_1$ and a linear classifier $\vy_1$, where a representation of $\va$, denoted by $\texttt{Enc}(\va;\vx_1)$, is produced by $\vx_1$, on which $\vy_1$ predicts label $b$. This separation isolates non-convex representation learning from convex classification~\citep{kang2019decoupling, kirichenko2023last}. We denote the loss on data $(\va,b)$ by $\ell^{\vx_1}(\vy_1,(\va,b))$.


\paragraph{Lower Level: Group DRO}
Let $D_{\text{train}}=\cup_{g=1}^G D_{\text{train}}^{(g)}$. 
We define
\begin{equation}\label{eq:dro_lower}
\min_{\vy_1} \max_{\vy_2 \in \mathcal{Y}}
\;\;\mathcal{L}_{\text{train}}(\vx_1,\vy_1,\vy_2;\eta)
:= \vy_2^T \mathcal{L}^{\vx_1}(\vy_1)
\;-\; \frac{\eta}{2}\|\vy_2 - \tfrac{1}{G}\mathbf{1}\|^2
\;+\; \lambda \|\vy_1\|^2,
\end{equation}
where $\mathcal{L}^{\vx_1}(\vy_1) := (\mathcal{L}_1^{\vx_1},\dots,\mathcal{L}_G^{\vx_1})$, $\mathcal{L}_g^{\vx_1}(\vy_1) := \frac{1}{|D_{\text{train}}^{(g)}|} \sum_{(\va,b)\in D_{\text{train}}^{(g)}} \ell^{\vx_1}(\vy_1,(\va,b))$, and $\mathcal{Y}=\{\vy\ge 0:\mathbf{1}^T \vy=1\}$.
The learnable weight $\vy_2$ emphasizes high-loss groups with $\eta$ controlling deviation from uniform weights. The lower problem is convex–concave in $(\vy_1,\vy_2)$ due to the linear classifier.

\paragraph{Upper Level: Hyperparameter Tuning}
We tune $\eta$ and $\vx_1$ using validation data $D_{\text{val}}$, partitioned analogously. The bilevel problem is
\begin{equation}\label{eq:dro_nested}
\min_{\vx_1,\vy_1,\vy_2,\eta\geq 0} \max_{\vx_2 \in \mathcal{Y}}
\;\; \vx_2^T \mathcal{L}_{\text{val}}^{\vx_1}(\vy_1)
+ \lambda \|\vx_1\|^2
\quad
\textbf{s.t.}\quad
(\vy_1,\vy_2) \in \arg\min_{\vz_1}\max_{\vz_2\in\mathcal{Y}}
\mathcal{L}_{\text{train}}(\vx_1,\vz_1,\vz_2;\eta),
\end{equation}
where $\mathcal{L}_{\text{val}}^{\vx_1}(\vy_1) := (\mathcal{L}_{\text{val},1}^{\vx_1},\dots,\mathcal{L}_{\text{val},G}^{\vx_1})$.
Here,  $\max_{\vx_2 \in \mathcal{Y}}$ computes the worst-group validation loss. 

\paragraph{Results}
We use the stochastic version of our method with $\rho=100$ and minibatch size of 256 for CelebA and 128 for Waterbirds.
We evaluate our method on CelebFaces Attributes (CelebA)~\citep{liu2015faceattributes} and Waterbirds~\citep{sagawa2020distributionally} datasets, where binary labels are confounded by two spurious attributes, yielding 4 groups, where the minority group contributes to 3\% of the entire training set. Our method yields average performance close to DRO while outperforming it on worst group accuracy, with and without $\ell_2$ penalty (See Table~\ref{tab:dro_results}).
See Appendix~\ref{sec:dro_hyper_tuning_details} for additional experimental details.






\begin{table}[ht]
\centering

\begin{tabular}{c c cc cc}
\toprule
\textbf{Dataset} & $\ell_2$ \textbf{Penalty}
& \multicolumn{2}{c}{\textbf{Worst Group}}
& \multicolumn{2}{c}{\textbf{Average}} \\
& $\lambda$ & Group DRO & Ours & Group DRO & Ours \\
\midrule

\multirow{2}{*}{\textbf{CelebA}}
 & 1e-4
 & $42.04 \pm 2.10$ & $\mathbf{86.06 \pm 2.48}$
 & $94.84 \pm 0.06$ & $91.11 \pm 1.07$ \\
 & 0.1
 & $85.37 \pm 1.40$ & $\mathbf{87.52 \pm 0.83}$
 & $93.70 \pm 0.26$ & $90.26 \pm 1.13$ \\
\midrule

\multirow{2}{*}{\textbf{Waterbirds}}
 & 1e-4
 & $76.53 \pm 1.72$ & $\mathbf{83.61 \pm 5.58}$
 & $97.25 \pm 0.07$ & $90.05 \pm 2.00$ \\
 & 1.0
 & $84.22 \pm 1.32$ & $\mathbf{87.13 \pm 1.57}$
 & $96.14 \pm 0.10$ & $90.89 \pm 2.77$ \\
\bottomrule
\end{tabular}

\vspace{1em}
\caption{Worst-group and average accuracy for DRO and our method on the test set.}
\label{tab:dro_results}
\end{table}
\vspace{-0.5cm}

\section{Limitations} \label{sec:limit}
This work has the following limitations: (1)  Assumption~\ref{assume:constrained} requires compactness of the feasible domains. (2) The lifted $\epsilon$-KKT point in Definition~\ref{def:epsilon_KKT} is defined in a lifted space with additional variables $\vz_1$ and $\vz_3$. It would be desirable to study a stronger stationarity notion that only involves the original variables. (3) In the stochastic case, the $\tilde{O}(\epsilon^{-9})$ complexity is high. We believe using the variance-reduction technique can reduce this complexity, which we leave it as future work.


\section{AI use statement}\label{sec:AI_statement}
All mathematical results and proofs in this paper were developed
by the authors except  Lemma~\ref{thm:con_lagrangian}, which was initially generated by
Gemini and polished and verified by the authors.
GPT-6 was used to assist with language editing
and to improve the clarity and presentation of the manuscript. It was also used to assist in writing portions of the code
for the numerical experiments. All AI-assisted code was carefully
reviewed and verified by the authors. 
The authors take full responsibility for the correctness of all
mathematical results and proofs, experimental implementations,
and the final manuscript.

%% file: apdx.tex
\section{Additional applications of \eqref{eq:bilevel}}
\label{sec:applications}
Model~\eqref{eq:bilevel} can be potentially applied to not only many machine learning problems but also bilevel optimization with a convex lower-level problem with nonlinear convex inequality constraints. We list some of these applications here. 

\paragraph{Distributionally robust learning with tuning hyperparameters}
Let $\va$ be a feature vector and $b$ be a target variable. Suppose a model $s(\vtheta,\va)$ parameterized by $\vtheta$ is used to predict $b$ based on $\va$. Let $\ell(s(\vtheta,\va), b)$ be the loss of the prediction. Given a training dataset $D_{\text{train}}=\{(\va_t^i,b_t^i)\}_{i=1}^{n_{\text{train}}}$ and a validation dataset $D_{\text{val}}=\{(\va_v^i,b_v^i)\}_{i=1}^{n_{\text{val}}}$, let $\vy_2\in\mathbb{R}^{n_{\text{train}}}$ and $\vx_2\in\mathbb{R}^{n_{\text{val}}}$ be the weights on the training and validation instances, respectively. The weighted average losses on the training and validation sets are defined as 
\begin{align}
\label{eq:LDRO_train}
\mathcal{L}_{\text{DRO}} (D_{\text{train}}; \vtheta,\vy_2) := &\sum_{i=1}^{n_{\text{train}}} y_2^i \cdot \ell(s(\vtheta,\va_t^i), b_t^i)\\
\label{eq:LDRO_val}
\mathcal{L}_{\text{DRO}} (D_{\text{val}}; \vtheta,\vx_2) :=& \sum_{i=1}^{n_{\text{val}}} x_2^i \cdot \ell(s(\vtheta,\va_v^i), b_v^i).
\end{align}
The distributionally robust learning problem with tuning hyperparameters can be formulated as 
\begin{align}
\label{eq:DRO_tuning}
    \min_{\eta\geq0,\vtheta}~& \max_{\vx_2 \in \Delta_{n_{\text{val}},p}} \mathcal{L}_{\text{DRO}} (D_{\text{val}}; \vtheta,\vx_2)\\\nonumber
    \text{s.t.} \quad &\vtheta\in\argmin_{\vtheta'}\max_{\vy_2\in \Delta_{n_{\text{train}},1}} \left\{\mathcal{L}_{\text{DRO}} (D_{\text{train}}; \vtheta',\vy_2)-\frac{\eta}{2} \|\vy_2 -  \frac{1}{n_{\text{train}}}\mathbf{1}\|^2\right\}
\end{align}
where  $\Delta_{n,p}=\{\vy\in \mathbb{R}_+^n| \mathbf{1}^\top \vy=1, y_i\in[0,p], i=1,\dots,n\}$ is the truncated probability simplex. The objective function above is chosen to minimize the average loss on the $1/(pn)\times 100\%$ worst cases. However, if one directly minimizes the average loss on the $1/(pn)\times 100\%$ worst training instances, the model is not guaranteed to minimize the $1/(pn)\times 100\%$ worst instances in the populations due to the generalization error, especially, when  $n_{\text{train}}$ is small. Therefore, a regularization term $\frac{\eta}{2} \|\vy_2 - 1/n_{\text{train}}\mathbf{1}\|^2$ with a changeable hyperparameter $\eta\geq0$ is introduced to the lower level problem. The goal is to minimize the average loss on the $1/(pn)\times 100\%$ worst cases on the validation set by the trained model through optimizing the parameter $\eta$. \eqref{eq:DRO_tuning} is thus an instance of \eqref{eq:bilevel} with $\vx_1=\eta$ and $\vy_1=\vtheta$.

\paragraph{Distributionally robust meta-learning} The robust learning framework in \eqref{eq:DRO_tuning} can be extended to meta learning
~\citep{franceschi2018bilevel, finn2017model, rajeswaran2019meta,collins2020task}. Suppose there are $|\mathcal{T}|$ tasks, indexed by $\{1,\dots,\mathcal{T}\}$, and the training and testing datasets for task $\tau$ are denoted by $D_{\text{train}}^\tau=\{(\va_t^{\tau,i},b_t^{\tau,i})\}_{i=1}^{n_{\text{train}}}$ and $D_{\text{val}}^\tau=\{(\va_v^{\tau,i},b_v^{\tau,i})\}_{i=1}^{n_{\text{val}}}$, respectively, for $\tau\in\mathcal{T}$. Let $\vtheta$ be the meta-parameter and $\vtheta^\tau$ be the task-specifc parameter for task $\tau$. The distributionally meta-learning can be formulated as follows
\begin{align}
\label{eq:DRO_meta_tuning}
    \min_{\vtheta,\eta\geq0,\lambda\in[\underline{\lambda},\bar\lambda],\vtheta^\tau,\tau\in\mathcal{T}}~&  \max_{\vx_2^\tau \in\Delta_{n_{\text{val}},p},\tau\in\mathcal{T}}\sum_{\tau\in\mathcal{T}}\mathcal{L}_{\text{DRO}} (D_{\text{val}}^\tau; \vtheta^\tau,\vx_2^\tau)\\\nonumber
    \text{s.t.}\quad &\vtheta^\tau\in\argmin_{\vtheta'}\max_{\vy_2^\tau\in \Delta_{n_{\text{train}},1}} \left\{
    \begin{array}{l}
    \mathcal{L}_{\text{DRO}} (D_{\text{train}}^\tau; \vtheta',\vy_2^\tau)\\
    -\frac{\eta}{2} \|\vy_2^\tau - \frac{1}{n_{\text{train}}}\mathbf{1}\|^2+\frac{\lambda}{2}\|\vtheta'-\vtheta \|^2
    \end{array}
    \right\}\text{ for }\tau\in\mathcal{T}.
\end{align}
In \eqref{eq:DRO_meta_tuning}, we perform distributionally robust learning with tuning hyperparameters for each task $\tau\in\mathcal{T}$ just like \eqref{eq:DRO_tuning}. The main difference is that a meta $\vtheta$ is introduced and 
an additional regularization term $\frac{\lambda}{2}\|\vtheta'-\vtheta \|^2$ is included in each lower-level problem to ensure $\vtheta^\tau$ is not too far away from $\vtheta$. It is still an instance of  \eqref{eq:bilevel} with $\vx_1=(\vtheta,\eta,\lambda)$, $\vy_1=(\vtheta^\tau)_{\tau\in\mathcal{T}}$, $\vx_2=(\vx_2^\tau)_{\tau\in\mathcal{T}}$ and $\vy_2=(\vy_2^\tau)_{\tau\in\mathcal{T}}$.

\section{Proof of Lemma~\ref{thm:connection}}
\label{sec:lemma_connection}
\begin{proof}[Proof of Lemma~\ref{thm:connection}]
Suppose $(\vx_1, \vy_1, \vy_2,\vx_2, \vz_1, \vz_2)$   is an $\epsilon$-primal-dual stationary point of \eqref{eq:new_minmax}. The first four conditions in Definition~\ref{def:epsilon_KKT} are trivially implied by the two conditions in Definition~\ref{def:epsilon_stationary}. 
By \eqref{eq:flow}, \eqref{eq:lagrange} and the fact that $p(\vx_1,\vy_1)-d(\vx_1,\vy_2)\geq0$, we have 
    \begin{align*}
   f_{\text{low}}  + \rho\left[p(\vx_1,\vy_1)-d(\vx_1,\vy_2)\right]\leq&~\max_{\vx_2'}f(\vx_1, \vx_2', \vy_1, \vy_2) + \rho\left[p(\vx_1,\vy_1)-d(\vx_1,\vy_2)\right]\\
    = &~\max_{\vx_2', \vz_1', \vz_2'}P_{\rho}(\vx_1, \vx_2', \vy_1, \vy_2, \vz_1', \vz_2').
\end{align*}
This  implies that 
\begin{align*}
 p(\vx_1,\vy_1)-d(\vx_1,\vy_2)\leq \rho^{-1}\left(\max_{\vx_2', \vz_1', \vz_2'}P_{\rho}(\vx_1, \vx_2', \vy_1, \vy_2, \vz_1', \vz_2')- f_{\text{low}}\right),
\end{align*}
which further implies $p(\vx_1,\vy_1)-d(\vx_1,\vy_2)\leq O(\epsilon)$ because of \eqref{eq:feasibility_bound_condition}. Hence, $(\vx_1, \vy_1, \vy_2)$ is a lifted $O(\epsilon)$-KKT  point of~\eqref{eq:bilevel}.
\end{proof}

\section{A deterministic first-order method for nonconvex-concave minimax problem}
\label{sec:wcc_minmax}

Our goal is to solve \eqref{eq:bilevel} and \eqref{eq:bilevel_equality}  by solving its penalty formulation \eqref{eq:new_minmax}. Under Assumption~\ref{assump:dom}, \eqref{eq:new_minmax} is an instance of the following  nonconvex-concave minimax problem
\begin{align}
\label{eq:minmax}
 H^*:=\min_{x}\max_{y}\left\{ H(x,y):=h(x,y)+p(x)-q(y)\right\}
\end{align}
that satisfies the following assumptions.
\begin{assumption}\label{assump:h}~
    \begin{enumerate}
    \item $\dom p$ and $\dom q$ are compact.
        \item $h(x,y)$ is concave in $y$ and $L_{\nabla h}$-smooth on $\dom p \times \dom q$.
        \item $p$ and $q$ are proper closed convex, and their proximal operators can be computed exactly. 
    \end{enumerate}
\end{assumption}
We introduce the following notations
\begin{align}
\label{eq:DpDq}
&D_p := \max\{\|u-v\| |u,v \in \dom\,p\}, \quad D_q := \max\{\|u-v\| |u,v \in \dom\,q \},\\\label{eq:H_low}
&H_{\text{low}}:= \min \{H(x,y)|(x,y)\in \dom p \times \dom q\}. \quad
\end{align}
We say $(x^{\epsilon},y^{\epsilon})\in\text{dom}p\times\text{dom}q$ is an \emph{$\epsilon$-primal-dual stationary point} of \eqref{eq:minmax} if 
\begin{equation}
\label{eq:pd_stationary_H}
\text{dist}(0,\partial_xH(x^{\epsilon},y^{\epsilon}))\leq \epsilon~\text{ and }~\text{dist}(0,\partial_y(-H(x^{\epsilon},y^{\epsilon})))\leq \epsilon.
\end{equation}

To find an $\epsilon$-primal-dual stationary point of \eqref{eq:minmax}, we directly use the state-of-the-art method, the inexact proximal-point method by \citet[Algorithm 2]{lu2023first},  which generates the next iterate $\left(x^{k+1}, y^{k+1}\right)$ from the current iterate $\left(x^{k}, y^{k}\right)$ by approximately solving a sequence of strongly-convex-strongly-concave (SCSC) minimax subproblems as follows
\begin{align}
\label{eq:minmax_prox_h}
\left(x^{k+1}, y^{k+1}\right)\approx\arg\min_{x} \max_{y} \left\{H_k(x,y):=h_k(x,y)+p(x)-q(y)\right\},
\end{align}
where
\begin{align}
\label{hk}
h_k(x,y)=h(x,y)-\epsilon\|y-y^0\|^2/(4D_q)+L_{\nabla h} \|x-x^k\|^2.
\end{align}

The subproblem \eqref{eq:minmax_prox_h} is an instance of the following SCSC minimax problem
\begin{equation}\label{ea-prob}
\min_{x}\max_{y}\left\{ \h(x,y)+p(x)-q(y)\right\},
\end{equation}
which satisfies the following assumptions.
\begin{assumption}\label{assump:barh}~
    \begin{enumerate}
    \item $\dom p$ and $\dom q$ are compact.
        \item $\bar{h}(x,y)$ is $\sigma_x$-strongly convex in $x$ and $\sigma_y$-strongly concave in $y$ and $L_{\nabla \bar{h}}$-smooth on $\dom p \times \dom q$.
        \item $p$ and $q$ are proper closed convex and their proximal operators can be computed exactly. 
    \end{enumerate}
\end{assumption}
Subproblem \eqref{eq:minmax_prox_h} is an instance of \eqref{ea-prob} with
\begin{equation}
\label{eq:mapbarh}
\begin{aligned}
\h(x,y) \leftarrow&~h(x,y)-\epsilon\|y-y^0\|^2/(4D_q)+L_{\nabla h} \|x-x^k\|^2.
\end{aligned}
\end{equation}
Additionally, when \eqref{eq:minmax} satisfies Assumption~\ref{assump:h}, \eqref{eq:minmax_prox_h} is an instance of \eqref{ea-prob} satisfying Assumption~\eqref{assump:barh} with
\begin{equation}
\label{eq:mapbarL}
\begin{aligned}
    L_{\nabla\h} \leftarrow&~L_{\nabla h}+\max\{2L_{\nabla h},\epsilon/(2D_q)\},\\
    \sigma_x\leftarrow&~ L_{\nabla h},\\
    \sigma_y\leftarrow&~\epsilon/(2D_q).\\
\end{aligned}
\end{equation}
where $D_q$ is defined in \eqref{eq:DpDq}.

\begin{algorithm}[t]
\caption{An optimal first-order method for \eqref{ea-prob}: \texttt{OptFOM}$(\bar\epsilon,x^0,y^0,\bar{h},p,q,L_{\nabla \h},\sigma_x,\sigma_y)$}
\label{mmax-alg1}
\begin{algorithmic}[1]
\STATE {\bfseries Input:} targeted stationarity level $\bar\epsilon>0$, initial solution $(x^0,y^0)\in\dom\,p\times\dom\,q$, $(\bar{h},p,q)$ from \eqref{ea-prob}, and $(L_{\nabla \h},\sigma_x,\sigma_y)$ from Assumption~\ref{assump:barh}.
\STATE  $z^0_f=z^0=-\sigma_x x^0$, $y^0_f=y^0$,  $\bar \alpha=\min\left\{1,\sqrt{8\sigma_y/\sigma_x}\right\}$, $\eta_z=\sigma_x/2$, $\eta_y=\min\left\{1/(2\sigma_y),4/(\bar \alpha\sigma_x)\right\}$, $\zeta=\left(2\sqrt{5}(1+8L_{\nabla\h}/\sigma_x)\right)^{-1}$, $\gamma_x=\gamma_y=8\sigma_x^{-1}$, and $\hat\zeta=\min\{\sigma_x,\sigma_y\}/L_{\nabla \h}^2$.
\FOR{$k=0,1,2,\ldots$}
\STATE $(z^k_g,y^k_g)=\bar \alpha(z^k,y^k)+(1-\bar \alpha)(z^k_f,y^k_f)$.
\STATE $(x^{k,-1},y^{k,-1})=(-\sigma_x^{-1}z^k_g,y^k_g)$.
\STATE $x^{k,0}=\prox_{\zeta\gamma_xp}(x^{k,-1}-\zeta\gamma_x a^k_x(x^{k,-1},y^{k,-1}))$.
\STATE $y^{k,0}=\prox_{\zeta\gamma_y q}(y^{k,-1}-\zeta\gamma_y a^k_y(x^{k,-1},y^{k,-1}))$.
\STATE $b^{k,0}_x=\frac{1}{\zeta\gamma_x}(x^{k,-1}-\zeta\gamma_x a^k_x(x^{k,-1},y^{k,-1})-x^{k,0})$.
\STATE $b^{k,0}_y=\frac{1}{\zeta\gamma_y}(y^{k,-1}-\zeta\gamma_y a^k_y(x^{k,-1},y^{k,-1})-y^{k,0})$.
\STATE $t=0$.
\WHILE{ $\begin{array}{c}
\gamma_x\|a^k_x(x^{k,t},y^{k,t})+b^{k,t}_x\|^2+\gamma_y\|a^k_y(x^{k,t},y^{k,t})+b^{k,t}_y\|^2\\
>\gamma_x^{-1}\|x^{k,t}-x^{k,-1}\|^2+\gamma_y^{-1}\|y^{k,t}-y^{k,-1}\|^2
\end{array}$}
\STATE $\beta_t=2/(t+3)$
\STATE $x^{k,t+1/2}=x^{k,t}+\beta_t(x^{k,0}-x^{k,t})-\zeta\gamma_x(a^k_x(x^{k,t},y^{k,t})+b^{k,t}_x)$.
\STATE $y^{k,t+1/2}=y^{k,t}+\beta_t(y^{k,0}-y^{k,t})-\zeta\gamma_y(a^k_y(x^{k,t},y^{k,t})+b^{k,t}_y)$.
\STATE $x^{k,t+1}=\prox_{\zeta\gamma_x p}(x^{k,t}+\beta_t(x^{k,0}-x^{k,t})-\zeta\gamma_x a^k_x(x^{k,t+1/2},y^{k,t+1/2}))$.
\STATE $y^{k,t+1}=\prox_{\zeta\gamma_y q}(y^{k,t}+\beta_t(y^{k,0}-y^{k,t})-\zeta\gamma_y a^k_y(x^{k,t+1/2},y^{k,t+1/2}))$.
\STATE $b^{k,t+1}_x=\frac{1}{\zeta\gamma_x}(x^{k,t}+\beta_t(x^{k,0}-x^{k,t})-\zeta\gamma_x a^k_x(x^{k,t+1/2},y^{k,t+1/2})-x^{k,t+1})$.
\STATE $b^{k,t+1}_y=\frac{1}{\zeta\gamma_y}(y^{k,t}+\beta_t(y^{k,0}-y^{k,t})-\zeta\gamma_y a^k_y(x^{k,t+1/2},y^{k,t+1/2})-y^{k,t+1})$.
\STATE $t \leftarrow t+1$.
\ENDWHILE
\STATE $(x^{k+1}_f,y^{k+1}_f)=(x^{k,t},y^{k,t})$.
\STATE $(z^{k+1}_f,w^{k+1}_f)=(\nabla_x\hat h(x^{k+1}_f,y^{k+1}_f)+b^{k,t}_x,-\nabla_y\hat h(x^{k+1}_f,y^{k+1}_f)+b^{k,t}_y)$.
\STATE $z^{k+1}=z^k+\eta_z\sigma_x^{-1}(z^{k+1}_f-z^k)-\eta_z(x^{k+1}_f+\sigma_x^{-1}z^{k+1}_f)$.
\STATE $y^{k+1}=y^k+\eta_y\sigma_y(y^{k+1}_f-y^k)-\eta_y(w^{k+1}_f+\sigma_yy^{k+1}_f)$.
\STATE $x^{k+1}=-\sigma_x^{-1}z^{k+1}$.
\STATE $\tx^{k+1}=\prox_{\hat\zeta p}(x^{k+1}-\hat\zeta\nabla_x\h(x^{k+1},y^{k+1}))$.
\STATE $\ty^{k+1}=\prox_{\hat\zeta q}(y^{k+1}+\hat\zeta\nabla_y\h(x^{k+1},y^{k+1}))$.
\STATE Terminate the algorithm and output $(\tx^{k+1},\ty^{k+1})$ if
\begin{equation*}
\|\hat\zeta^{-1}(x^{k+1}-\tx^{k+1},\ty^{k+1}-y^{k+1})-(\nabla \h(x^{k+1},y^{k+1})-\nabla \h(\tx^{k+1},\ty^{k+1}))\|\leq\bar\epsilon.
\end{equation*}
\ENDFOR
\end{algorithmic}							
\end{algorithm}

The optimal algorithms for \eqref{ea-prob} under Assumption~\eqref{assump:barh} include \citet[Algorithm 4]{kovalev2022first} and \citet[Algorithm 1]{lu2023first}, which are also used in \cite{lu2024firstorder,lu2025solving} as a subroutine for solving bilevel optimization. We present \citet[Algorithm 1]{lu2023first} in Algorithm \ref{mmax-alg1}, where the notations $\hat h$, $a^k_x$ and $a^k_y$ are defined as follows:
\begin{align}
&\hat h(x,y)=\h(x,y)-\sigma_x\|x\|^2/2+\sigma_y\|y\|^2/2,\\
&a^k_x(x,y)=\nabla_x\hat h(x,y)+\sigma_x(x-\sigma_x^{-1}z^k_g)/2,\\
&a^k_y(x,y)=-\nabla_y\hat h(x,y)+\sigma_y y+\sigma_x(y-y^k_g)/8,
\end{align}
and $y^k_g$ and $z^k_g$ are generated at iteration $k$ of Algorithm \ref{mmax-alg1}. According to \citet[Theorem 1]{lu2023first}, under Assumption~\ref{assump:barh}, Algorithm \ref{mmax-alg1} finds an $\bar\epsilon$-primal-dual stationary point of  \eqref{ea-prob} with $O(\frac{L_{\nabla \h}}{\sqrt{\sigma_x\sigma_y}}\ln(1/\bar\epsilon))$ oracle complexity. Here, $\bar\epsilon$ is the first input of Algorithm \ref{mmax-alg1}.

Using Algorithm~\ref{mmax-alg1} to approximately solve subprpoblem \eqref{eq:minmax_prox_h} in each main iteration leads to the inexact proximal-point method in \citep[Algorithm 2]{lu2023first} for \eqref{eq:minmax}, which is presented in Algorithm~\ref{alg:first-order}. The oracle complexity of Algorithm~\ref{alg:first-order} is established in \citet[Theorem 2]{lu2023first} and is re-stated in Theorem~\ref{thm:lu_thm}.

\begin{algorithm}[t]
\caption{An inexact proximal point method \citep[Algorithm 2]{lu2023first} for~\eqref{eq:minmax}}
\label{alg:first-order}
\begin{algorithmic}[1]
\STATE {\bfseries Input:} Outer targeted stationarity level $\epsilon>0$, inner targeted stationarity level $\hat{\epsilon}>0$, initial solution $(x^0,y^0)\in\dom\,p\times\dom\,q$, $(h,p,q)$ from \eqref{eq:minmax},  $L_{\nabla h}$ from Assumption~\ref{assump:h}, and $D_q$ in \eqref{eq:DpDq}.
\FOR{$k=0,1,2,\ldots$}
\STATE Set $\bar{h}$ as in $\eqref{eq:mapbarh}$ and set $(L_{\nabla \h},\sigma_x,\sigma_y)$ as in \eqref{eq:mapbarL}.
\STATE 
$$
(x^{k+1}, y^{k+1})\leftarrow\texttt{OptFOM}\left(\frac{\hat{\epsilon}}{k+1},x^{k}, y^{k}, \bar{h}, p,q,L_{\nabla \h},\sigma_x,\sigma_y\right)
$$
\IF{
$\|x^{k+1}-x^k\|\leq\epsilon/(4L_{\nabla h})$
}
\STATE Return $(x^{\epsilon},y^{\epsilon})=(x^{k+1},y^{k+1})$.
\ENDIF
\ENDFOR
\end{algorithmic}
\end{algorithm}

\begin{theorem}\label{thm:lu_thm}Suppose Assumption~\ref{assump:h} holds. Let $H$, $H^*$, $D_p$, $D_q$ and $H_{\text{low}}$ be defined in \eqref{eq:minmax}, \eqref{eq:DpDq} and \eqref{eq:H_low}. Let 
    \begin{align*}
        \alpha = &~\min \{ 1, \sqrt{4\epsilon/(D_q L_{\nabla h})} \},\\
        \delta = &~(2+ \alpha^{-1})L_{\nabla h}D_p^2 + \max\{ \epsilon/D_q, \alpha L_{\nabla h}/4 \} D_q^2,\\
         K = &~\left\lceil 16(\max_{y}H(x^0,y)- H^* + \epsilon D_q/4)L_{\nabla h} \epsilon^{-2} + 32 \hat{\epsilon}^2(1+4D_q^2 L_{\nabla h}^2\epsilon^{-2}) \epsilon^{-2}-1)\right\rceil_+\\
        C =  &~\frac{4\max\{ \frac{1}{2L_{\nabla h} }, \min \{ \frac{D_q}{\epsilon}, \frac{4}{\alpha L_{\nabla h}} \} \}\Big(\delta + 2\alpha^{-1} (H^* - H_{\text{low}} + \epsilon D_q/4 + L_{\nabla h}D_p^2)\Big)}{[(3L_{\nabla h} + \epsilon/(2D_q))^2/\min \{ L_{\nabla h}, \epsilon/(2D_q)\} + 3 L_{\nabla h} + \epsilon/(2D_q)]^{-2}\hat{\epsilon}^2}\\
         N = &~(\left\lceil 96\sqrt{2}(1+(24L_{\nabla h} + 4\epsilon/D_q)L_{\nabla h}^{-1}) \right\rceil + 2)\max\{ 2, \sqrt{D_qL_{\nabla h}\epsilon^{-1}} \}\\ &~~~~~~~~~~\times \Big((K+1) \times (\log(C))_{+} + K + 1 + 2K\log(K+1) \Big).
    \end{align*}
    Then Algorithm~\ref{alg:first-order} outputs an $\epsilon$-primal-dual stationary point $(x^{\epsilon},y^{\epsilon})$ of ~\eqref{eq:minmax} (e.g., \eqref{eq:pd_stationary_H} holds) in oracle complexity no more than $N$. 
    Moreover, $(x^{\epsilon},y^{\epsilon})$ satisfies
    \begin{align}
    \label{eq:HxK_bound}
    \max_{y}H(x^{\epsilon},y)\leq \max_{y} H(x^0, y)+\epsilon D_q/4+2\hat{\epsilon}^2(L_{\nabla h}^{-1}+4D_q^2 L_{\nabla h} \epsilon^{-2}).
    \end{align}
\end{theorem}
\begin{remark}
According to this theorem, Algorithm~\ref{alg:first-order} finds  an $\epsilon$-primal-dual stationary point of ~\eqref{eq:minmax} in $K = O(L_{\nabla h}\epsilon^{-2}+\hat{\epsilon}^2L_{\nabla h}^2\epsilon^{-4})$ iterations with a total oracle complexity of $N=\tilde{O}(\sqrt{L_{\nabla h}}\epsilon^{-0.5}K)$. 
If $\hat{\epsilon}$ is further set to be $O(\epsilon)$ and $L_{\nabla h}=O(1)$, the complexity becomes $N=\tilde{O}(\epsilon^{-2.5})$.
\end{remark}

\section{Proof of Theorem~\ref{thm:minmax}}\label{sec:proof_thm1}
\begin{proof}[Proof of Theorem~\ref{thm:minmax}]
We first verify that the deterministic case of Algorithm~\ref{alg:first-order-new} is exactly Algorithm~\ref{alg:first-order} applied to \eqref{eq:new_minmax}.

By \eqref{eq:lagrange_split},  problem \eqref{eq:new_minmax} is in the form of \eqref{eq:minmax} with
\begin{equation}
\label{eq:maph}
\begin{aligned}
x \leftarrow&~(\vx_1,\vy_1,\vy_2), \\
y \leftarrow&~(\vx_2,\vz_1,\vz_2), \\
h(x,y) \leftarrow&~P_1(\vx_1,\vx_2, \vy_1,\vy_2, \vz_1, \vz_2),\\
p(x) \leftarrow&~P_2(\vx_1,\vy_1,\vy_2),\\
q(y) \leftarrow&~P_3(\vx_2,\vz_1, \vz_2).
\end{aligned}
\end{equation}
where $P_1$, $P_2$ and $P_3$ are defined in \eqref{eq:P1}, \eqref{eq:P2} and \eqref{eq:P3}. Moreover, according to Assumption~\ref{assump:dom}, 
the smoothness parameters of \eqref{eq:minmax} and the diameters of the domains, when specified to \eqref{eq:new_minmax}, are
\begin{equation}
\label{eq:mapLh}
\begin{aligned}
    L_{\nabla h} \leftarrow &~L_{\nabla P_1}=L_{\nabla f} + 2\rho L_{\nabla \tilde{f}},\\
    D_p\leftarrow &~D_1,\\
    D_q\leftarrow &~D_2,
\end{aligned}
\end{equation}
where $D_1$ and $D_2$ are defined in \eqref{eq:D1} and \eqref{eq:D2}. 
Given the mappings in \eqref{eq:maph} and \eqref{eq:mapLh}, the procedure in \eqref{eq:minmax_prox_h}, when applied to \eqref{eq:new_minmax}, becomes exactly \eqref{eq:minmax_prox} if we set  $\rho_1$ and $\rho_2$ as in Algorithm~\ref{alg:first-order-new}, i.e., 
\begin{align}
\label{eq:rho}
\rho_1=2L_{\nabla P_1}=2L_{\nabla f} + 4\rho L_{\nabla \tilde{f}},\quad \rho_2=\epsilon/(2D_2).
\end{align}

Similarly, subproblem \eqref{eq:minmax_prox} is an instance of \eqref{ea-prob} with
\begin{equation}
\label{eq:mapPrho}
\begin{aligned}
x\leftarrow&~(\vx_1,\vy_1,\vy_2), \\
y\leftarrow&~(\vx_2,\vz_1,\vz_2), \\
\h(x,y) \leftarrow&~P_1(\vx_1,\vx_2, \vy_1,\vy_2, \vz_1, \vz_2)\\
&~+\frac{\rho_1}{2}\left\|(\vx_1, \vy_1, \vy_2)-(\vx_1^{(k)}, \vy_1^{(k)}, \vy_2^{(k)})\right\|^2-\frac{\rho_2}{2}\left\|(\vx_2, \vz_1, \vz_2)-(\vx_2^{(0)}, \vz_1^{(0)}, \vz_2^{(0)})\right\|^2\\
p(x)\leftarrow&~P_2(\vx_1,\vy_1,\vy_2)\\
q(y)\leftarrow&~P_3(\vx_2,\vz_1, \vz_2),
\end{aligned}
\end{equation}
Additionally, when \eqref{eq:bilevel} satisfies Assumption~\ref{assump:dom} and $\rho_1$ and $\rho_2$ are set as in Algorithm~\ref{alg:first-order-new} (i.e., as in \eqref{eq:rho}),  \eqref{eq:minmax_prox} is an instance of \eqref{ea-prob} satisfying Assumption~\eqref{assump:barh} with
\begin{equation}
\label{eq:mapPrhoL}
\begin{aligned}
    L_{\nabla\h} \leftarrow&~L_{\nabla f} + 2\rho L_{\nabla \tilde{f}}+\max\{\rho_1,\rho_2\}, \\
    \sigma_x\leftarrow&~L_{\nabla f}+ 2\rho L_{\nabla \tilde{f}},\\
    \sigma_y\leftarrow&~\epsilon/(2D_2),
\end{aligned}
\end{equation}
where $D_2$ is defined in \eqref{eq:D2}. Since problem \eqref{eq:new_minmax} is also an instance of \eqref{eq:minmax} under the setting of \eqref{eq:maph} and \eqref{eq:mapLh}, the deterministic caes of Algorithm~\ref{alg:first-order-new} is exactly Algorithm~\ref{alg:first-order}
under the setting of \eqref{eq:maph} and \eqref{eq:mapLh}. Therefore, we can directly apply Theorem~\ref{thm:lu_thm} and Lemma~\ref{thm:connection} to prove Theorem~\ref{thm:minmax}. While doing so, we just need to set $\hat\epsilon=\epsilon^{1.5}$ in Theorem~\ref{thm:lu_thm} according to Line 3 of Algorithm~\ref{alg:first-order-new}. 

Let $f^*$, $D_1$, $D_2$ and $f_{\text{low}}$ be defined in \eqref{eq:old_minmax}, \eqref{eq:D1}, \eqref{eq:D2} and \eqref{eq:flow}.  Let
    \begin{align*}
        L_{\nabla P_1} = &~L_{\nabla f_1}+2\rho L_{\nabla\tilde{f}_1}, \\
        \alpha_P =&~ \min \{ 1, \sqrt{4\epsilon/(D_2 L_{\nabla P_1})} \},\\
        \delta_P =&~ (2+ \alpha_P^{-1})L_{\nabla P_1}D_1^2+ \max\{ \epsilon/D_2, \alpha_P L_{\nabla P_1}/4 \}D_2^2,\\
         K_P = &~\left\lceil 
        \begin{array}{ll}
        16(\max_{\vx_2}f(\vx_1^{(0)}, \vx_2, \vy_1^{(0)}, \vy_2^{(0)})+1- f_{\text{low}} + \epsilon D_2/4)L_{\nabla P_1} \epsilon^{-2}\\
        + 32 \epsilon^3(1+4(D_2^2L_{\nabla P_1}^2\epsilon^{-2})\epsilon^{-2} -1)    \end{array}\right\rceil_+\\
       C_P = &~  \frac{4\max\{ \frac{1}{2L_{\nabla P_1} }, \min \{ \frac{D_2}{\epsilon}, \frac{4}{\alpha_P L_{\nabla P_1}} \} \}\Big(\delta_P + 2\alpha_P^{-1} (f^* - f_{\text{low}} + \epsilon D_2/4 + L_{\nabla P_1}D_1^2)\Big)}{[(3L_{\nabla P_1} + \epsilon/(2D_2))^2/\min \{ L_{\nabla P_1}, \epsilon/(2D_2)\} + 3 L_{\nabla P_1} + \epsilon/(2D_2)]^{-2}\epsilon^3}\nonumber\\
       N_P =  &~ (\lceil 96\sqrt{2}(1+(24L_{\nabla P_1} + 4\epsilon/D_2)L_{\nabla P_1}^{-1}) \rceil + 2)\max\{ 2, \sqrt{D_2L_{\nabla P_1}\epsilon^{-1}} \}\\ &~~~~~~~~~~\times \Big((K_P+1) \times (\log(C_P))_{+} + K_P + 1 + 2K_P\log(K_P+1) \Big).
    \end{align*}
    



Note that $\max_{y}H(x^0,y)$ in Theorem~\ref{thm:lu_thm} becomes
\begin{align}
\nonumber
&\max_{\vx_2, \vz_1, \vz_2} P_\rho(\vx_1^{(0)},\vx_2, \vy_1^{(0)},\vy_2^{(0)}, \vz_1, \vz_2)\\\nonumber
=&\max_{\vx_2}f(\vx_1^{(0)}, \vx_2, \vy_1^{(0)}, \vy_2^{(0)})+
\rho\left(p(\vx_1^{(0)}, \vy_1^{(0)})-d(\vx_1^{(0)} ,\vy_2^{(0)})\right)\\\label{eq:upperbound_Hx0}
\leq&
\max_{\vx_2}f(\vx_1^{(0)}, \vx_2, \vy_1^{(0)}, \vy_2^{(0)})+1
\end{align}
because of the initial condition \eqref{eq:initial_condition} and the fact that $\rho=\frac{1}{\epsilon}$ by Algorithm~\ref{alg:first-order-new}. 
Recall that $\hat{\epsilon}$ is set to be $\epsilon^{1.5}$ in Algorithm~\ref{alg:first-order-new}.  Applying Theorem~\ref{thm:lu_thm} under this setting with $\hat{\epsilon}=\epsilon^{1.5}$, 
we can  show that Algorithm~\ref{alg:first-order-new} outputs an $\epsilon$-primal-dual stationary point $\left((\vx_1^{\epsilon}, \vy_1^{\epsilon}, \vy_2^{\epsilon}),(\vx_2^{\epsilon}, \vz_1^{\epsilon}, \vz_2^{\epsilon})\right)$ of~\eqref{eq:new_minmax} in oracle complexity no more than $N_P$ above. Since $L_{\nabla P_1}=O(\rho)=O(\epsilon^{-1})$, we have $K_P = O(\epsilon^{-3})$ and $N_P=\tilde O(\epsilon^{-4})$.

Moreover, \eqref{eq:HxK_bound} in Theorem~\ref{thm:lu_thm} also implies that $\left((\vx_1^{\epsilon}, \vy_1^{\epsilon}, \vy_2^{\epsilon}),(\vx_2^{\epsilon}, \vz_1^{\epsilon}, \vz_2^{\epsilon})\right)$  satisfies
\begin{align*}
    \max_{\vx_2, \vz_1, \vz_2}P_\rho(\vx_1^{\epsilon},\vx_2, \vy_1^{\epsilon},\vy_2^{\epsilon}, \vz_1, \vz_2) \leq& \max_{\vx_2, \vz_1, \vz_2} P_\rho(\vx_1^{(0)},\vx_2, \vy_1^{(0)},\vy_2^{(0)}, \vz_1, \vz_2)\\
    &+\epsilon D_2/4+2\epsilon^3(L_{\nabla P_1}^{-1}+4D_2^2 L_{\nabla P_1} \epsilon^{-2}).
\end{align*}
Combining the inequality above with \eqref{eq:upperbound_Hx0} gives
\begin{align*}
    &\rho^{-1}\left(\max_{\vx_2, \vz_1, \vz_2}P_\rho(\vx_1^{\epsilon},\vx_2, \vy_1^{\epsilon},\vy_2^{\epsilon}, \vz_1, \vz_2) 
    -f_{\text{low}}\right)\\
    \leq&  \rho^{-1}\left( \max_{\vx_2, \vz_1, \vz_2} P_\rho(\vx_1^{(0)},\vx_2, \vy_1^{(0)},\vy_2^{(0)}, \vz_1, \vz_2)
    -f_{\text{low}}+\epsilon D_2/4+2\epsilon^3(L_{\nabla P_1}^{-1}+4D_2^2 L_{\nabla P_1} \epsilon^{-2})\right)\\
    \leq&  \rho^{-1}\left( \max_{\vx_2}f(\vx_1^{(0)}, \vx_2, \vy_1^{(0)}, \vy_2^{(0)})+1
    -f_{\text{low}}+\epsilon D_2/4+2\epsilon^3(L_{\nabla P_1}^{-1}+4D_2^2 L_{\nabla P_1} \epsilon^{-2})\right)\\
    \leq& O(\epsilon),
\end{align*}
where the last inequality is because $L_{\nabla P_1}=O(\rho)=O(\epsilon^{-1})$. By Lemma~\ref{thm:connection}, $(\vx_1^{\epsilon}, \vy_1^{\epsilon}, \vy_2^{\epsilon})$ is a lifted $O(\epsilon)$-KKT point of \eqref{eq:bilevel}.
\end{proof}

\paragraph{Dependence of total oracle complexity on problem parameters.}
We further examine the dependence of the deterministic first-order
oracle complexity on the problem parameters.
Let
\[
\Delta_0 :=
\max_{\vx_2}
f(\vx_1^{(0)},\vx_2,\vy_1^{(0)},\vy_2^{(0)})
-f_{\mathrm{low}}
\]
denote the initial objective gap relative to the lower bound
$f_{\mathrm{low}}$.
Recall that $\rho=\epsilon^{-1}$ and $L_{\nabla P_1}
= L_{\nabla f_1}
+2\epsilon^{-1}L_{\nabla\tilde{f}_1}$.
By the definitions of $K_P$ and $N_P$, 
we have
\begin{align*}
K_P
= O\left(
    \frac{(\Delta_0+\epsilon D_2)(L_{\nabla f_1}
+\epsilon^{-1}L_{\nabla\tilde{f}_1})}{\epsilon^2}
    \right),\quad
N_P
= \tilde{O}\left(
    \sqrt{\frac{D_2(L_{\nabla f_1}
+\epsilon^{-1}L_{\nabla\tilde{f}_1})}{\epsilon}}K_P
    \right).
\end{align*}
These bounds explicitly characterize the dependence on the
smoothness constants $L_{\nabla f_1}$ and
$L_{\nabla\tilde{f}_1}$, as well as the diameter $D_2$. 

Note that the above complexity bounds guarantee a lifted
$O(\epsilon)$-KKT point of~\eqref{eq:bilevel} satisfying all
conditions in Definition~\ref{def:epsilon_KKT}, except that the
last condition is guaranteed only up to a constant factor, namely,
\[
p(\vx_1,\vy_1)-d(\vx_1,\vy_2)\leq O(\epsilon).
\]
To obtain the explicit bound
$p(\vx_1,\vy_1)-d(\vx_1,\vy_2)\leq\epsilon$,
it suffices to choose $\rho=c_\rho/\epsilon$ with a
sufficiently large constant $c_\rho>0$.
In particular, the last inequality in the proof of
Theorem~\ref{thm:minmax} shows that
$c_\rho=\Theta(1+\Delta_0)$ is sufficient.
The corresponding oracle complexity bounds can then be
obtained by replacing $\rho=1/\epsilon$ to $\rho=c_\rho/\epsilon$ and
retaining the dependence on $c_\rho$ in
$L_{\nabla P_1}$ and all other quantities involving $\rho$.

\section{Proof of technical lemmas in Section~\ref{sec:cons_minmax}
}
\label{sec:tech_proof}
\subsection{Proof of Lemma~\ref{thm:con_lagrangian}}
\label{sec:proof_lemma1}
\begin{proof}
Let's consider the convex program 
\begin{align}
    \min_{\vz_1}\bar{f}(\vx_1, \vz_1) \qquad \text{s.t.} \quad \bar{\vg}(\vx_1, \vz_1) \leq 0. \label{eq:lowerlevel}
\end{align}
By Assumption~\ref{assume:constrained}.2, $\bar{f}(\vx_1, \vz_1)$ is convex in $\vz_1$, $\bar{g}_i(\vx_1, \vz_1)$ is convex in $\vz_1$. By Assumption~\ref{assume:constrained}.4, Slater holds with margin $G>0$. Then the strong duality holds and there exists a Lagrange multiplier $\vlam^* \in \mathbb{R}_+^l$ such that $(\vz_1^*, \vlam^*)$ is a saddle point of the Lagrangian $\bar{f}(\vx_1, \vz_1) + \vlam^T \bar{\vg}(\vx_1, \vz_1)$. That is
\begin{align} \label{eq:lower-tildef-saddle}
    \bar{f}(\vx_1, \vz_1^*) + \vlam^T \bar{\vg}(\vx_1, \vz_1^*) \leq \bar{f}(\vx_1, \vz_1^*) + \vlam^{*T} \bar{\vg}(\vx_1, \vz_1^*)\leq \bar{f}(\vx_1, \vz_1) + \vlam^{*T} \bar{\vg}(\vx_1, \vz_1)
\end{align}
for any $\vx_1\in\mathcal{X}_1$, $\vz_1\in\mathcal{Y}_1$ and $\vlam\in\mathbb{R}_+^l$. Due to complementary slackness,
we have 
$$\bar{f}(\vx_1, \vz_1^*)
= \bar{f}(\vx_1, \vz_1^*) + \vlam^{*T} \bar{\vg}(\vx_1, \vz_1^*)
$$

By Assumption~\ref{assume:constrained}.4, given any $\vx_1\in\mathcal{X}_1$, there exists  $\hat{\vz}_1 \in \mathcal{Y}_1$ satisfying $\bar{g}_i(\vx_1, \hat{\vz}_1) \leq -G$. Setting $\vz_1=\hat{\vz}_1$ in the second inequality of~\eqref{eq:lower-tildef-saddle}, we have
\begin{gather*}
    \bar{f}(\vx_1, \vz_1^*)
    \leq \bar{f}(\vx_1, \hat{\vz}_1) + \vlam^{*T} \bar{\vg}(\vx_1, \hat{\vz}_1)= \bar{f}(\vx_1, \hat{\vz}_1) + \sum_{i=1}^l \lambda_i^* \bar{g}_i(\vx_1, \hat{\vz}_1) \leq \bar{f}(\vx_1, \hat{\vz}_1) -G \sum_{i=1}^l \lambda_i^*\\
    \Rightarrow ~ \sum_{i=1}^l \lambda_i^* \leq \frac{\bar{f}(\vx_1, \hat{\vz}_1)-\bar{f}(\vx_1,
    \vz_1^*)}{G}
\end{gather*}
Now, using Assumption~\ref{assume:constrained}.2 and \eqref{eq:cons_y1}, we have $\bar{f}(\vx_1, \hat{\vz}_1) - \bar{f}(\vx_1,\vz_1^*) \leq L_{\bar{f}}D_{\mathcal{Y}_1}$.
Therefore,
\begin{align*}
    \sum_{i=1}^l \lambda_i^* \leq \frac{L_{\bar{f}}D_{\mathcal{Y}_1}}{G} < B.
\end{align*}
Since $\vlam^* \geq 0$, this implies the componentwise $\lambda^{*}_i < B, ~\forall i$. Hence,  $\vlam^* \in [0, B)^l$.

We now prove that 
a solution to \eqref{eq:lowerlevel} corresponds to a saddle point of $\tilde{f}(\vx_1,\vz_1,\vz_2)$  defined in~\eqref{eq:tildef_cons}.

``$\Rightarrow$'': Let $\vy_1$ be an optimal solution to the lower-level problem~\eqref{eq:lowerlevel}. As proven above, there exists an optimal dual multiplier $\vlam^{*} \ge 0$ such that $(\vy_1,\vlam^{*})$ forms a global saddle point for the standard Lagrangian~\eqref{eq:lower-tildef-saddle} 
and the optimal multiplier satisfies the bound $\vlam^{*} \in [0,B)^{l}$.

Let $\vy_{2} = \vlam^{*}$. We must prove that $(\vy_1,\vy_{2})$ is a saddle point of the specific minimax objective $\tilde{f}(\vx_1,\vz_1,\vz_2) = \overline{f}(\vx_1,\vz_1) + \vz_2^{T}\bar{\vg}(\vx_1,\vz_1) - \textbf{1}_{[0,B]^{l}}(\vz_2)$.
\begin{enumerate}
    \item \textbf{Maximization over the dual variable ($\vz_2$):} From the first inequality of~\eqref{eq:lower-tildef-saddle}, we have $\overline{f}(\vx_1,\vy_1) + \vz_2^{T}\bar{\vg}(\vx_1,\vy_1) \le \overline{f}(\vx_1,\vy_1) + (\vlam^{*})^{T}\bar{\vg}(\vx_1,\vy_1)$ for all $\vz_2 \ge 0$. This inequality trivially holds for any $\vz_2$ restricted to the subset $[0,B]^{l}$. For any $\vz_2 \notin [0,B]^{l}$, the indicator function $-\textbf{1}_{[0,B]^{l}}(\vz_2)$ evaluates to $-\infty$, making the inequality $\tilde{f}(\vx_1,\vy_1,\
    \vz_2) \le \tilde{f}(\vx_1,\vy_1,\vy_2)$ strictly and trivially true.
    \item \textbf{Minimization over the primal variable ($\vz_1$):} From the second inequality of~\eqref{eq:lower-tildef-saddle}, we have $\overline{f}(\vx_1,\vy_1) + (\vlam^{*})^{T}\bar{\vg}(\vx_1,\vy_1) \le \overline{f}(\vx_1,\vz_1) + (\vlam^{*})^{T}\bar{\vg}(\vx_1,\vz_1)$ for all $\vz_1 \in \mathcal{Y}_{1}$. Since $\vy_2 = \vlam^{*} \in [0,B]^{l}$, the indicator penalty is zero ($-\textbf{1}_{[0,B]^{l}}(\vy_2) = 0$). Adding this zero penalty to both sides yields exactly $\tilde{f}(\vx_1,\vy_1,\vy_2) \le \tilde{f}(\vx_1,\vz_1,\vy_2)$ for all $\vz_1 \in \mathcal{Y}_{1}$.
\end{enumerate}
Since both inequalities hold for all $\vz_1 \in \mathcal{Y}_{1}$ and all $\vz_2$, the pair $(\vy_1,\vy_2)$ is exactly a saddle point of $\min_{\vz_1} \max_{\vz_2} \tilde{f}(\vx_1,\vz_1,\vz_2)$.

``$\Leftarrow$'': Conversely, suppose there exists $\vy_2 \in [0,B]^{l}$ such that $(\vy_1,\vy_2)$ is a saddle point of $\min_{\vz_1} \max_{\vz_2} \tilde{f}(\vx_1,\vz_1,\vz_2)$, and equivalently a saddle point of $\min_{\vz_1} \max_{\vz_2 \in [0,B]^{l}} \{ \overline{f}(\vx_1,\vz_1) + \vz_2^{T}\bar{\vg}(\vx_1,\vz_1) \}$. 
The saddle point definition imposes two inequalities:
\begin{enumerate}
    \item \textbf{Maximization over the dual variable ($\vz_2$):} We have $\overline{f}(\vx_1,\vy_1) + \vz_2^{T}\bar{\vg}(\vx_1,\vy_1) \le \overline{f}(\vx_1,\vy_1) + \vy_2^{T}\bar{\vg}(\vx_1,\vy_1)$ for all $\vz_2 \in [0,B]^{l}$. For this to hold, $\vy_2$ must satisfy a component-wise maximization rule: if $\bar{\vg}_{i}(\vx_1,\vy_1) > 0$, it forces $y_{2,i} = B$; if $\bar{\vg}_{i}(\vx_1,\vy_1) < 0$, it forces $y_{2,i} = 0$.
    \item \textbf{Minimization over the primal variable ($\vz_1$):} We have $\overline{f}(\vx_1,\vy_1) + \vy_2^{T}\bar{\vg}(\vx_1,\vy_1) \le \overline{f}(\vx_1,\vz_1) + \vy_2^{T}\bar{\vg}(\vx_1,\vz_1)$ for all $\vz_1 \in \mathcal{Y}_{1}$.
\end{enumerate}

To prove primal feasibility, suppose for contradiction that $\bar{\vg}_{j}(\vx_1,\vy_1) > 0$ for some index $j$. By the maximization rule, this forces $y_{2,j} = B$, which implies $\vy_2^{T}\bar{\vg}(\vx_1,\vy_1) \ge 0$. Evaluating the minimization inequality at the Slater point $\hat{\vz}_1$ (where $\bar{\vg}_{i}(\vx_1,\hat{\vz}_1) \le -G$ for all $i$) yields:
$$ \overline{f}(\vx_1,\vy_1) + \vy_2^{T}\bar{\vg}(\vx_1,\vy_1) \le \overline{f}(\vx_1,\hat{\vz}_1) + \vy_2^{T}\bar{\vg}(\vx_1,\hat{\vz}_1) $$
Rearranging and utilizing the Lipschitz bound $\overline{f}(\vx_1,\hat{\vz}_1) - \overline{f}(\vx_1,\vy_1) \le L_{\overline{f}}D_{\mathcal{Y}_{1}}$, we obtain:
$$ -\vy_2^{T}\bar{\vg}(\vx_1,\hat{\vz}_1) \le \overline{f}(\vx_1,\hat{\vz}_1) - \overline{f}(\vx_1,\vy_1) - \vy_2^{T}\bar{\vg}(\vx_1,\vy_1) \le L_{\overline{f}}D_{\mathcal{Y}_{1}} $$
Since $y_{2,j} = B$ and $\vy_2 \ge 0$, we can strictly bound the left side: $-\vy_2^{T}\bar{\vg}(\vx_1,\hat{\vz}_1) \ge G \sum_{i=1}^{l} y_{2,i} \ge G \cdot B$. 
This gives $G \cdot B \le L_{\overline{f}}D_{\mathcal{Y}_{1}}$, or $B \le \frac{L_{\overline{f}}D_{\mathcal{Y}_{1}}}{G}$, which strictly contradicts our chosen bound $B =  \frac{2 L_{\overline{f}}D_{\mathcal{Y}_{1}}}{G}$.

Therefore, no constraint can be violated, guaranteeing {primal feasibility} ($\bar{\vg}(\vx_1,\vy_1) \le 0$). Because all constraints are feasible, the maximization rule forces $y_{2,i} = 0$ whenever $\bar{\vg}_{i}(\vx_1,\vy_1) < 0$, which establishes exact {complementary slackness} ($\vy_2^{T}\bar{\vg}(\vx_1,\vy_1) = 0$).

Substituting this complementary slackness equality into the minimization condition simplifies it to $\overline{f}(\vx_1,\vy_1) \le \overline{f}(\vx_1,\vz_1)$ for all feasible $\vz_1$ (where $\vy_2^{T}\bar{\vg}(\vx_1,\vz_1) \le 0$). Hence, $\vy_1$ is an optimal solution to the lower-level problem.
\end{proof}

\subsection{Proof of Lemma~\ref{thm:e-kkt}}
\label{sec:proof_of_cons_minmax}
\begin{proof}[Proof of Lemma~\ref{thm:e-kkt}]
Note that \eqref{eq:cons_minmax} is an instance of \eqref{eq:bilevel} without $\vx_2$. Suppose $(\vx_1, \vy_1, \vy_2)\in\mathcal{X}_1\times\mathcal{Y}_1\times [0,B]^l$ is a lifted $\epsilon$-KKT solution of problem \eqref{eq:cons_minmax} under Definition~\ref{def:epsilon_KKT}. There must exist $\rho\geq0$ and $(\vz_1, \vz_2)\in \mathcal{Y}_1\times [0,B]^l$ such that the following inequalities hold.
    \begin{align}
        &~\text{dist}\left(\mathbf{0} ,\partial_{(\vx_1,\vy_1,\vy_2)}\left[   f(\vx_1, \vy_1) + \rho\left(
        \begin{array}{c}
     \bar{f}(\vx_1, \vy_1) + \vz_2^\top\bar{\vg}(\vx_1, \vy_1)-\mathbf{1}_{[0,B]^l}(\vz_2)\\
        - \bar{f}(\vx_1, \vz_1) - \vy_2^\top\bar{\vg}(\vx_1, \vz_1)+\mathbf{1}_{[0,B]^l}(\vy_2) 
        \end{array}
        \right)
        \right]\right)\leq \epsilon, \label{eq:dist_x1y1y2}\\
     &~\text{dist}\left(\mathbf{0}, \rho\partial_{\vz_1}\left[\bar{f}(\vx_1, \vz_1) + \vy_2^\top\bar{\vg}(\vx_1, \vz_1)\right]\right)\leq \epsilon,\label{eq:dist_z1}\\
        &~\text{dist}\left(\mathbf{0}, \rho\partial_{\vz_2}\left[-\vz_2^\top\bar{\vg}(\vx_1, \vy_1) + \mathbf{1}_{[0,B]^l}(\vz_2)\right]\right)\leq \epsilon, \label{eq:dist_z2}\\
        &~ \max_{\vz_2 \in [0,B]^l} \left\{\bar{f}(\vx_1, \vy_1) + \vz_2^T\bar{\vg}(\vx_1, \vy_1) \right\}- \min_{\vz_1}\left\{\bar{f}(\vx_1, \vz_1) + \vy_2^\top\bar{\vg}(\vx_1, \vz_1)\right\} \leq \epsilon.\label{eq:epsilon_gap}
    \end{align}

Let 
\begin{gather*}
\bar{\vlam}=\vy_2~~\text{and}~~\vlam=\rho\vz_2.
\end{gather*}
Then \eqref{eq:dist_z1} can be written as 
\begin{align}
\label{eq:stationary_constrain_lower}
\text{dist}\left(\mathbf{0}, \rho\partial_{\vz_1}\left[\bar{f}(\vx_1, \vz_1) + \bar{\vlam}^T\bar{g}(\vx_1, \vz_1)\right]  \right)\leq\epsilon.
\end{align}
It is  easy to see that \eqref{eq:dist_x1y1y2} implies 
\begin{align}
\label{eq:stationary_constrain}
\text{dist}\left(\mathbf{0},\partial_{(\vx_1,\vy_1)} \left[f(\vx_1, \vy_1) +\rho\big(\bar{f}(\vx_1, \vy_1)-\bar{f}(\vx_1, \vz_1) - \bar{\vlam}^T\bar{g}(\vx_1, \vz_1)\big)+\vlam^T\bar{g}(\vx_1, \vy_1)\right]\right)\leq\epsilon
\end{align}
and
\begin{align}
\text{dist}\left(\mathbf{0}, \rho\partial_{\vy_2} \left[- \vy_2^\top\bar{\vg}(\vx_1, \vz_1)+\mathbf{1}_{[0,B]^l}(\vy_2) \right]\right)\leq \epsilon \label{eq:dist_y2}.
\end{align}
Equation \eqref{eq:dist_y2} is equivalent to 
\begin{align}
\text{dist}\left(\rho\bar{\vg}(\vx_1, \vz_1), \partial_{\vy_2} \mathbf{1}_{[0,B]^l}(\vy_2)\right)\leq \epsilon \label{eq:dist_y2_new}.
\end{align}
Let $\vy_2=(y_{2,1},\dots,y_{2,l})^\top$ and we can split the indexes $\{1,\dots,l\}$ into three subsets:
$$
I_1:=\{i|y_{2,i}=0\},\quad I_2:=\{i|y_{2,i}\in(0,B)\},\quad I_3:=\{i|y_{2,i}=B\}.
$$
Let $\vy_{2,I_1}$, $\vy_{2,I_2}$ and $\vy_{2,I_3}$ be the sub-vectors of $\vy_2$ where the indexes of the components are in $I_1$, $I_2$ and $I_3$, respectively. We define $\vg_{I_1}(\vx_1, \vz_1)$, $\vg_{I_2}(\vx_1, \vz_1)$ and $\vg_{I_3}(\vx_1, \vz_1)$ in a similar way. Because
\begin{align*}
\partial \mathbf{1}_{[0,B]}(y)
=
\begin{cases}
(-\infty,0] & \text{if } y=0,\\
\{0\} & \text{if } y\in(0,B),\\
[0,+\infty) & \text{if } y=B,
\end{cases}
\end{align*}
\eqref{eq:dist_y2_new} implies 
\begin{align}
\label{eq:gxz_norm_old}
\rho\sqrt{\|[\vg_{I_1}(\vx_1, \vz_1)]_+\|^2+\|\vg_{I_2}(\vx_1, \vz_1)\|^2+\|[\vg_{I_3}(\vx_1, \vz_1)]_-\|^2}\leq \epsilon,
\end{align}
where $[\cdot]_+:=\max\{\cdot,0\}$ and $[\cdot]_-:=\min\{\cdot,0\}$. Therefore, 
\begin{align}
\nonumber
-\rho \vy_2^\top\bar{\vg}(\vx_1, \vz_1)
&=~-\rho \vy_{2,I_1}^\top\bar{\vg}_{2,I_1}(\vx_1, \vz_1)
-\rho \vy_{2,I_2}^\top\bar{\vg}_{2,I_2}(\vx_1, \vz_1)
-\rho \vy_{2,I_3}^\top\bar{\vg}_{2,I_3}(\vx_1, \vz_1)\\\nonumber
&\leq~-\rho \vy_{2,I_1}^\top[\bar{\vg}_{2,I_1}(\vx_1, \vz_1)]_+
-\rho \vy_{2,I_2}^\top\bar{\vg}_{2,I_2}(\vx_1, \vz_1)
-\rho \vy_{2,I_3}^\top[\bar{\vg}_{2,I_3}(\vx_1, \vz_1)]_-\\\nonumber
&\leq~ \rho\|\vy_2\|\sqrt{\|[\vg_{I_1}(\vx_1, \vz_1)]_+\|^2+\|\vg_{I_2}(\vx_1, \vz_1)\|^2+\|[\vg_{I_3}(\vx_1, \vz_1)]_-\|^2}\\\label{eq:boundyg}
&\leq~ \epsilon\|\vy_2\|\leq \epsilon\|\vy_2\|_1.
\end{align}

We next prove that $I_3=\emptyset$. Let $\hat{\vy}_1$ be the solution that satisfies Assumption~\ref{def:cons_kkt}.4, so $\bar{g}_i(\vx_1, \hat{\vy}_1)\leq -G$ for $i=1,\dots,l$. By \eqref{eq:dist_z1}, we have 
$$
\rho (\partial_{\vz_1}\bar{f}(\vx_1, \vz_1) + \partial_{\vz_1}\bar{\vg}(\vx_1, \vz_1)\vy_2)^\top(\vz_1-\hat{\vy}_1) \leq D_{\mathcal{Y}_1}\epsilon,
$$
or equivalently, 
$$
 \rho \partial_{\vz_1} \bar{f}(\vx_1, \vz_1)^\top (\vz_1-\hat{\vy}_1) \leq \rho \left\langle\nabla_{\vz_1}\bar{\vg}(\vx_1, \vz_1)\vy_2,\hat{\vy}_1-\vz_1\right\rangle + D_{\mathcal{Y}_1}\epsilon,
$$
which implies 
    \begin{align*}
    -\rho L_{\bar{f}}D_{\mathcal{Y}_1}\leq &~ \rho\left(\bar{f}(\vx_1, \vz_1)-\bar{f}(\vx_1, \hat{\vy}_1)\right)\leq
    \rho \left\langle \partial_{\vz_1} \bar{f}(\vx_1, \vz_1),  \vz_1-\hat{\vy}_1 \right\rangle \\
    \leq &~\rho \left\langle\nabla_{\vz_1}\bar{\vg}(\vx_1, \vz_1)\vy_2,\hat{\vy}_1-\vz_1\right\rangle + D_{\mathcal{Y}_1}\epsilon\\
     \leq &~ \rho \vy_2^\top (\bar{\vg}(\vx_1, \hat{\vy}_1) - \bar{\vg}(\vx_1, \vz_1))+ D_{\mathcal{Y}_1}\epsilon\\
     \leq &~ -\rho G\|\vy_2\|_1+ \epsilon\|\vy_2\|_1+ D_{\mathcal{Y}_1}\epsilon,
    \end{align*} 
where the first inequality is because of the $L_{\bar{f}}$-Lipschitz continuity of $\bar{f}$, the second is by the convexity of $\bar{f}$ in its second argument, the fourth is by the convexity of $\bar{g}$ in its second argument, and the last inequality is because of \eqref{eq:boundyg} and the fact that $\bar{g}_i(\vx_1, \hat{\vy}_1)\leq -G$ for $i=1,\dots,l$. 

Suppose $I_3\neq\emptyset$ so there is one coordinate of $\vy_2$ equals $B$. We have $\|\vy_2\|_1\geq B$. Since $\epsilon\leq \rho G $, the inequality above implies 
$$
B(\rho G -\epsilon)\leq D_{\mathcal{Y}_1}\epsilon+ \rho L_{\bar{f}}D_{\mathcal{Y}_1}.
$$
Since $\epsilon \leq \min\{\frac{\rho L_{\bar{f}}}{4},\frac{\rho G }{4}\}$, by \eqref{eq:cons_y1}, it holds that 
$$
\frac{ 2L_{\bar{f}}D_{\mathcal{Y}_1}}{G }
= B \leq \frac{D_{\mathcal{Y}_1}\epsilon+ \rho L_{\bar{f}}D_{\mathcal{Y}_1}}{(\rho G -\epsilon)}\leq \frac{ 5 L_{\bar{f}}D_{\mathcal{Y}_1}}{3 G }.
$$
This contradiction means $I_3=\emptyset$.

Therefore, by \eqref{eq:gxz_norm_old}, we have 
\begin{align}
\nonumber
|\rho \vy_2^\top\bar{\vg}(\vx_1, \vz_1)|&=~|\rho \vy_{2,I_1}^\top\bar{\vg}_{2,I_1}(\vx_1, \vz_1)
+\rho \vy_{2,I_2}^\top\bar{\vg}_{2,I_2}(\vx_1, \vz_1)
+\rho \vy_{2,I_3}^\top\bar{\vg}_{2,I_3}(\vx_1, \vz_1)|\\\nonumber
&=~|\rho \vy_{2,I_2}^\top\bar{\vg}_{2,I_2}(\vx_1, \vz_1)|\\\nonumber
&\leq~ \rho\|\vy_2\|\|\vg_{I_2}(\vx_1, \vz_1)\|\\\nonumber
&\leq~ \epsilon \sqrt{l}B,
\end{align}
which implies 
\begin{align}
\label{eq:rhoygxz}
|\vy_2^\top\bar{\vg}(\vx_1, \vz_1)|\leq \frac{\epsilon \sqrt{l}B}{\rho}.
\end{align}
Moreover, \eqref{eq:gxz_norm_old} and the fact that $I_3=\emptyset$ imply that 
\begin{align}
\label{eq:gxz_norm}
|[\bar{\vg}(\vx_1,\vz_1)]_+\|\leq \sqrt{\|[\vg_{I_1}(\vx_1, \vz_1)]_+\|^2+\|\vg_{I_2}(\vx_1, \vz_1)\|^2}\leq \frac{\epsilon}{\rho}.
\end{align}

By the same argument as the one above, we can derive from \eqref{eq:dist_z2} that
\begin{align}
\label{eq:rhoygxy}
|\vz_2^\top\bar{\vg}(\vx_1, \vy_1)|\leq \frac{\epsilon \sqrt{l}B}{\rho}.
\end{align}
and
\begin{align}
\label{eq:gxy_norm}
|[\bar{\vg}(\vx_1,\vy_1)]_+\|\leq \frac{\epsilon}{\rho}.
\end{align}

Finally, let $\vz_1^*\in \argmin_{\vz_1'}  \left\{\bar{f}(\vx_1, \vz_1')|\bar{\vg}(\vx_1, \vz_1')\leq 0\right\}$.  We then have 
\begin{align}
\nonumber
&~\bar{f}(\vx_1, \vy_1) - \bar{f}^*(\vx_1)\\\nonumber
\leq&~ \left\{\bar{f}(\vx_1, \vy_1) + (\mathbf{0})^\top\bar{\vg}(\vx_1, \vy_1)\right\}
-\left\{\bar{f}(\vx_1, \vz_1^*) + \vy_2^\top\bar{\vg}(\vx_1, \vz_1^*)\right\}\\\nonumber
\leq&~ \max_{\vz_2' \in [0, B]} \left\{\bar{f}(\vx_1, \vy_1) + (\vz_2')^\top\bar{\vg}(\vx_1, \vy_1)\right\}
-\min_{\vz_1'}\left\{\bar{f}(\vx_1, \vz_1') + \vy_2^\top\bar{\vg}(\vx_1, \vz_1')\right\}\\\label{eq:ffstar}
\leq&~ \epsilon
\end{align}
where the last inequality is due to \eqref{eq:epsilon_gap}.

The conclusion is thus proved as we have proved that \eqref{eq:stationary_constrain}, \eqref{eq:stationary_constrain_lower}, \eqref{eq:rhoygxz},  \eqref{eq:gxz_norm},   \eqref{eq:rhoygxy},  \eqref{eq:gxy_norm}, and \eqref{eq:ffstar} hold.
\end{proof}

\section{Proof of Theorem~\ref{thm:constrained}}\label{sec:proof_thmcons}


\begin{proof}
Since Theorem~\ref{thm:constrained} is an application of Theorem~\ref{thm:minmax} to a particular form of lower level objective, it remains to 
derive the Lipschitz constant $L_{\nabla P_1}$ for the penalty function. 
Based on the reformulation, we decompose $P_1$ as follows:
$$ P_1 = f_1(\mathbf{x}_1, \mathbf{y}_1) + \rho \underbrace{\big( \overline{f}_1(\mathbf{x}_1, \mathbf{y}_1) + \mathbf{z}_2^\top \bar{\vg}(\mathbf{x}_1, \mathbf{y}_1) \big)}_{=:H_1(\mathbf{x}_1, \mathbf{y}_1, \mathbf{z}_2)} - \rho \underbrace{\big( \overline{f}_1(\mathbf{x}_1, \mathbf{z}_1) + \mathbf{y}_2^\top \bar{\vg}(\mathbf{x}_1, \mathbf{z}_1) \big)}_{=:H_2(\mathbf{x}_1, \mathbf{z}_1, \mathbf{y}_2)} $$

Let $\mathbf{u} = (\mathbf{x}_1, \mathbf{y}_1)$. The gradient of $H_1(\mathbf{u}, \mathbf{z}_2)$ is $\nabla H_1 = [\nabla \overline{f}_1(\mathbf{u}) + J_{\bar{\vg}}(\mathbf{u})^\top \mathbf{z}_2, \, \bar{\vg}(\mathbf{u})]^\top$.
For any two points $(\mathbf{u}, \mathbf{z}_2)$ and $(\mathbf{u}', \mathbf{z}_2')$, we bound the difference in the $\mathbf{u}$-gradients:
\begin{align*}
& \|\nabla_{\mathbf{u}} H_1(\mathbf{u}, \mathbf{z}_2) - \nabla_{\mathbf{u}} H_1(\mathbf{u}', \mathbf{z}_2')\| \\
\le & \|\nabla \overline{f}_1(\mathbf{u}) - \nabla \overline{f}_1(\mathbf{u}')\| + \|J_{\bar{\vg}}(\mathbf{u})^\top \mathbf{z}_2 - J_{\bar{\vg}}(\mathbf{u}')^\top \mathbf{z}_2'\| \\
\le & L_{\nabla \overline{f}_1} \|\mathbf{u} - \mathbf{u}'\| + \|(J_{\bar{\vg}}(\mathbf{u}) - J_{\bar{\vg}}(\mathbf{u}'))^\top \mathbf{z}_2\| + \|J_{\bar{\vg}}(\mathbf{u}')^\top (\mathbf{z}_2 - \mathbf{z}_2')\| \\
\le & (L_{\nabla \overline{f}_1} + B\sqrt{l} L_{\nabla \bar{\vg}}) \|\mathbf{u} - \mathbf{u}'\| + L_{\bar{\vg}} \|\mathbf{z}_2 - \mathbf{z}_2'\|
\end{align*}
where we utilized the constraint bounds $\|\mathbf{z}_2\| \le B\sqrt{l}$ (since $\mathbf{z}_2 \in [0,B]^l$), the $L_{\nabla \bar{\vg}}$-smoothness and the $L_{\bar{\vg}}$-Lipschitz continuity ($\|J_{\bar{\vg}}(\mathbf{u})\|\leq L_{\bar{\vg}}$) of the constraint mapping $\bar{\vg}$ with respect to operator norm.

For the dual variable gradient, we have:
$$ \|\nabla_{\mathbf{z}_2} H_1(\mathbf{u}, \mathbf{z}_2) - \nabla_{\mathbf{z}_2} H_1(\mathbf{u}', \mathbf{z}_2')\| = \|\bar{\vg}(\mathbf{u}) - \bar{\vg}(\mathbf{u}')\| \le L_{\bar{\vg}} \|\mathbf{u} - \mathbf{u}'\| $$

Since 
\begin{align*}
&\| \nabla H_1(\mathbf{u}, \mathbf{z}_2) - \nabla H_1(\mathbf{u}', \mathbf{z}_2') \| \\
=& \sqrt{\|\nabla_{\mathbf{u}} H_1 (\mathbf{u}, \mathbf{z}_2) - \nabla_{\mathbf{u}} H_1(\mathbf{u}', \mathbf{z}_2')\|^2 + \|\nabla_{\mathbf{z}_2} H_1 (\mathbf{u}, \mathbf{z}_2) - \nabla_{\mathbf{z}_2} H_1(\mathbf{u}', \mathbf{z}_2')\|^2} \\
\le& \|\nabla_{\mathbf{u}} H_1 (\mathbf{u}, \mathbf{z}_2) - \nabla_{\mathbf{u}} H_1(\mathbf{u}', \mathbf{z}_2')\| + \|\nabla_{\mathbf{z}_2} H_1 (\mathbf{u}, \mathbf{z}_2) - \nabla_{\mathbf{z}_2} H_1(\mathbf{u}', \mathbf{z}_2')\|,
\end{align*}
summing the above bounds provides the total Lipschitz constant for $H_1$:
$$ 
L_{\nabla H_1} \le L_{\nabla \overline{f}_1} + B\sqrt{l} L_{\nabla \bar{\vg}} + 2L_{\bar{\vg}} .
$$
By structural symmetry, because the dual variable $\mathbf{y}_2$ is constrained to the exact same box $[0,B]^l$, the gradient of $H_2(\mathbf{x}_1, \mathbf{z}_1, \mathbf{y}_2)$ shares is $L_{\nabla H_2}$-Lipschitz continuous with $L_{\nabla H_2}=L_{\nabla H_1}$. 

Combining $L_{\nabla H_1}$ and $L_{\nabla H_2}$ with the $L_{\nabla f_1}$-smoothness of the upper-level objective $f_1$ yields the final overall bound:
$$ L_{\nabla P_1} \le L_{\nabla f_1} + 2\rho(L_{\nabla \overline{f}_1} + B\sqrt{l}L_{\nabla \bar{\vg}} + 2L_{\bar{\vg}}). $$
We define the following quantities, which are all finite according to Assumption~\ref{assume:constrained}.
\begin{align} 
 \label{eq:cons_D1}
   D_1 := & \max\{\|(\vx_1, \vy_1, \vy_2) - (\vx_1', \vy_1', \vy_2')\| \;|\; (\vx_1, \vy_1, \vy_2), (\vx_1', \vy_1', \vy_2') \in \mathcal{X}_1  \times \mathcal{Y}_1 \times [0,B]^l\},\\   \label{eq:cons_D2}
   D_2 := &\max\{\|(\vz_1,\vz_2) - (\vz_1',\vz_2')\| \;|\; (\vz_1,\vz_2), (\vz_1',\vz_2') \in \mathcal{Y}_1\times [0,B]^l\}, \\
    f_{\text{low}}:=& \min \{f(\vx_1, \vy_1) \;|\; (\vx_1, \vy_1)\in \mathcal{X}_1 \times \mathcal{Y}_1 \},\label{eq:cons_flow}.
\end{align}
Then we can apply Theorem~\ref{thm:minmax} on \eqref{eq:cons_minmax} and \eqref{eq:tildef_cons}. Let $D_1$, $D_2$ and $f_{\text{low}}$ be defined in \eqref{eq:cons_D1}, \eqref{eq:cons_D2} and \eqref{eq:cons_flow}. Let
    \begin{align*}
    f^*= &~\min_{\vx_1, \vy_1} \left\{f(\vx_1, \vy_1)+\rho\left[\bar{f}(\vx_1, \vy_1) + \max_{\vz_2 \in [0, B]}\vz_2^\top\bar{\vg}(\vx_1, \vy_1)
-\bar{f}^*(\vx_1)\right]\right\},\\
        L_{\nabla P_1} = &~L_{\nabla f_1}+2\rho(L_{\nabla\bar{f_1}}+B\sqrt{l} L_{\nabla \bar{g}} + 2L_{\bar{g}}), \\
        \alpha_c =&~ \min \{ 1, \sqrt{4\epsilon/(D_2 L_{\nabla P_1})} \},\\
        \delta_c =&~ (2+ \alpha_c^{-1})L_{\nabla P_1}D_1^2+ \max\{ \epsilon/D_2, \alpha_c L_{\nabla P_1}/4 \}D_2^2,\\
         K_c = &~\left\lceil 
        \begin{array}{ll}
        16(f(\vx_1^{(0)}, \vy_1^{(0)}) +1 - f_{\text{low}} + \epsilon D_2/4)L_{\nabla P_1} \epsilon^{-2}\\
        + 32 \epsilon^3(1+4(D_2^2L_{\nabla P_1}^2\epsilon^{-2})\epsilon^{-2} -1)    \end{array}\right\rceil_+\\
       C_c = &~  \frac{4\max\{ \frac{1}{2L_{\nabla P_1} }, \min \{ \frac{D_2}{\epsilon}, \frac{4}{\alpha_c L_{\nabla P_1}} \} \}\Big(\delta_c + 2\alpha_c^{-1} (f^* - f_{\text{low}} + \epsilon D_2/4 + L_{\nabla P_1}D_1^2)\Big)}{[(3L_{\nabla P_1} + \epsilon/(2D_2))^2/\min \{ L_{\nabla P_1}, \epsilon/(2D_2)\} + 3 L_{\nabla P_1} + \epsilon/(2D_2)]^{-2}\epsilon^3}\nonumber\\
       N_c =  &~ (\lceil 96\sqrt{2}(1+(24L_{\nabla P_1} + 4\epsilon/D_2)L_{\nabla P_1}^{-1}) \rceil + 2)\{ 2, \sqrt{D_2L_{\nabla P_1}\epsilon^{-1}} \}\\ &~~~~~~~~~~\times \Big((K_c+1) \times (\log(C_c))_{+} + K_c + 1 + 2K_c\log(K_c+1) \Big).
    \end{align*}
Then by Theorem~\ref{thm:minmax} and Lemma~\ref{thm:e-kkt}, Algorithm~\ref{alg:first-order-new} outputs a lifted $O(\epsilon)$-KKT point $(\vx_1^{\epsilon}, \vy_1^{\epsilon})$ of~\eqref{eq:blo_cons} in oracle complexity no more than $N_c$. Here, $L_{\nabla P_1}=O(\rho)=O(\epsilon^{-1})$, $K_c = O(\epsilon^{-3})$, and $N_c=\tilde O(\epsilon^{-4})$.
\end{proof}

\paragraph{Dependence of total oracle complexity on problem parameters.}
We further examine the dependence of the deterministic first-order
oracle complexity on the problem parameters in the constrained
bilevel setting.
Let
\[
\Delta_0 :=
f(\vx_1^{(0)},\vy_1^{(0)})-f_{\mathrm{low}}
\]
denote the initial objective gap relative to $f_{\mathrm{low}}$, and define
\[
\widetilde{L}
:=L_{\nabla\bar{f}_1}
+B\sqrt{l}\,L_{\nabla\bar{g}}
+2L_{\bar{g}},
\]
which can be viewed as an overall smoothness parameter of the lower-level constrained problem. 
With the choice $\rho=\epsilon^{-1}$, we have
$L_{\nabla P_1}
=L_{\nabla f_1}+2\epsilon^{-1}\widetilde{L}$.
By the definition of $K_c$ and $N_c$ above, we obtain
\begin{align*}
K_c
= O\left(
    \frac{(\Delta_0+\epsilon D_2)(L_{\nabla f_1}
+\epsilon^{-1}\widetilde{L})}{\epsilon^2}
    \right),\quad
N_c
= \tilde{O}\left(
    \sqrt{\frac{D_2(L_{\nabla f_1}
+\epsilon^{-1}\widetilde{L})}{\epsilon}}K_c
    \right).
\end{align*}
These bounds explicitly characterize the dependence on the
upper-level smoothness constant $L_{\nabla f_1}$.

We next examine the dependence on the number of lower-level
constraints and the margin of Slater's condition, i.e., $G$.
Since the dual feasible set is $[0,B]^l$, we have
$$
D_2\leq D_{\mathcal{Y}_1} +B\sqrt{l}.
$$
For simplicity, we next only keep the terms with leading-order dependence on $\epsilon$,
$B$ and $l$, and the preceding bound implies
\[
K_c
= O\left(
    \frac{B\sqrt{l}}{\epsilon^3}
    \right),\quad
N_c
=\tilde{O}\left(
\frac{
B^2l
}{\epsilon^4}
\right).
\]

Furthermore, by \eqref{eq:cons_y1}, the bound $B$ depends on the
Slater margin $G$ in the order of $O(G^{-1})$.
Consequently, when the remaining problem parameters are
fixed, the complexity's dependence on the
 Slater margin is $O(G^{-2})$ in the leading term.

\section{A stochastic first-order method for nonconvex-concave minimax problem}
\label{sec:SAPD}
In this section, our goal is to solve \eqref{eq:bilevel} and \eqref{eq:bilevel_equality}  by solving its penalty formulation \eqref{eq:new_minmax} under the stochastic setting. We still view \eqref{eq:new_minmax} as an instance of \eqref{eq:minmax} and, in addition to Assumption~\ref{assump:h}, we assume the existence of the stochastic first-order oracle of $h$. 
\begin{assumption}
\label{assump:stoc_h}
There exist a random variable $\omega$ and a stochastic gradient estimator  of $\nabla h(x,y)$, which is a mapping of $(x,y)$ and $\omega$ denoted by $\hat{\nabla} h(x,y;\omega)$, and satisfies
 \begin{enumerate}
    \item $\mathbb{E}[\hat{\nabla} h(x,y;\omega)] = \nabla h(x,y)$.
    \item $\mathbb{E}[\|\hat{\nabla} h(x,y;\omega) - \nabla h(x,y)\|^2] \leq \delta^2 $.
 \end{enumerate}
\end{assumption}

Our goal here is to find a \emph{nearly $\epsilon$-primal-dual stationary point} of \eqref{eq:minmax}, which is a solution $(x^{\epsilon},y^{\epsilon})$ such that there exists $(\hat{x}^{\epsilon},\hat{y}^{\epsilon})$ such that $\|(x^{\epsilon},y^{\epsilon})-(\hat{x}^{\epsilon},\hat{y}^{\epsilon})\|\leq \epsilon$ and $(\hat{x}^{\epsilon},\hat{y}^{\epsilon})$ is a $\epsilon$-primal-dual stationary point of \eqref{eq:minmax} (defined in \eqref{eq:pd_stationary_H}). To find a nearly $\epsilon$-primal-dual stationary point of \eqref{eq:minmax}, we adopt the inexact proximal-point method again, which generates the next iterate by approximately solving the subproblem \eqref{eq:minmax_prox_h} using a stochastic first-order method. Similarly to the deterministic case, this subproblem is an instance of \eqref{ea-prob} satisfying the following assumption in addition to Assumption~\ref{assump:barh}.
\begin{assumption}
\label{assump:stoc_barh}
There exist a random variable $\omega$ and a stochastic gradient estimator  of $\nabla \bar{h}(x,y)$, which is a mapping of $(x,y)$ and $\omega$ denoted by $\hat{\nabla} \bar{h}(x,y;\omega)$, and satisfies
 \begin{enumerate}
    \item $\mathbb{E}[\hat{\nabla} \bar{h}(x,y;\omega)] = \nabla \bar{h}(x,y)$.
    \item $\mathbb{E}[\|\hat{\nabla} \bar{h}(x,y;\omega) - \nabla \bar{h}(x,y)\|^2] \leq \delta^2 $.
 \end{enumerate}
\end{assumption}
Note that, in the case of \eqref{eq:minmax_prox_h} where $\bar{h}$ \eqref{eq:mapbarh} holds, we are under the setting where
\begin{align}
\label{eq:map_hatnabla_barh}
\hat{\nabla} \bar{h}(x,y;\omega)\leftarrow \hat{\nabla} h(x,y;\omega)+(2L_{\nabla h}(x-x^k), -\epsilon/(2D_q)(y-y^0))
\end{align}
and $\hat{\nabla} \bar{h}(x,y;\omega)$ and $\hat{\nabla} h(x,y;\omega)$ in \eqref{eq:map_hatnabla_barh} have the same vriance bound $\delta^2$.

To solve the generic problem \eqref{ea-prob} in the stochastic setting, one of the optimal algorithms is the SAPD method~\citep{zhang2024robust,zhang2022sapd+} presented in Algorithm~\ref{alg:sapd}. According to \citet[Theorem 2.11]{zhang2024robust} and inequalities (2.28) and (2.32) in its proof, to achieve the optimal complexity, the number of iterations $T$ and the control parameters $(\tau,\sigma,\theta)$ are set as follows.
    \begin{align}
    \label{eq:sapd_T}
        &T~= \left\lceil1 + \ln\left(\frac{6\sigma_xD_p^2+6\sigma_yD_q^2}{\hat{\epsilon}^2}\right)\max\left\{(1-\overline{\theta}_1)^{-1},(1-\overline{\theta}_2)^{-1},(1-\overline{\overline{\theta}}_1)^{-1},(1-\overline{\overline{\theta}}_2)^{-1}\right\}\right\rceil.
    \end{align}
\begin{align}
\label{eq:SAPD-parameter-choice-R1}
    \tau = \frac{1-\theta}{\sigma_x \theta},\quad \sigma = \frac{1-\theta}{\sigma_y\theta},\quad \theta=\max\{\overline{\theta}_1,~\overline{\theta}_2,~ \overline{\overline{\theta}}_1 ,\; \overline{\overline{\theta}}_2 \},
\end{align}
where 
\small
\begin{equation}
\label{eq:theta1-R1}
       \overline{\theta}_1:= 1 -\tfrac{\beta (L_{\nabla \bar{h}}+ \sigma_x) \sigma_y}{4L_{\nabla \bar{h}}^2}
          \Big(\sqrt{ 1+ \tfrac{8\sigma_xL_{\nabla \bar{h}}^2}{\beta\sigma_y(L_{\nabla \bar{h}}+\sigma_x)^2}}-1\Big),\quad \overline{\theta}_2:=
        1 - \tfrac{(1-\beta)^2}{32}\tfrac{\sigma_y^2}{L_{\nabla \bar{h}}^2} \Big( \sqrt{1+\tfrac{64L_{\nabla \bar{h}}^2}{(1-\beta)^2\sigma_y^2}}-1\Big).
\end{equation}
\begin{align}
\label{eq:theta_bound_2-R1}
\overline{\overline{\theta}}_1 := \max\Big\{0,1 -
\frac{1}{12\Xi^x(\beta)}\frac{\sigma_x}{\delta^2}\hat{\epsilon}^2\Big\},\quad \overline{\overline{\theta}}_2 := \max\Big\{0,1 -
\frac{1}{12\Xi^y(\beta)}\frac{\sigma_y}{\delta^2}\hat{\epsilon}^2\Big\},
\end{align}
\begin{align*}
{\Xi}^x(\beta):= 1+ \Psi(\beta),\quad {\Xi}^y(\beta):=\frac{27+3\beta}{2} + \frac{\sigma_y}{\sigma_x}\Psi(\beta),\\
\Psi(\beta):= \min\left\{\sqrt{\frac{\beta}{2}\frac{\sigma_x}{\sigma_y}},\frac{1-\beta}{4}\right\},\quad \beta:= \min\left\{\frac{1}{2},\frac{\sigma_y}{\sigma_x},\frac{\sigma_x}{\sigma_y}\right\}.
\end{align*}
\normalsize
Using Algorithm~\ref{alg:sapd} to approximately solve subprpoblem \eqref{eq:minmax_prox_h} in each main iteration leads to a variant of the SAPD+ algorithm~\citep{zhang2022sapd+}, which is presented in Algorithm~\ref{alg:first-order-stoch} for solving  \eqref{eq:minmax}. The stochastic case of Algorithm~\ref{alg:first-order-new} is the specification of Algorithm~\ref{alg:first-order-stoch} when applied to \eqref{eq:new_minmax}.

\begin{algorithm}[t]
\caption{A stochastic accelerated primal-pual (SAPD) algorithm for \eqref{ea-prob}:\texttt{SAPD}$(\bar\epsilon,x^0,y^0,\bar{h},p,q,L_{\nabla \h},\sigma_x,\sigma_y,\delta^2)$}
\label{alg:sapd}
\begin{algorithmic}[1]
\STATE \textbf{Input:} targeted stationarity level $\bar\epsilon>0$, initial solution $(x^0,y^0)\in\dom\,p\times\dom\,q$, $(\bar{h},p,q)$ from \eqref{ea-prob},  $(L_{\nabla \h},\sigma_x,\sigma_y)$ from Assumption~\ref{assump:barh}, and $\delta^2$ from Assumption~\ref{assump:stoc_h}.
\STATE Set $T$, $\tau$, $\sigma$, and $\theta$  as in \eqref{eq:sapd_T}, \eqref{eq:SAPD-parameter-choice-R1},  \eqref{eq:theta1-R1} and \eqref{eq:theta_bound_2-R1}.
\STATE {$\tilde{q}_0\gets 0$}
\FOR{$k=0,1,2,...,T-1$}
\STATE $s^k \leftarrow \hat{\nabla}_y\bar{h}(x^k,y^k;\omega_y^k) + \theta \tilde{q}_k$
\STATE $y^{k+1}\leftarrow\prox_{\sigma q}(y^k+\sigma s_k)$
\STATE $x^{k+1}\leftarrow\prox_{\tau p}(x^k-\tau\hat{\nabla}_x\bar{h}(x^k,y^{k+1};\omega_x^k))$
\STATE $q_{k+1}\leftarrow \hat{\nabla}_y\bar{h}(x^{k+1},y^{k+1};\omega_y^{k+1}) -\hat{\nabla}_y\bar{h}(x^k,y^k;\omega_y^k)$
\ENDFOR
\STATE {\textbf{Output}:}{$(x^T,y^T)$}
\end{algorithmic}%
\end{algorithm}

\begin{algorithm}[t]
\caption{A stochastic inexact proximal point method for~\eqref{eq:minmax}}
\label{alg:first-order-stoch}
\begin{algorithmic}[1]
\STATE {\bfseries Input:} number of iterations $K>0$, inner targeted stationarity level $\hat{\epsilon}>0$, initial solution $(x^0,y^0)\in\dom\,p\times\dom\,q$, $(h,p,q)$ from \eqref{eq:minmax},  $L_{\nabla h}$ from Assumption~\ref{assump:h}, $D_q$ in \eqref{eq:DpDq}, and $\delta^2$ from Assumption~\ref{assump:stoc_h}.
\FOR{$k=0,1,2,\ldots$}
\STATE Set $\bar{h}$ as in $\eqref{eq:mapbarh}$ and set $(L_{\nabla \h},\sigma_x,\sigma_y)$ as in \eqref{eq:mapbarL}.
\STATE 
$$
(x^{k+1}, y^{k+1})\leftarrow\texttt{SAPD}\left(\hat{\epsilon},x^{k}, y^{k}, \bar{h}, p,q,L_{\nabla \h},\sigma_x,\sigma_y,\delta^2\right)
$$
\IF{$k=K$
}
\STATE Return $(x^{\epsilon},y^{\epsilon})=(x^{k'+1},y^{k'+1})$,  where $k'$ is uniformaly randomly sampled from $\{0,1,\dots,K-1\}$.
\ENDIF
\ENDFOR
\end{algorithmic}
\end{algorithm}

The complexity of Algorithm~\ref{alg:sapd} is established in~\citet{zhang2024robust,zhang2022sapd+} and Algorithm~\ref{alg:first-order-stoch} is established in \citet{zhang2022sapd+}. However, the result in \citet{zhang2022sapd+} cannot be directly applied to our problem because their convergence theory is for finding a nearly $\epsilon$-stationary point of the outer function, or more specifically, a point $x$ such that $\|\nabla \phi_{2L_{\nabla h}}(x)\|\leq\epsilon$, where
\begin{align}
\label{eq:me_h_old}
\phi_{2L_{\nabla h}}(x):=\min_{x'}\max_y\left[H(x',y)+L_{\nabla h} \|x'-x\|^2\right]
\end{align}
and $H$ is from \eqref{eq:minmax}. However, our goal is instead to find a nearly $\epsilon$-primal-dual stationary point of \eqref{eq:minmax}. Given this difference, we prove the following theorem that characterizes the complexity of Algorithm~\ref{alg:first-order-stoch} for finding a nearly $\epsilon$-primal-dual stationary point. This theorem is not only the key result for establishing the complexity of finding a nearly lifted $\epsilon$-KKT point of \eqref{eq:bilevel} but also an interesting result itself.

\begin{theorem}\label{thm:sapd+_thm}Suppose Assumptions~\ref{assump:h} and ~\ref{assump:stoc_h} hold. Let $H$, $H^*$, $D_p$, and $D_q$ be defined in \eqref{eq:minmax}, \eqref{eq:DpDq} and \eqref{eq:H_low}.
Let 
    \begin{align*}
       &~ K = \left\lceil \left(\mathbb{E}\left[\max_{y}H(x^0,y)\right]- H^* + \epsilon D_q/4\right)\hat{\epsilon}^{-2}\right\rceil,\\
       &~ C= \Bigg\lceil
1
+
\ln\!\left(
\frac{
6L_{\nabla h}D_p^2+3\epsilon D_q
}{
\hat{\epsilon}^2
}
\right)
\Bigg[
\left(
L_{\nabla h}+\max\left\{2L_{\nabla h},\frac{\epsilon}{2D_q}\right\}
\right)
\left(
\frac{2\sqrt{2D_q}}{\sqrt{L_{\nabla h}\epsilon}}
+\frac{3}{L_{\nabla h}}
+\frac{24D_q}{\epsilon}
\right) \\
&\hspace{6cm}
+
\frac{\delta^2}{\hat{\epsilon}^2}
\left(
\frac{21}{L_{\nabla h}}
+
\frac{360D_q}{\epsilon}
\right)
\Bigg]
\Bigg\rceil ,
\\
        &~ N = KC.
    \end{align*}
    Then Algorithm~\ref{alg:first-order-stoch} terminates and outputs a solution $(x^{\epsilon},y^{\epsilon})$ such that there exists a stochastic solution $(\hat{x}^{\epsilon},\hat{y}^{\epsilon})$ such that 
    \begin{align}
    \label{eq:distance_stoc_h}
    \mathbb{E}\|(x^{\epsilon},y^{\epsilon})-(\hat{x}^{\epsilon},\hat{y}^{\epsilon})\|\leq \max\left\{\frac{\hat{\epsilon}}{\sqrt{L_{\nabla h}}},\frac{\hat{\epsilon}}{\sqrt{\epsilon/(2D_q)}}\right\}
    \end{align}
    and
    $(\hat{x}^{\epsilon},\hat{y}^{\epsilon})$ is a $ \max\left\{\sqrt{2L_{\nabla h}}\hat{\epsilon},\frac{\epsilon}{2}\right\}$-primal-dual stationary point of \eqref{eq:minmax} in expectation. As a consequence, Algorithm~\ref{alg:first-order-stoch} finds a nearly $\max\left\{\sqrt{2L_{\nabla h}}\hat{\epsilon},\frac{\epsilon}{2},\frac{\hat{\epsilon}}{\sqrt{L_{\nabla h}}},\frac{\hat{\epsilon}}{\sqrt{\epsilon/(2D_q)}}\right\}$-primal-dual stationary point $(x^{\epsilon},y^{\epsilon})$ of ~\eqref{eq:minmax} in expectation in complexity no more than $N$. Moreover, it holds that
\begin{align}
\label{eq:bound_phik}
\mathbb{E}\left[\max_y H(\hat{x}^{\epsilon},y)\right]\leq 2\mathbb{E}\left[\max_{y}H(x^{0},y)\right]-H^*+\epsilon D_q/2,
\end{align} 
\end{theorem}

\begin{proof}
We first focus on the the $k$th main iteration. Let $(\hat{x}^k,\hat{y}^k)$ be the unique saddle point of \eqref{eq:minmax_prox_h}. Recall that $(x^{k+1},y^{k+1})$ is returned by  Algorithm~\ref{alg:sapd} initialized at $(x^k,y^k)$ and after running for $T$ iterations. The parameters $\tau$, $\sigma$ and $\theta$ in 
Algorithm~\ref{alg:sapd} are set as in \eqref{eq:SAPD-parameter-choice-R1}, \eqref{eq:theta1-R1} and \eqref{eq:theta_bound_2-R1} according to \citet[Theorem 2.11]{zhang2024robust}. Therefore, by \citet[Theorem 2]{zhang2024robust} and inequalities (2.27) and (2.28) in its proof in their paper, Algorithm~\ref{alg:sapd} ensures
\begin{equation}\label{eq:sapd_result}
     \mathbb{E}[\sigma_x\|x^{k+1}-\hat{x}^k\|^2+\sigma_y\|y^{k+1}-\hat{y}^k\|^2 ]
              \leq \hat{\epsilon}^2
\end{equation}
if 
\small
\begin{equation}\label{eq:sapd_Told}
T\geq 1 + \ln\left(\frac{6\sigma_x\|x^k-\hat{x}^k\|^2+6\sigma_y\|y^k-\hat{y}^k\|^2}{\hat{\epsilon}^2}\right)\max\left\{(1-\overline{\theta}_1)^{-1},(1-\overline{\theta}_2)^{-1},(1-\overline{\overline{\theta}}_1)^{-1},(1-\overline{\overline{\theta}}_2)^{-1}\right\}.
\end{equation}
\normalsize
By \eqref{eq:DpDq}, we have $\|x^k-\hat{x}^k\|^2\leq D_p^2$ and $\|y^k-\hat{y}^k\|^2\leq D_q^2$. Therefore, the choice of $T$ in Algorithm~\ref{alg:sapd} does satisfy \eqref{eq:sapd_Told}, so \eqref{eq:sapd_result} holds.

Note that
\begin{align*}
         (1-\overline{\theta}_1)^{-1}
         =&  \frac{1}{2}\left(\frac{L_{\nabla \bar{h}}}{\sigma_x}+1\right) + \sqrt{\frac{1}{4}\left(\frac{L_{\nabla \bar{h}}}{\sigma_x} + 1\right)^2 + \frac{2L_{\nabla \bar{h}}^2}{\beta \sigma_x\sigma_y}},\\
         (1-\overline{\theta}_2)^{-1}
     =& \frac{1}{2} + \sqrt{\frac{1}{4} +\frac{16L_{\nabla \bar{h}}^2}{\left(1-\beta\right)^2\sigma_y^2}},\\
  (1-\overline{\overline{\theta}}_1)^{-1}=&12{\Xi}^x(\beta)\frac{\delta^2}{\sigma_x\hat{\epsilon}^2},\qquad (1-\overline{\overline{\theta}}_2)^{-1}=12{\Xi}^y(\beta)\frac{\delta^2}{\sigma_y\hat{\epsilon}^2}.
\end{align*}
Using the facts that $\sqrt{a^2+b^2}\leq a+b$, that $1/\min\{a,b\}\leq \frac{1}{a}+\frac{1}{b}$, that $\max\{\sigma_x,\sigma_y\}\leq L_{\nabla \bar{h}}$, and the definitions of $\beta$, ${\Xi}^x(\beta)$ and ${\Xi}^y(\beta)$, we have ${\Xi}^x(\beta)\leq 1+\sqrt{1/2}$, ${\Xi}^y(\beta)\leq (27+3/2)/2+1/\sqrt{2}$,
\begin{align*}
         (1-\overline{\theta}_1)^{-1}
         \leq& \frac{2L_{\nabla \bar{h}}}{\sqrt{\sigma_x\sigma_y}}+ \frac{3L_{\nabla \bar{h}}}{\sigma_x}+ \frac{3L_{\nabla \bar{h}}}{\sigma_y},\qquad
         (1-\overline{\theta}_2)^{-1}
     \leq \frac{9L_{\nabla \bar{h}}}{\sigma_y},\\
  (1-\overline{\overline{\theta}}_1)^{-1}=&\frac{21\delta^2}{\sigma_x\hat{\epsilon}^2},\qquad (1-\overline{\overline{\theta}}_2)^{-1}=\frac{180\delta^2}{\sigma_y\hat{\epsilon}^2}.
\end{align*}
These results imply that 
$$
T\leq\left\lceil1 + \ln\left(\frac{6\sigma_xD_p^2+6\sigma_yD_q^2}{\hat{\epsilon}^2}\right)\left(\frac{2L_{\nabla \bar{h}}}{\sqrt{\sigma_x\sigma_y}}+ \frac{3L_{\nabla \bar{h}}}{\sigma_x}+ \frac{12L_{\nabla \bar{h}}}{\sigma_y}+\frac{21\delta^2}{\sigma_x\hat\epsilon^2}+\frac{180\delta^2}{\sigma_y\hat\epsilon^2}\right)\right\rceil.
$$
Since we are analyzing Algorithm~\ref{alg:sapd} when it is applied \eqref{eq:minmax_prox_h} under the setting of \eqref{eq:mapbarh} and \eqref{eq:mapbarL}, replacing $L_{\nabla\h}$, $\sigma_x$, and $\sigma_y$ in the right-hand side above gives
\begin{equation}
\label{eq:TC}
\begin{aligned}
T
\leq
\Bigg\lceil
1
&+
\ln\!\left(
\frac{
6L_{\nabla h}D_p^2+3\epsilon D_q
}{
\hat{\epsilon}^2
}
\right)
\Bigg[
\left(
L_{\nabla h}+\max\left\{2L_{\nabla h},\frac{\epsilon}{2D_q}\right\}
\right)
\left(
\frac{2\sqrt{2D_q}}{\sqrt{L_{\nabla h}\epsilon}}
+\frac{3}{L_{\nabla h}}
+\frac{24D_q}{\epsilon}
\right) \\
&\hspace{4cm}
+
\frac{\delta^2}{\hat{\epsilon}^2}
\left(
\frac{21}{L_{\nabla h}}
+
\frac{360D_q}{\epsilon}
\right)
\Bigg]
\Bigg\rceil =C.
\end{aligned}
\end{equation}
This result shows that the complexity of each call of Algorithm~\ref{alg:sapd} is bounded by $C$ in each main iteration of Algorithm~\ref{alg:first-order-stoch}.

Nest, we will analyze the number of main iterations needed in Algorithm~\ref{alg:first-order-stoch}. 
Let $\phi_{2L_{\nabla h}}^{\epsilon}(x)$ be the Moreau envelope of $\max_y\left[H(x,y)-\epsilon\|y-y^0\|^2/(4D_q)\right]$, i.e., 
\begin{align}
\label{eq:me_h}
\phi_{2L_{\nabla h}}^{\epsilon}(x):=\min_{x'}\max_y\left[H(x',y)-\epsilon\|y-y^0\|^2/(4D_q)+L_{\nabla h} \|x'-x\|^2\right].
\end{align}
We then have 
\begin{align}
\nonumber
&\phi_{2L_{\nabla h}}^{\epsilon}(x^{k+1})\\\nonumber
=&\min_{x}\max_y\left[H(x,y)-\epsilon\|y-y^0\|^2/(4D_q)+L_{\nabla h} \|x-x^{k+1}\|^2\right]\\\nonumber
\leq&\max_y\left[H(\hat{x}^k,y)-\epsilon\|y-y^0\|^2/(4D_q)\right]+L_{\nabla h} \|\hat{x}^k-x^{k+1}\|^2\\\nonumber
=&\max_y\left[H(\hat{x}^k,y)-\epsilon\|y-y^0\|^2/(4D_q)\right]+L_{\nabla h} \|\hat{x}^k-x^{k}\|^2-L_{\nabla h} \|\hat{x}^k-x^{k}\|^2+L_{\nabla h} \|\hat{x}^k-x^{k+1}\|^2\\\nonumber
=&\phi_{2L_{\nabla h}}^{\epsilon}(x^{k})-L_{\nabla h} \|\hat{x}^k-x^{k}\|^2+L_{\nabla h} \|\hat{x}^k-x^{k+1}\|^2.
\end{align}
Taking the expectation and summing the inequality above for $k=0,1,\dots,K-1$ gives
\begin{align}
\nonumber
\sum_{k=0}^{K-1} L_{\nabla h} \mathbb{E}\|\hat{x}^k-x^{k}\|^2
\leq&\mathbb{E}\left[\phi_{2L_{\nabla h}}^{\epsilon}(x^{0})-\phi_{2L_{\nabla h}}^{\epsilon}(x^{K})\right]
+\sum_{k=0}^{K-1} L_{\nabla h} \mathbb{E}\|\hat{x}^k-x^{k+1}\|^2\\\label{eq:summup_k}
\leq &\mathbb{E}\left[\phi_{2L_{\nabla h}}^{\epsilon}(x^{0})-\phi_{2L_{\nabla h}}^{\epsilon}(x^{K})\right]+K\hat{\epsilon}^2
\end{align}
where the last inequality is because of \eqref{eq:sapd_result} and the fact that $\sigma_x=L_{\nabla h}$. Moreover, by definitions, it is easy to show that
$$
\phi_{2L_{\nabla h}}^{\epsilon}(x^{0})\leq \max_{y}H(x^{0},y),\quad
\phi_{2L_{\nabla h}}^{\epsilon}(x^{K})\geq H^*-\epsilon D_q/4.
$$
Recall that $k'$ is uniformly randomly sampled from $\{0,1,\dots,K-1\}$. 
The two inequalities above and  \eqref{eq:summup_k} together imply 
\begin{align}
L_{\nabla h}^2 \mathbb{E}\|\hat{x}^{k'}-x^{k'}\|^2\label{eq:summup_exp}
\leq &\frac{L_{\nabla h}(\mathbb{E}[\max_{y}H(x^{0},y)]-H^*+\epsilon D_q/4)}{K}+L_{\nabla h}\hat{\epsilon}^2
\leq2L_{\nabla h}\hat{\epsilon}^2,
\end{align}
where the second inequality is from the definition of $K$.

Recall that Algorithm~\ref{alg:first-order-stoch} returns $(x^{\epsilon},y^{\epsilon})=(x^{k'+1},y^{k'+1})$,  where $k'$ is uniformaly randomly sampled from $\{0,1,\dots,K-1\}$. We will show that $(\hat{x}^{\epsilon},\hat{y}^{\epsilon})=(\hat{x}^{k'},\hat{y}^{k'})$ will be the stochastic solution in the statement of the theorem.

By the optimal condition of $(\hat{x}^k,\hat{y}^k)$, it holds for any $k\geq0$ that
\begin{align*}
\mathbf{0}\in&\partial p(\hat{x}^k)+\nabla_x h(\hat{x}^k,\hat{y}^k)+L_{\nabla h}(\hat{x}^k-x^k),\\
\mathbf{0}\in&\partial q(\hat{x}^k)-\nabla_y h(\hat{x}^k,\hat{y}^k)+\epsilon(\hat{y}^k-y^0)/(2D_q),
\end{align*}
By \eqref{eq:summup_exp} and Jensen's inequality, 
the two inclusions above imply that 
\begin{align}
\label{eq:xhat_stationary}
\mathbb{E}\text{dist}\left(\mathbf{0},\partial p(\hat{x}^{k'})+\nabla_x h(\hat{x}^{k'},\hat{y}^{k'})\right)\leq& L_{\nabla h}\mathbb{E}\|\hat{x}^{k'}-x^{k'}\|\leq \sqrt{2L_{\nabla h}}\hat{\epsilon},\\
\label{eq:yhat_stationary}
\mathbb{E}\text{dist}\left(\mathbf{0},\partial q(\hat{x}^{k'})-\nabla_y h(\hat{x}^{k'},\hat{y}^{k'})\right)\leq &\epsilon\mathbb{E}\|\hat{y}^{k'}-y^0\|/(2D_q)\leq \frac{\epsilon}{2},
\end{align}
which means $(\hat{x}^{\epsilon},\hat{y}^{\epsilon})=(\hat{x}^{k'},\hat{y}^{k'})$ is an $\max\left\{\sqrt{2L_{\nabla h}}\hat{\epsilon},\frac{\epsilon}{2}\right\}$-primal-dual solution of \eqref{eq:minmax} in expectation.

Additionally, by \eqref{eq:sapd_result}, Jensen's inequality and the facts that $\sigma_x=L_{\nabla h}$ and $\sigma_y=\epsilon/(2D_q)$, it also holds that
\begin{align}
\nonumber
   \mathbb{E}\|(x^{\epsilon},y^{\epsilon})-(\hat{x}^{\epsilon},\hat{y}^{\epsilon})\|
  =& \mathbb{E}\|(x^{k'+1},y^{k'+1})-(\hat{x}^{k'},\hat{y}^{k'})\|
  \leq \sqrt{\mathbb{E}\|(x^{k'+1},y^{k'+1})-(\hat{x}^{k'},\hat{y}^{k'})\|^2}\\ \label{eq:xhatyhat_distance}
  \leq& \sqrt{\frac{\hat{\epsilon}^2}{\min\{\sigma_x,\sigma_y\}}} \leq \max\left\{\frac{\hat{\epsilon}}{\sqrt{L_{\nabla h}}},\frac{\hat{\epsilon}}{\sqrt{\epsilon/(2D_q)}}\right\} .
\end{align}

Comparing the orders of the right-hand sides in \eqref{eq:xhat_stationary}, \eqref{eq:yhat_stationary}, and \eqref{eq:xhatyhat_distance}, we conclude by definition that $\left(x^{\epsilon},y^{\epsilon}\right)=\left(x^{k'+1},y^{k'+1}\right)$ is a nearly $\max\left\{\sqrt{2L_{\nabla h}}\hat{\epsilon},\frac{\epsilon}{2},\frac{\hat{\epsilon}}{\sqrt{L_{\nabla h}}},\frac{\hat{\epsilon}}{\sqrt{\epsilon/(2D_q)}}\right\}$-primal-dual stationary point of ~\eqref{eq:minmax} in expectation. It is found in complexity no more than $TK\leq CK=N$. 

Next we prove the second conclusion. By the same argument we used to show \eqref{eq:summup_k}, it is easy to see that \eqref{eq:summup_k} holds even if we replace $K$ by $k=0,1,\dots,K-1$. In fact, we have, for $k=0,1,\dots,K-1$, that
\begin{align}
\nonumber
0\leq \sum_{l=0}^{k-1} L_{\nabla h} \mathbb{E}\|\hat{x}^l-x^{l}\|^2
\leq \mathbb{E}\left[\phi_{2L_{\nabla h}}^{\epsilon}(x^{0})-\phi_{2L_{\nabla h}}^{\epsilon}(x^{k})\right]+k\hat{\epsilon}^2.
\end{align}
which implies
\begin{align}
\nonumber
\mathbb{E}\phi_{2L_{\nabla h}}^{\epsilon}(x^{k})\leq &\phi_{2L_{\nabla h}}^{\epsilon}(x^{0})+k\hat{\epsilon}^2
\leq\mathbb{E}[\max_{y}H(x^{0},y)]+k\hat{\epsilon}^2\\\label{eq:bound_phik_old}
\leq& \mathbb{E}[\max_{y}H(x^{0},y)]+K\hat{\epsilon}^2\leq 2\mathbb{E}[\max_{y}H(x^{0},y)]-H^*+\epsilon D_q/4,
\end{align}
for $k=0,\dots,K-1$, where the last inequality is from the definition of $K$.
On the other hand, we have
$$
\phi_{2L_{\nabla h}}^{\epsilon}(x^{k})=\max_y\left[H(\hat{x}^{k},y)-\epsilon\|y-\hat{y}^0\|^2/(4D_q)+L_{\nabla h} \|\hat{x}^{k}-x^k\|^2\right]
\geq \max_y H(\hat{x}^{k},y)-\epsilon D_q/4.
$$
Applying the inequality above to \eqref{eq:bound_phik_old} leads to 
\begin{align*}
\mathbb{E}\left[\max_y H(\hat{x}^{k},y)\right]\leq 2\mathbb{E}[\max_{y}H(x^{0},y)]-H^*+\epsilon D_q/2,
\end{align*}
for $k=0,1,\dots,K-1$. Letting $k=k'$ above and taking  expectation over $k'$ leads to \eqref{eq:bound_phik}.
\end{proof}

\begin{remark}
According to this theorem, when $L_{\nabla h}=O(1)$, Algorithm~\ref{alg:first-order-stoch} finds  a nearly $O\left(\max\left\{\epsilon,\frac{\hat{\epsilon}}{\sqrt{\epsilon}}\right\}\right)$-primal-dual stationary point of ~\eqref{eq:minmax} in $K = O(\hat\epsilon^{-2})$ iterations with a total stochastic oracle complexity of $N=\tilde{O}((\epsilon^{-1}+\delta^2\epsilon^{-1}\hat\epsilon^{-2})K)$. 
If $\hat{\epsilon}$ is further set to be $\Theta\left(\epsilon^{1.5}\right)$ and $L_{\nabla h}=O(1)$, the complexity becomes $N=\tilde{O}((\epsilon^{-3}+\delta^2\epsilon^{-7}))$.
\end{remark}

\section{Proof of Theorem~\ref{thm:minmax_stoch}}\label{sec:proof_thm2}
\begin{proof}[Proof of Theorem~\ref{thm:minmax_stoch}]

Just like the deterministic case, subproblem \eqref{eq:minmax_prox} is solved by Algorithm~\ref{alg:sapd} as an instance of \eqref{ea-prob} under the setting of \eqref{eq:mapPrho} and \eqref{eq:mapPrhoL}. In this setting, the stochastic gradient of $\bar{h}$ is set as 
\begin{align}
\nonumber
\hat{\nabla} \bar{h}(x,y;\omega)\leftarrow& ~
\hat\nabla f_1(\vx_1, \vx_2, \vy_1, \vy_2;\omega) + \rho\big(\hat\nabla\tilde{f}_1(\vx_1, \vy_1, \vz_2;\omega)
-\hat\nabla\tilde{f}_1(\vx_1 , \vz_1, \vy_2;\omega)\big)\\ \label{eq:map_hatnabla_barh_P}
&~+\left(\begin{array}{c}
2L_{\nabla P_1}\left[(\vx_1,\vy_1,\vy_2)-(\vx_1^{(k)},\vy_1^{(k)},\vy_2^{(k)})\right],\\
-\frac{\epsilon}{D_2}\left[(\vx_2,\vz_1,\vz_2)-(\vx_2^{(k)},\vz_1^{(k)},\vz_2^{(k)})\right]
\end{array}
\right),
\end{align}
where $L_{\nabla P_1}=L_{\nabla f} + 2\rho L_{\nabla \tilde{f}}$ and $D_2$ is defined in \eqref{eq:D2}. Under Assumption~\ref{assump:stoc}, we have
\begin{align*}
&~\mathbb{E}\left\|
\begin{array}{ll}
\hat\nabla f_1(\vx_1, \vx_2, \vy_1, \vy_2;\omega) + \rho\big(\hat\nabla\tilde{f}_1(\vx_1, \vy_1, \vz_2;\omega)-\hat\nabla\tilde{f}_1(\vx_1 , \vz_1, \vy_2;\omega)\big)\\
-
\nabla f_1(\vx_1, \vx_2, \vy_1, \vy_2) + \rho\big(\nabla\tilde{f}_1(\vx_1, \vy_1, \vz_2)-\nabla\tilde{f}_q(\vx_1 , \vz_1, \vy_2)\big)
\end{array}
\right\|^2\\
\leq&~3\mathbb{E}\left\|\hat\nabla f_1(\vx_1, \vx_2, \vy_1, \vy_2;\omega)-\nabla f_1(\vx_1, \vx_2, \vy_1, \vy_2)\right\|^2\\
&+3\rho^2\mathbb{E}\left\|\hat\nabla\tilde{f}_1(\vx_1, \vy_1, \vz_2;\omega)-\nabla\tilde{f}_1(\vx_1, \vy_1, \vz_2)\right\|^2 \\
&+3\rho^2\mathbb{E}\left\|\hat\nabla\tilde{f}_1(\vx_1 , \vz_1, \vy_2;\omega)-\nabla\tilde{f}_q(\vx_1 , \vz_1, \vy_2)\right\|^2\\
\leq&~ 3\delta_{f}^2 + 6\rho^2\delta_{\tilde{f}}^2=\delta_P^2,
\end{align*}
where $\delta_P^2$ is the same as in Algorithm~\ref{alg:first-order-new}. 
This means $\hat{\nabla} \bar{h}(x,y;\omega)$ 
specified in \eqref{eq:map_hatnabla_barh_P} satisfies Assumption~\ref{assump:stoc_barh} with 
\begin{align}
\label{eq:delta_P}
\delta^2=\delta_{P}^2=3\delta_{f}^2 + 6\rho^2\delta_{\tilde{f}}^2=O(\epsilon^{-2}).
\end{align}

Since problem \eqref{eq:new_minmax} is also an instance of \eqref{eq:minmax} under the setting of \eqref{eq:maph} and \eqref{eq:mapLh}, Algorithm~\ref{alg:first-order-new} is exactly Algorithm~\ref{alg:first-order-stoch}
under the setting of \eqref{eq:maph} and \eqref{eq:mapLh}. Therefore, we can directly apply Theorem~\ref{thm:sapd+_thm} and Lemma~\ref{thm:connection} to prove Theorem~\ref{thm:minmax_stoch} as we show below.


Let $f^*$, $D_1$, $D_2$ and $f_{\text{low}}$ be defined in \eqref{eq:old_minmax}, \eqref{eq:D1}, \eqref{eq:D2} and \eqref{eq:flow}. Let
\small
    \begin{align*}
        &~L_{\nabla P_1} = L_{\nabla f_1}+2\rho L_{\nabla\tilde{f}_1},\\
        &~\delta_{P}^2=3\delta_{f}^2 + 6\rho^2\delta_{\tilde{f}}^2,\\\label{eq:sapd_TP}
        &~ C_P= \Bigg\lceil
1
+
\ln\!\left(
\frac{
6L_{\nabla P_1}D_p^2+3\epsilon D_2
}{
\hat{\epsilon}^2
}
\right)\\
&~~\times
\Bigg[
\left(
L_{\nabla P_1}+\max\left\{2L_{\nabla P_1},\frac{\epsilon}{2D_2}\right\}
\right)
\left(
\frac{2\sqrt{2D_2}}{\sqrt{L_{\nabla P_1}\epsilon}}
+\frac{3}{L_{\nabla P_1}}
+\frac{24D_2}{\epsilon}
\right) 
+
\frac{\delta_P^2}{\hat{\epsilon}^2}
\left(
\frac{21}{L_{\nabla P_1}}
+
\frac{360D_2}{\epsilon}
\right)
\Bigg]
\Bigg\rceil\\
       &~ K_P = \left\lceil \left(\mathbb{E}\left[\max_{\vx_2}f(\vx_1^{(0)}, \vx_2, \vy_1^{(0)}, \vy_2^{(0)})\right]+1- f^* + \epsilon D_2/4\right)\hat{\epsilon}^{-2}\right\rceil\\
        &~ N_P = C_PK_P.
    \end{align*}
\normalsize
    Note that 
\begin{equation}\label{eq:initial_bound_stoch}
\mathbb{E}\left[\max_{\vx_2, \vz_1, \vz_2} P_\rho(\vx_1^{(0)},\vx_2, \vy_1^{(0)},\vy_2^{(0)}, \vz_1, \vz_2)\right]
\leq
\mathbb{E}\left[\max_{\vx_2}f(\vx_1^{(0)}, \vx_2, \vy_1^{(0)}, \vy_2^{(0)})\right]+1
\end{equation}
because of the initial condition \eqref{eq:initial_condition_stoch} and the fact that $\rho=\frac{1}{\epsilon}$.

Recall that $\hat{\epsilon}$ is set to be $\epsilon^{1.5}$ in Algorithm~\ref{alg:first-order-new}. 
According to the discussion above, Algorithm~\ref{alg:first-order-new} is exactly Algorithm~\ref{alg:first-order-stoch}
under the setting of \eqref{eq:maph} and \eqref{eq:mapLh}. In this setting, $\max_{y}H(x^0,y)$ in Theorem~\ref{thm:sapd+_thm} corresponds to 
$$
\max_{\vx_2, \vz_1, \vz_2} P_\rho(\vx_1^{(0)},\vx_2, \vy_1^{(0)},\vy_2^{(0)}, \vz_1, \vz_2).
$$
Applying Theorem~\ref{thm:sapd+_thm} under this setting with $\hat{\epsilon}=\epsilon^{1.5}$  and with
$\max_{y}H(x^0,y)$ in $K$ in Theorem~\ref{thm:sapd+_thm}  replaced by 
its upper bound 
\begin{equation}\label{eq:initial_bound_f_stoch}
\mathbb{E}\left[\max_{\vx_2}f(\vx_1^{(0)}, \vx_2, \vy_1^{(0)}, \vy_2^{(0)})\right]+1,
\end{equation}
we can  show that, in complexity no more than $N_P$, Algorithm~\ref{alg:first-order-new} outputs $\left((\vx_1^{\epsilon}, \vy_1^{\epsilon}, \vy_2^{\epsilon}),(\vx_2^{\epsilon}, \vz_1^{\epsilon}, \vz_2^{\epsilon})\right)$ such that there exists a stochastic solution $\left((\hat{\vx}_1^{\epsilon}, \hat{\vy}_1^{\epsilon}, \hat{\vy}_2^{\epsilon}),(\hat{\vx}_2^{\epsilon}, \hat{\vz}_1^{\epsilon}, \hat{\vz}_2^{\epsilon})\right)$ such that 
\begin{equation}
\label{eq:distance_stoc_P}
 \begin{aligned}
&\mathbb{E}\left\|
(\vx_1^{\epsilon}, \vy_1^{\epsilon}, \vy_2^{\epsilon})
-
(\hat{\vx}_1^{\epsilon}, \hat{\vy}_1^{\epsilon}, \hat{\vy}_2^{\epsilon})
\right\|
\\
\leq&
\mathbb{E}\left\|
\left(
(\vx_1^{\epsilon}, \vy_1^{\epsilon}, \vy_2^{\epsilon}),
(\vx_2^{\epsilon}, \vz_1^{\epsilon}, \vz_2^{\epsilon})
\right)
-
\left(
(\hat{\vx}_1^{\epsilon}, \hat{\vy}_1^{\epsilon}, \hat{\vy}_2^{\epsilon}),
(\hat{\vx}_2^{\epsilon}, \hat{\vz}_1^{\epsilon}, \hat{\vz}_2^{\epsilon})
\right)
\right\|
\nonumber\\
\leq&
\max\left\{
\frac{\hat{\epsilon}}{\sqrt{L_{\nabla P_1}}},
\frac{\hat{\epsilon}}{\sqrt{\epsilon/(2D_2)}}
\right\}
\end{aligned}
\end{equation}
and   $\left((\hat{\vx}_1^{\epsilon}, \hat{\vy}_1^{\epsilon}, \hat{\vy}_2^{\epsilon}),(\hat{\vx}_2^{\epsilon}, \hat{\vz}_1^{\epsilon}, \hat{\vz}_2^{\epsilon})\right)$ is an $ \max\left\{\sqrt{2L_{\nabla P_1}}\hat{\epsilon},\frac{\epsilon}{2}\right\}$-primal-dual stationary point of \eqref{eq:new_minmax} in expectation.

Since $\hat{\epsilon}=\epsilon^{1.5}$ and $L_{\nabla P_1}=O(\rho)=O(\epsilon^{-1})$, we have
$$
\max\left\{\sqrt{2L_{\nabla P_1}}\hat{\epsilon},\frac{\epsilon}{2},\frac{\hat{\epsilon}}{\sqrt{L_{\nabla P_1}}},\frac{\hat{\epsilon}}{\sqrt{\epsilon/(2D_2)}}\right\}
=O(\epsilon)
$$
meaning that 
\begin{align}
\label{eq:distance_xhatx}
\mathbb{E}\left\|
(\vx_1^{\epsilon}, \vy_1^{\epsilon}, \vy_2^{\epsilon})
-
(\hat{\vx}_1^{\epsilon}, \hat{\vy}_1^{\epsilon}, \hat{\vy}_2^{\epsilon})
\right\|=O(\epsilon)
\end{align}
and $\left((\hat{\vx}_1^{\epsilon}, \hat{\vy}_1^{\epsilon}, \hat{\vy}_2^{\epsilon}),(\hat{\vx}_2^{\epsilon}, \hat{\vz}_1^{\epsilon}, \hat{\vz}_2^{\epsilon})\right)$ is an $O(\epsilon)$-primal-dual stationary point of \eqref{eq:new_minmax} in expectation.

Additionally, \eqref{eq:bound_phik} in Theorem~\ref{thm:sapd+_thm} also implies that $\left((\hat{\vx}_1^{\epsilon}, \hat{\vy}_1^{\epsilon}, \hat{\vy}_2^{\epsilon}),(\hat{\vx}_2^{\epsilon}, \hat{\vz}_1^{\epsilon}, \hat{\vz}_2^{\epsilon})\right)$  satisfies
\begin{align*}
    &\mathbb{E}\left[\max_{\vx_2, \vz_1, \vz_2}P_\rho(\hat{\vx}_1^{\epsilon},\vx_2, \hat{\vy}_1^{\epsilon},\hat{\vy}_2^{\epsilon}, \vz_1, \vz_2)\right] \\
    \leq &2\mathbb{E}\left[\max_{\vx_2, \vz_1, \vz_2} P_\rho(\vx_1^{(0)},\vx_2, \vy_1^{(0)},\vy_2^{(0)}, \vz_1, \vz_2)\right]-f^*
    +\epsilon D_2/2.
\end{align*}
On the other hand, by condition ~\eqref{eq:initial_condition_stoch}, it holds that
\begin{align*}
& \mathbb{E} \left[\max_{\vx_2, \vz_1, \vz_2}P_\rho(\vx_1^{(0)},\vx_2, \vy_1^{(0)},\vy_2^{(0)}, \vz_1, \vz_2) \right] \\
= & \mathbb{E}\left[\max_{\vx_2}f(\vx_1^{(0)}, \vx_2, \vy_1^{(0)}, \vy_2^{(0)}) + \rho\big(p(\vx_1^{(0)}, \vy_1^{(0)})-d(\vx_1^{(0)} ,\vy_2^{(0)})\big)\right]\\
\leq &
\mathbb{E}\left[\max_{\vx_2}f(\vx_1^{(0)}, \vx_2, \vy_1^{(0)}, \vy_2^{(0)})\right]+1.
\end{align*}
Combining the two inequalities above gives
\begin{align}
\nonumber
    &\rho^{-1}\left(\mathbb{E}\left[\max_{\vx_2, \vz_1, \vz_2}P_\rho(\hat{\vx}_1^{\epsilon},\vx_2, \hat{\vy}_1^{\epsilon},\hat{\vy}_2^{\epsilon}, \vz_1, \vz_2) \right]
    -f_{\text{low}}\right)\\\nonumber
    \leq&  \rho^{-1}\left(2\mathbb{E}\left[\max_{\vx_2, \vz_1, \vz_2} P_\rho(\vx_1^{(0)},\vx_2, \vy_1^{(0)},\vy_2^{(0)}, \vz_1, \vz_2)\right]
    -f^*-f_{\text{low}}+\epsilon D_2/2\right)\\\nonumber
    \leq&  \rho^{-1}\left( 2\mathbb{E}\left[\max_{\vx_2}f(\vx_1^{(0)}, \vx_2, \vy_1^{(0)}, \vy_2^{(0)})\right]+2
    -f^*-f_{\text{low}}+\epsilon D_2/2\right)\\\label{eq:rho_reciprocal_bound_stoch}
    \leq& O(\epsilon)
\end{align}
where the last inequality is because $L_{\nabla P_1}=O(\rho)=O(\epsilon^{-1})$. By Lemma~\ref{thm:connection}, $(\hat{\vx}_1^{\epsilon}, \hat{\vy}_1^{\epsilon}, \hat{\vy}_2^{\epsilon})$ is a lifted $O(\epsilon)$-KKT point of \eqref{eq:bilevel} in expectation. Because of \eqref{eq:distance_xhatx}, $(\vx_1^{\epsilon}, \vy_1^{\epsilon}, \vy_2^{\epsilon})$ is a nearly lifted $O(\epsilon)$-KKT point of \eqref{eq:bilevel} in expectation.  

Moreover, since $L_{\nabla P_1}=O(\rho)=O(\epsilon^{-1})$, $\delta_P^2=O(\rho^2)=O(\epsilon^{-2})$, we have 
$C_P=\tilde{O}(\epsilon^{-6})$, $K_P = O(\epsilon^{-3})$ and $N_P=\tilde O(\epsilon^{-9})$. Therefore, Algorithm~\ref{alg:first-order-new}
has a complexity of $\tilde O(\epsilon^{-9})$ for finding  a nearly lifted $O(\epsilon)$-KKT point of \eqref{eq:bilevel} in expectation.
\end{proof}

\section{Deterministic and stochastic first-order methods for nonconvex-strongly-concave minimax problem}
\label{sec:wcc_minmax_strong}

In this section, we first consider the oracle complexity for finding an $\epsilon$-primal-dual stationary point of \eqref{eq:minmax} (defined in \eqref{eq:pd_stationary_H}) when $h$ is nonconvex in $x$ but strongly concave in $y$. More specifically, we make the following assumption in addition to Assumption~\ref{assump:h}.
\begin{assumption}\label{assump:h_strong}~
 $h(x,y)$ is $\mu_y$-strongly concave in $y$.
\end{assumption}

To find an $\epsilon$-primal-dual stationary point of \eqref{eq:minmax}, we still apply the procedure \eqref{eq:minmax_prox_h}, but the difference is that 
subproblem \eqref{eq:minmax_prox_h} becomes an instance of \eqref{ea-prob} under the setting of \eqref{eq:mapbarh} and
\begin{equation}
\label{eq:mapbarL_strong}
\begin{aligned}
    L_{\nabla\h} \leftarrow&~L_{\nabla h}+\max\{2L_{\nabla h},\epsilon/(2D_q)\},\\
    \sigma_x\leftarrow&~ L_{\nabla h},\\
    \sigma_y\leftarrow&~\epsilon/(2D_q)+\mu_y.\\
\end{aligned}
\end{equation}
where $D_q$ is defined in \eqref{eq:DpDq} and $\mu_y$ is from Assumption~\ref{assump:h_strong}. 

We can still apply the optimal Algorithm \ref{mmax-alg1} to \eqref{ea-prob} under the setting of \eqref{eq:mapbarh} and \eqref{eq:mapbarL_strong}, which still leads to Algorithm~\ref{alg:first-order} except that we need to set $(L_{\nabla \h},\sigma_x,\sigma_y)$ as in \eqref{eq:mapbarL_strong} in Line 3 of Algorithm~\ref{alg:first-order}. The oracle complexity of this modified Algorithm~\ref{alg:first-order} can be proved after making a small change in the proof of \citep[Theorem 2]{lu2023first} to adapt it for the case when \eqref{eq:minmax} is nonconvex-strongly-concave. In particular, the only change we need is to replace the strong concavity module $\sigma_y$ of subproblem~\eqref{hk} in their proof by $\epsilon/(2D_q)$ to $\epsilon/(2D_q)+\mu_h$. This gives us the following oracle complexity of the modified Algorithm~\ref{alg:first-order}, which covers Theorem~\ref{thm:lu_thm} as a special case when $\mu_y=0$. Since the proof is almost identical to that of Theorem~\ref{thm:lu_thm}, we omit it here.

\begin{theorem}\label{thm:lu_thm_strong}Suppose Assumptions~\ref{assump:h} and \ref{assump:h_strong} hold. Also, suppose  $(L_{\nabla \h},\sigma_x,\sigma_y)$ is set as in \eqref{eq:mapbarL_strong} in Line 3 of Algorithm~\ref{alg:first-order}. Let $H$, $H^*$, $D_p$, $D_q$ and $H_{\text{low}}$ be defined in \eqref{eq:minmax}, \eqref{eq:DpDq} and \eqref{eq:H_low}. Let $\sigma_y$ be defined in \eqref{eq:mapbarL_strong}.
Let 
    \begin{align*}
        \alpha = &~\min \{ 1, \sqrt{8\sigma_y/L_{\nabla h}} \},\\
        \delta = &~(2+ \alpha^{-1})L_{\nabla h}D_p^2 + \max\{ 2\sigma_y, \alpha L_{\nabla h}/4 \} D_q^2,\\
         K = &~\left\lceil 16(\max_{y}H(x^0,y)- H^* + \epsilon D_q/4)L_{\nabla h} \epsilon^{-2} + 32 \hat{\epsilon}^2(1+L_{\nabla h}^2/\sigma_y^2) \epsilon^{-2}-1)\right\rceil_+\\
        C =  &~\frac{4\max\{ \frac{1}{2L_{\nabla h} }, \min \{ \frac{1}{2\sigma_y}, \frac{4}{\alpha L_{\nabla h}} \} \}\Big(\delta + 2\alpha^{-1} (H^* - H_{\text{low}} + \epsilon D_q/4 + L_{\nabla h}D_p^2)\Big)}{[(3L_{\nabla h} + \epsilon/(2D_q))^2/\min \{ L_{\nabla h}, \sigma_y\} + 3 L_{\nabla h} + \epsilon/(2D_q)]^{-2}\hat{\epsilon}^2}\\
         N = &~(\left\lceil 96\sqrt{2}(1+(24L_{\nabla h} + 4\epsilon/D_q)L_{\nabla h}^{-1}) \right\rceil + 2)\max\{ 2, \sqrt{L_{\nabla h}/(2\sigma_y)} \}\\ &~~~~~~~~~~\times \Big((K+1) \times (\log(C))_{+} + K + 1 + 2K\log(K+1) \Big).
    \end{align*}
    Then Algorithm~\ref{alg:first-order} outputs an $\epsilon$-primal-dual stationary point $(x^{\epsilon},y^{\epsilon})$ of ~\eqref{eq:minmax} (e.g., \eqref{eq:pd_stationary_H} holds) in oracle complexity no more than $N$. 
    Moreover, $(x^{\epsilon},y^{\epsilon})$ satisfies
    \begin{align}
    \label{eq:HxK_bound_strong}
    \max_{y}H(x^{\epsilon},y)\leq \max_{y} H(x^0, y)+\epsilon D_q/4+2\hat{\epsilon}^2(L_{\nabla h}^{-1}+L_{\nabla h}/\sigma_y^2).
    \end{align}
\end{theorem}
\begin{remark}
Note that $\sigma_y=\Omega(\mu_y)$. According to this theorem, when $h$ is nonconvex-strongly-concave, the modified Algorithm~\ref{alg:first-order} finds  an $\epsilon$-primal-dual stationary point of ~\eqref{eq:minmax} in $K = O(L_{\nabla h}\epsilon^{-2}+\hat{\epsilon}^2(L_{\nabla h}^2/\mu_y^2)\epsilon^{-2})$ iterations with a total oracle complexity of $N=\tilde{O}(\sqrt{L_{\nabla h}/\mu_y}K)$. 
If $\hat{\epsilon}$ is further set to be $O(1)$ and $L_{\nabla h}=O(1)$ and $\mu_y=O(1)$, the complexity becomes $N=\tilde{O}(\epsilon^{-2})$.
\end{remark}

Next, we consider the stochastic setting where Assumptions~\ref{assump:h}, \ref{assump:stoc_h} and \ref{assump:h_strong} hold, and study the stochastic oracle complexity for finding a nearly $\epsilon$-primal-dual stationary point of \eqref{eq:minmax} defined at the beginning of Section~\ref{sec:SAPD}. Just like the deterministic case above, we apply the procedure \eqref{eq:minmax_prox_h}, where subproblem \eqref{eq:minmax_prox_h} is an instance of \eqref{ea-prob} under the setting of \eqref{eq:mapbarh} and \eqref{eq:mapbarL_strong}. 

Due to the stochastic setting, we still apply Algorithm \ref{alg:sapd} to \eqref{ea-prob} under the setting of \eqref{eq:mapbarh} and \eqref{eq:mapbarL_strong}, which still leads to Algorithm~\ref{alg:first-order-stoch} except that we need to set $(L_{\nabla \h},\sigma_x,\sigma_y)$ as in \eqref{eq:mapbarL_strong} in Line 3 of Algorithm~\ref{alg:first-order-stoch}. The oracle complexity of this modified Algorithm~\ref{alg:first-order-stoch} can be proved after making a small change in the proof of Theorem~\ref{thm:sapd+_thm} to adapt it for the case when \eqref{eq:minmax} is nonconvex-strongly-concave. In particular, the only change we need is to replace the strong concavity module $\sigma_y$ of subproblem~\eqref{hk} in that proof by $\epsilon/(2D_q)$ to $\epsilon/(2D_q)+\mu_h$. This gives us the following oracle complexity of the modified Algorithm~\ref{alg:first-order-stoch}, which covers Theorem~\ref{thm:sapd+_thm} as a special case when $\mu_y=0$. Since the proof is almost identical to that of Theorem~\ref{thm:sapd+_thm}, we omit it here.

\begin{theorem}\label{thm:sapd+_thm_strong}Suppose Assumptions~\ref{assump:h}, \ref{assump:stoc_h} and \ref{assump:h_strong} hold. Also, suppose  $(L_{\nabla \h},\sigma_x,\sigma_y)$ is set as in \eqref{eq:mapbarL_strong} in Line 3 of Algorithm~\ref{alg:first-order-stoch}. Let $H$, $H^*$, $D_p$, and $D_q$ be defined in \eqref{eq:minmax}, \eqref{eq:DpDq} and \eqref{eq:H_low}. Let $\sigma_y$ be defined in \eqref{eq:mapbarL_strong}.
Let
\small
    \begin{align*}
       &~ K = \left\lceil \left(\mathbb{E}\left[\max_{y}H(x^0,y)\right]- H^* + \epsilon D_q/4\right)\hat{\epsilon}^{-2}\right\rceil,\\
       &~ C= \Bigg\lceil
1
+
\ln\!\left(
\frac{
6L_{\nabla h}D_p^2+6\sigma_yD_q^2
}{
\hat{\epsilon}^2
}
\right)
\Bigg[
\left(
L_{\nabla h}+\max\left\{2L_{\nabla h},\frac{\epsilon}{2D_q}\right\}
\right)
\left(
\frac{2}{\sqrt{L_{\nabla h}\sigma_y}}
+\frac{3}{L_{\nabla h}}
+\frac{12}{\sigma_y}
\right) \\
&\hspace{6cm}
+
\frac{\delta^2}{\hat{\epsilon}^2}
\left(
\frac{21}{L_{\nabla h}}
+
\frac{180}{\sigma_y}
\right)
\Bigg]
\Bigg\rceil ,
\\
        &~ N = KC.
    \end{align*}
\normalsize
    Then Algorithm~\ref{alg:first-order-stoch} terminates and outputs a solution $(x^{\epsilon},y^{\epsilon})$ such that there exists a stochastic solution $(\hat{x}^{\epsilon},\hat{y}^{\epsilon})$ such that 
    \begin{align}
    \label{eq:distance_stoc_h_storng}
    \mathbb{E}\|(x^{\epsilon},y^{\epsilon})-(\hat{x}^{\epsilon},\hat{y}^{\epsilon})\|\leq \max\left\{\frac{\hat{\epsilon}}{\sqrt{L_{\nabla h}}},\frac{\hat{\epsilon}}{\sqrt{\sigma_y}}\right\}
    \end{align}
    and
    $(\hat{x}^{\epsilon},\hat{y}^{\epsilon})$ is a $ \max\left\{\sqrt{2L_{\nabla h}}\hat{\epsilon},\frac{\epsilon}{2}\right\}$-primal-dual stationary point of \eqref{eq:minmax} in expectation. As a consequence, Algorithm~\ref{alg:first-order-stoch} finds a nearly $\max\left\{\sqrt{2L_{\nabla h}}\hat{\epsilon},\frac{\epsilon}{2},\frac{\hat{\epsilon}}{\sqrt{L_{\nabla h}}},\frac{\hat{\epsilon}}{\sqrt{\sigma_y}}\right\}$-primal-dual stationary point $(x^{\epsilon},y^{\epsilon})$ of ~\eqref{eq:minmax} in expectation in complexity no more than $N$. Moreover, \eqref{eq:bound_phik} holds.
\end{theorem}

\begin{remark}
Note that $\sigma_y=\Omega(\mu_y)=\Omega(1)$. 
According to this theorem, when $L_{\nabla h}=O(1)$, the modified Algorithm~\ref{alg:first-order-stoch} finds  a nearly $O\left(\max\left\{\epsilon,\hat{\epsilon}\right\}\right)$-primal-dual stationary point of ~\eqref{eq:minmax} in $K = O(\hat\epsilon^{-2})$ iterations with a total stochastic oracle complexity of $N=\tilde{O}((1+\delta^2\epsilon^{-1}\hat\epsilon^{-2})K)$. 
If $\hat{\epsilon}$ is further set to be $\Theta\left(\epsilon\right)$ and $L_{\nabla h}=O(1)$, the complexity becomes $N=\tilde{O}((\epsilon^{-2}+\delta^2\epsilon^{-5}))$.
\end{remark}

\section{Complexity under strong convexity and strong concavity}
\label{sec:deterministic_strongly}
In this section, we present additional complexity results for Algorithm~\ref{alg:first-order-new} in both the deterministic and stochastic cases when $f$ and $\tilde {f}$ further satisfy strong convexity and strong concavity in addition to Assumption~\ref{assump:dom}.

\subsection{Reduced complexity in deterministic case}
\label{sec:deterministic_strong}

We first state the following assumption. 
\begin{assumption}\label{assump:dom_strong}
$f_1(\vx_1, \vx_2, \vy_1, \vy_2)$ is $\mu_{\vx_2}$-strongly concave in $\vx_2$. $\tilde{f}_1(\vx_1, \vy_1, \vy_2)$ is $\mu_{\vy_1}$-strongly convex in $\vy_1$ and $\mu_{\vy_2}$-strongly concave in $\vy_2$. 
\end{assumption}

\begin{theorem}\label{thm:minmax_strong}
Suppose Assumptions~\ref{assump:dom} and ~\ref{assump:dom_strong}  hold and $(\vx_1^{(0)},\vy_1^{(0)}, \vy_2^{(0)})$ in Algorithm~\ref{alg:first-order-new} satisfies \eqref{eq:initial_condition}. Also, suppose $\sigma_y$ in Line 5 of Algorithm~\ref{alg:first-order-new} is set to $\epsilon/(2D_2)+\min\{ \mu_{\vx_2}, \rho \mu_{\vy_1}, \rho \mu_{\vy_2} \}$ and $\hat\epsilon$ in Line 3 is set to $\epsilon^{0.5}$. Algorithm~\ref{alg:first-order-new} outputs a lifted $O(\epsilon)$-KKT point $(\vx_1^{\epsilon}, \vy_1^{\epsilon}, \vy_2^{\epsilon})$ of~\eqref{eq:bilevel} with oracle complexity $\tilde O(\epsilon^{-3.5})$.
    \end{theorem}

\begin{proof}[Proof of Theorem~\ref{thm:minmax_strong}]
Due to Assumption~\ref{assump:dom_strong}, 
$P_1(\vx_1,\vx_2, \vy_1,\vy_2, \vz_1, \vz_2)$ is $\min\{ \mu_{\vx_2}, \rho \mu_{\vy_1}, \rho \mu_{\vy_2} \}$-strongly concave in $(\vx_2, \vz_1, \vz_2)$.
By the same argument as in the proof of Theorem~\ref{thm:minmax}, we can still show that the deterministic case of Algorithm~\ref{alg:first-order-new} 
with $\sigma_y$ in Line 5 set to $\epsilon/(2D_2)+\min\{ \mu_{\vx_2}, \rho \mu_{\vy_1}, \rho \mu_{\vy_2} \}$ is exactly Algorithm~\ref{alg:first-order} with
$(L_{\nabla \h},\sigma_x,\sigma_y)$  in Line 3 set as in \eqref{eq:mapbarL_strong}. Similarly, when \eqref{eq:bilevel} satisfies Assumptions~\ref{assump:dom} and ~\ref{assump:dom_strong}  and $\rho_1$ and $\rho_2$ are set as in Algorithm~\ref{alg:first-order-new} (i.e., as in \eqref{eq:rho}),  \eqref{eq:minmax_prox} is an instance of \eqref{ea-prob} satisfying Assumption~\eqref{assump:barh} with
\begin{equation}
\label{eq:mapPrhoL_strong}
\begin{aligned}
    L_{\nabla\h} \leftarrow&~L_{\nabla f} + 2\rho L_{\nabla \tilde{f}}+\max\{\rho_1,\rho_2\}, \\
    \sigma_x\leftarrow&~L_{\nabla f}+ 2\rho L_{\nabla \tilde{f}},\\
    \sigma_y\leftarrow&~\epsilon/(2D_2)+\min\{ \mu_{\vx_2}, \rho \mu_{\vy_1}, \rho \mu_{\vy_2} \}.
\end{aligned}
\end{equation}
Therefore, we can directly apply Theorem~\ref{thm:lu_thm_strong} and Lemma~\ref{thm:connection} to prove Theorem~\ref{thm:minmax_strong}. While doing so, we just need to set $\hat\epsilon=\epsilon^{0.5}$ in Theorem~\ref{thm:lu_thm_strong} according to the new Line 3 of Algorithm~\ref{alg:first-order-new}.

Let $\sigma_y$, $f^*$, $D_1$, $D_2$ and $f_{\text{low}}$ be defined in \eqref{eq:mapPrhoL_strong}, \eqref{eq:old_minmax}, \eqref{eq:D1}, \eqref{eq:D2} and \eqref{eq:flow}.  Let
    \begin{align*}
        L_{\nabla P_1} = &~L_{\nabla f_1}+2\rho L_{\nabla\tilde{f}_1}, \\
        \alpha_P =&~ \min \{ 1, \sqrt{8\sigma_y/ L_{\nabla P_1}} \},\\
        \delta_P =&~ (2+ \alpha_P^{-1})L_{\nabla P_1}D_1^2+ \max\{ 2\sigma_y, \alpha_P L_{\nabla P_1}/4 \}D_2^2,\\
         K_P = &~\left\lceil 
        \begin{array}{ll}
        16(\max_{\vx_2}f(\vx_1^{(0)}, \vx_2, \vy_1^{(0)}, \vy_2^{(0)})+1- f_{\text{low}} )L_{\nabla P_1} \epsilon^{-2}\\
        + 32 \epsilon(1+\sigma_y^{-2}L_{\nabla P_1}^2)\epsilon^{-2} -1    \end{array}\right\rceil_+\\ 
       C_P = &~  \Big(\frac{4\max\{ \frac{1}{2L_{\nabla P_1} }, \min \{ \frac{1}{2\sigma_y}, \frac{4}{\alpha_P L_{\nabla P_1}} \} \}}{[(3L_{\nabla P_1} + \epsilon/(2D_2))^2/\min \{ L_{\nabla P_1}, \sigma_y\} + 3 L_{\nabla P_1}+ \epsilon/(2D_2) ]^{-2}\epsilon}\Big)\\
       &~~~~~~~~~~~~\times\Big(\delta_P + 2\alpha_P^{-1} (f^* - f_{\text{low}} + L_{\nabla P_1}D_1^2)\Big), \nonumber\\
        N_P = &~(\left\lceil 96\sqrt{2}(1+(24L_{\nabla P_1} + 4\epsilon/D_2)L_{\nabla P_1}^{-1}) \right\rceil + 2)\max\{ 2, \sqrt{L_{\nabla P_1}/(2\sigma_y)} \}\\ &~~~~~~~~~~\times \Big((K_P+1) \times (\log(C_P))_{+} + K_P + 1 + 2K_P\log(K_P+1) \Big).
    \end{align*}
    
Recall that $\hat{\epsilon}$ is set to be $\epsilon^{0.5}$ in the modified Algorithm~\ref{alg:first-order-new}.  Applying Theorem~\ref{thm:lu_thm_strong} under this setting with $\hat{\epsilon}=\epsilon^{0.5}$, 
we can  show that Algorithm~\ref{alg:first-order-new} outputs an $\epsilon$-primal-dual stationary point $\left((\vx_1^{\epsilon}, \vy_1^{\epsilon}, \vy_2^{\epsilon}),(\vx_2^{\epsilon}, \vz_1^{\epsilon}, \vz_2^{\epsilon})\right)$ of~\eqref{eq:new_minmax} in oracle complexity no more than $N_P$ above. Since $\sigma_y=\Omega(\mu_{\vx_2})=\Omega(1)$ and  $L_{\nabla P_1}=O(\rho)=O(\epsilon^{-1})$, we have $K_P = O(\epsilon^{-3})$ and $N_P=\tilde O(\epsilon^{-3.5})$.

Moreover, \eqref{eq:HxK_bound_strong} in Theorem~\ref{thm:lu_thm_strong} also implies that $\left((\vx_1^{\epsilon}, \vy_1^{\epsilon}, \vy_2^{\epsilon}),(\vx_2^{\epsilon}, \vz_1^{\epsilon}, \vz_2^{\epsilon})\right)$  satisfies
\begin{align*}
    \max_{\vx_2, \vz_1, \vz_2}P_\rho(\vx_1^{\epsilon},\vx_2, \vy_1^{\epsilon},\vy_2^{\epsilon}, \vz_1, \vz_2) \leq& \max_{\vx_2, \vz_1, \vz_2} P_\rho(\vx_1^{(0)},\vx_2, \vy_1^{(0)},\vy_2^{(0)}, \vz_1, \vz_2)\\
    &+\epsilon D_2/4+2\epsilon(L_{\nabla P_1}^{-1}+L_{\nabla P_1}/\sigma_y^2).
\end{align*}
Combining the inequality above with \eqref{eq:upperbound_Hx0} gives
\begin{align*}
    &\rho^{-1}\left(\max_{\vx_2, \vz_1, \vz_2}P_\rho(\vx_1^{\epsilon},\vx_2, \vy_1^{\epsilon},\vy_2^{\epsilon}, \vz_1, \vz_2) 
    -f_{\text{low}}\right)\\
    \leq&  \rho^{-1}\left( \max_{\vx_2, \vz_1, \vz_2} P_\rho(\vx_1^{(0)},\vx_2, \vy_1^{(0)},\vy_2^{(0)}, \vz_1, \vz_2)
    -f_{\text{low}}+\epsilon D_2/4+2\epsilon(L_{\nabla P_1}^{-1}+L_{\nabla P_1}/\sigma_y^2\right)\\
    \leq&  \rho^{-1}\left( \max_{\vx_2}f(\vx_1^{(0)}, \vx_2, \vy_1^{(0)}, \vy_2^{(0)})+1
    -f_{\text{low}}+\epsilon D_2/4+2\epsilon(L_{\nabla P_1}^{-1}+L_{\nabla P_1}/\sigma_y^2\right)\\
    \leq& O(\epsilon),
\end{align*}
where the last inequality is because $\sigma_y=\Omega(\mu_{\vx_2})=\Omega(1)$  and $L_{\nabla P_1}=O(\rho)=O(\epsilon^{-1})$. By Lemma~\ref{thm:connection}, $(\vx_1^{\epsilon}, \vy_1^{\epsilon}, \vy_2^{\epsilon})$ is a lifted $O(\epsilon)$-KKT point of \eqref{eq:bilevel}.
\end{proof}

We further consider the assumption below. 
\begin{assumption}\label{assump:dom_strong_nox2}The upper-level maximization variable $\vx_2$ is absent; that is,
$f_1(\vx_1, \vx_2,\vy_1, \vy_2)=f_1(\vx_1, \vy_1, \vy_2)$ and $f_3\equiv 0$.
Furthermore, $\tilde{f}_1(\vx_1,\vy_1,\vy_2)$ is
$\mu_{\vy_1}$-strongly convex in $\vy_1$ and
$\mu_{\vy_2}$-strongly concave in $\vy_2$.
\end{assumption}

\begin{theorem}\label{thm:minmax_strong_nox2}
Suppose Assumptions~\ref{assump:dom} and ~\ref{assump:dom_strong_nox2}  hold and $(\vx_1^{(0)},\vy_1^{(0)}, \vy_2^{(0)})$ in Algorithm~\ref{alg:first-order-new} satisfies \eqref{eq:initial_condition}. Also, suppose $\sigma_y$ in Line 5 of Algorithm~\ref{alg:first-order-new} is set to $\epsilon/(2D_2)+\min\{ \rho \mu_{\vy_1}, \rho \mu_{\vy_2} \}$ and $\hat\epsilon$ in Line 3 is set to $1$. Algorithm~\ref{alg:first-order-new} outputs a lifted $O(\epsilon)$-KKT point $(\vx_1^{\epsilon}, \vy_1^{\epsilon}, \vy_2^{\epsilon})$ of~\eqref{eq:bilevel} with oracle complexity $\tilde O(\epsilon^{-3})$.   
\end{theorem}

\begin{proof}[Proof of Theorem~\ref{thm:minmax_strong_nox2}]
Under Assumption~\ref{assump:dom_strong_nox2}, a lifted $\epsilon$-KKT point is still defined as in Definition~\ref{def:epsilon_KKT} except that the conditions related to $\vx_2$ are removed. Due to Assumption~\ref{assump:dom_strong_nox2}, 
$P_1(\vx_1,\vx_2, \vy_1,\vy_2, \vz_1, \vz_2)$ is $\min\{  \rho \mu_{\vy_1}, \rho \mu_{\vy_2} \}$-strongly concave in $(\vz_1, \vz_2)$. Then, the proof steps are almost the same as the proof of Theorem~\ref{thm:minmax_strong}. The only difference is that \eqref{eq:minmax_prox} is an instance of \eqref{ea-prob} satisfying Assumption~\eqref{assump:barh} with
\begin{equation}
\label{eq:mapPrhoL_strong_nox2}
\begin{aligned}
    L_{\nabla\h} \leftarrow&~L_{\nabla f} + 2\rho L_{\nabla \tilde{f}}+\max\{\rho_1,\rho_2\}, \\
    \sigma_x\leftarrow&~L_{\nabla f}+ 2\rho L_{\nabla \tilde{f}},\\
    \sigma_y\leftarrow&~\epsilon/(2D_2)+\min\{  \rho \mu_{\vy_1}, \rho \mu_{\vy_2} \}.
\end{aligned}
\end{equation}
Let 
\begin{equation}
\label{eq:D2_nox2}
D_2 := \max\{\|(\vz_1,\vz_2) - (\vz_1',\vz_2')\| \;|\; (\vz_1,\vz_2),(,\vz_1',\vz_2') \in  \mathcal{Y}_1 \times \mathcal{Y}_2\}.
\end{equation}
Let $\sigma_y$, $f^*$, $D_1$, and $f_{\text{low}}$ be defined in \eqref{eq:mapPrhoL_strong_nox2}, \eqref{eq:old_minmax}, \eqref{eq:D1}, and \eqref{eq:flow}. Let
   \begin{align*}
        L_{\nabla P_1} = &~L_{\nabla f_1}+2\rho L_{\nabla\tilde{f}_1}, \\
        \alpha_P =&~ \min \{ 1, \sqrt{8\sigma_y/ L_{\nabla P_1}} \},\\
        \delta_P =&~ (2+ \alpha_P^{-1})L_{\nabla P_1}D_1^2+ \max\{ 2\sigma_y, \alpha_P L_{\nabla P_1}/4 \}D_2^2,\\
         K_P = &~\left\lceil 
        \begin{array}{ll}
        16(\max_{\vx_2}f(\vx_1^{(0)}, \vx_2, \vy_1^{(0)}, \vy_2^{(0)})+1- f_{\text{low}} )L_{\nabla P_1} \epsilon^{-2}\\
        + 32 (1+\sigma_y^{-2}L_{\nabla P_1}^2)\epsilon^{-2} -1    \end{array}\right\rceil_+\\ 
       C_P = &~  \Big(\frac{4\max\{ \frac{1}{2L_{\nabla P_1} }, \min \{ \frac{1}{2\sigma_y}, \frac{4}{\alpha_P L_{\nabla P_1}} \} \}}{[(3L_{\nabla P_1} + \epsilon/(2D_2))^2/\min \{ L_{\nabla P_1}, \sigma_y\} + 3 L_{\nabla P_1}+ \epsilon/(2D_2) ]^{-2}}\Big)\\
       &~~~~~~~~~~~~\times\Big(\delta_P + 2\alpha_P^{-1} (f^* - f_{\text{low}} + L_{\nabla P_1}D_1^2)\Big), \nonumber\\
        N_P = &~(\left\lceil 96\sqrt{2}(1+(24L_{\nabla P_1} + 4\epsilon/D_2)L_{\nabla P_1}^{-1}) \right\rceil + 2)\max\{ 2, \sqrt{L_{\nabla P_1}/(2\sigma_y)} \}\\ &~~~~~~~~~~\times \Big((K_P+1) \times (\log(C_P))_{+} + K_P + 1 + 2K_P\log(K_P+1) \Big).
    \end{align*}
Recall that $\hat{\epsilon}$ is set to be $1$ in the modified Algorithm~\ref{alg:first-order-new}.  Applying Theorem~\ref{thm:lu_thm_strong} under this setting with $\hat{\epsilon}=1$, 
we can  show that Algorithm~\ref{alg:first-order-new} outputs an $\epsilon$-primal-dual stationary point $\left((\vx_1^{\epsilon}, \vy_1^{\epsilon}, \vy_2^{\epsilon}),(\vx_2^{\epsilon}, \vz_1^{\epsilon}, \vz_2^{\epsilon})\right)$ of~\eqref{eq:new_minmax} in oracle complexity no more than $N_P$ above. Since $\sigma_y=\Omega(\rho)=\Omega(\epsilon^{-1})$ and  $L_{\nabla P_1}=O(\rho)=O(\epsilon^{-1})$, we have $K_P = O(\epsilon^{-3})$ and $N_P=\tilde O(\epsilon^{-3})$.

Moreover, \eqref{eq:HxK_bound_strong} in Theorem~\ref{thm:lu_thm_strong} also implies that $\left((\vx_1^{\epsilon}, \vy_1^{\epsilon}, \vy_2^{\epsilon}),(\vx_2^{\epsilon}, \vz_1^{\epsilon}, \vz_2^{\epsilon})\right)$  satisfies
\begin{align*}
    \max_{\vx_2, \vz_1, \vz_2}P_\rho(\vx_1^{\epsilon},\vx_2, \vy_1^{\epsilon},\vy_2^{\epsilon}, \vz_1, \vz_2) \leq& \max_{\vx_2, \vz_1, \vz_2} P_\rho(\vx_1^{(0)},\vx_2, \vy_1^{(0)},\vy_2^{(0)}, \vz_1, \vz_2)\\
    &+\epsilon D_2/4+2(L_{\nabla P_1}^{-1}+L_{\nabla P_1}/\sigma_y^2).
\end{align*}
Combining the inequality above with \eqref{eq:upperbound_Hx0} gives
\begin{align*}
    &\rho^{-1}\left(\max_{\vx_2, \vz_1, \vz_2}P_\rho(\vx_1^{\epsilon},\vx_2, \vy_1^{\epsilon},\vy_2^{\epsilon}, \vz_1, \vz_2) 
    -f_{\text{low}}\right)\\
    \leq&  \rho^{-1}\left( \max_{\vx_2, \vz_1, \vz_2} P_\rho(\vx_1^{(0)},\vx_2, \vy_1^{(0)},\vy_2^{(0)}, \vz_1, \vz_2)
    -f_{\text{low}}+\epsilon D_2/4+2\epsilon(L_{\nabla P_1}^{-1}+L_{\nabla P_1}/\sigma_y^2\right)\\
    \leq&  \rho^{-1}\left( \max_{\vx_2}f(\vx_1^{(0)}, \vx_2, \vy_1^{(0)}, \vy_2^{(0)})+1
    -f_{\text{low}}+\epsilon D_2/4+2(L_{\nabla P_1}^{-1}+L_{\nabla P_1}/\sigma_y^2\right)\\
    \leq& O(\epsilon),
\end{align*}
where the last inequality is because $\sigma_y=\Omega(\mu_{\vx_2})=\Omega(\rho)=\Omega(\epsilon^{-1})$  and $L_{\nabla P_1}=O(\rho)=O(\epsilon^{-1})$. By Lemma~\ref{thm:connection}, $(\vx_1^{\epsilon}, \vy_1^{\epsilon}, \vy_2^{\epsilon})$ is a lifted $O(\epsilon)$-KKT point of \eqref{eq:bilevel}.
\end{proof}

\subsection{Reduced complexity in stochastic case}
\label{sec:stochastic_strong}

\begin{theorem}\label{thm:minmax_stoch_concavity}
Suppose Assumptions~\ref{assump:dom}, \ref{assump:stoc} and \ref{assump:dom_strong} hold, $(\vx_1^{(0)},\vy_1^{(0)}, \vy_2^{(0)})$ in Algorithm~\ref{alg:first-order-new} satisfies
\eqref{eq:initial_condition_stoch}
and $K$ in Algorithm~\ref{alg:first-order-new}  satisfies 
$
K= \left\lceil \left(\mathbb{E}\left[\max_{\vx_2}f(\vx_1^{(0)}, \vx_2, \vy_1^{(0)}, \vy_2^{(0)})\right]+1- f^*+\epsilon D_2/4\right)\epsilon^{-3}\right\rceil.
$
Also, suppose $\sigma_y$ in Line 5 of Algorithm~\ref{alg:first-order-new} is set to $\epsilon/(2D_2)+\min\{ \mu_{\vx_2}, \rho \mu_{\vy_1}, \rho \mu_{\vy_2} \}$. Algorithm~\ref{alg:first-order-new} terminates and outputs a nearly lifted $O(\epsilon)$-KKT point  of~\eqref{eq:bilevel} in expectation in stochastic oracle complexity $\tilde{O}(\epsilon^{-8})$.
\end{theorem}

\begin{proof}[Proof of Theorem~\ref{thm:minmax_stoch_concavity}]
Due to Assumption~\ref{assump:dom_strong}, 
$P_1(\vx_1,\vx_2, \vy_1,\vy_2, \vz_1, \vz_2)$ is $\min\{ \mu_{\vx_2}, \rho \mu_{\vy_1}, \rho \mu_{\vy_2} \}$-strongly concave in $(\vx_2, \vz_1, \vz_2)$.
By the same argument as in the proof of Theorem~\ref{thm:minmax_stoch}, we can still show that the stochastic case of Algorithm~\ref{alg:first-order-new} 
with $\sigma_y$ in Line 5 set to $\epsilon/(2D_2)+\min\{ \mu_{\vx_2}, \rho \mu_{\vy_1}, \rho \mu_{\vy_2} \}$ is exactly Algorithm~\ref{alg:first-order-stoch} with
$(L_{\nabla \h},\sigma_x,\sigma_y)$  in Line 3 set as in \eqref{eq:mapbarL_strong}. Similarly, when \eqref{eq:bilevel} satisfies Assumptions~\ref{assump:dom} and ~\ref{assump:dom_strong}  and $\rho_1$ and $\rho_2$ are set as in Algorithm~\ref{alg:first-order-new} (i.e., as in \eqref{eq:rho}),  \eqref{eq:minmax_prox} is an instance of \eqref{ea-prob} satisfying Assumption~\eqref{assump:barh} with \eqref{eq:mapPrhoL_strong} and satisfying Assumption~\eqref{assump:stoc_h} with \eqref{eq:delta_P}.
Therefore, we can directly apply Theorem~\ref{thm:sapd+_thm_strong} and Lemma~\ref{thm:connection} to prove Theorem~\ref{thm:minmax_stoch_concavity}. While doing so, we just need to set $\hat\epsilon=\epsilon^{1.5}$ in Theorem~\ref{thm:sapd+_thm_strong} according to the  Line 3 of Algorithm~\ref{alg:first-order-new}.

Let $\sigma_y$, $f^*$, $D_1$, $D_2$ and $f_{\text{low}}$ be defined in \eqref{eq:mapPrhoL_strong}, \eqref{eq:old_minmax}, \eqref{eq:D1}, \eqref{eq:D2} and \eqref{eq:flow}. 
Let
\small
    \begin{align*}
        &L_{\nabla P_1} = L_{\nabla f_1}+2\rho L_{\nabla\tilde{f}_1},\\
        &\delta_{P}^2=3\delta_{f}^2 + 6\rho^2\delta_{\tilde{f}}^2,\\\label{eq:sapd_TP}
        & C_P= \Bigg\lceil
1+
\ln\!\left(
\frac{
6L_{\nabla P_1}D_1^2+6\sigma_y D_2^2
}{
\hat{\epsilon}^2
}
\right)
\Bigg[
\left(
L_{\nabla P_1}+\max\left\{2L_{\nabla P_1},\frac{\epsilon}{2D_2}\right\}
\right)
\left(
\frac{2}{\sqrt{L_{\nabla P_1}\sigma_y}}
+\frac{3}{L_{\nabla P_1}}
+\frac{12}{\sigma_y}
\right) \\
&\hspace{6cm}
+
\frac{\delta_P^2}{\hat{\epsilon}^2}
\left(
\frac{21}{L_{\nabla P_1}}
+
\frac{180}{\sigma_y}
\right)
\Bigg]
\Bigg\rceil\\
       &K_P = \left\lceil \left(\mathbb{E}\left[\max_{\vx_2}f(\vx_1^{(0)}, \vx_2, \vy_1^{(0)}, \vy_2^{(0)})\right]+1- f^* + \epsilon D_2/4\right)\hat{\epsilon}^{-2}\right\rceil\\
        &N_P = C_PK_P.
    \end{align*}
\normalsize

Note that \eqref{eq:initial_bound_stoch} still holds because of \eqref{eq:initial_condition_stoch} and the fact that $\rho=\frac{1}{\epsilon}$. 
 Applying Theorem~\ref{thm:sapd+_thm_strong} under this setting with
$\max_{y}H(x^0,y)$ in $K$ in Theorem~\ref{thm:sapd+_thm_strong}  replaced by 
its upper bound \eqref{eq:initial_bound_f_stoch}, we can  show that, in complexity no more than $N_P$, Algorithm~\ref{alg:first-order-new} Algorithm~\ref{alg:first-order-new} outputs $\left((\vx_1^{\epsilon}, \vy_1^{\epsilon}, \vy_2^{\epsilon}),(\vx_2^{\epsilon}, \vz_1^{\epsilon}, \vz_2^{\epsilon})\right)$ such that there exists a stochastic solution $\left((\hat{\vx}_1^{\epsilon}, \hat{\vy}_1^{\epsilon}, \hat{\vy}_2^{\epsilon}),(\hat{\vx}_2^{\epsilon}, \hat{\vz}_1^{\epsilon}, \hat{\vz}_2^{\epsilon})\right)$ such that 
\begin{equation*}
\mathbb{E}\left\|
(\vx_1^{\epsilon}, \vy_1^{\epsilon}, \vy_2^{\epsilon})
-
(\hat{\vx}_1^{\epsilon}, \hat{\vy}_1^{\epsilon}, \hat{\vy}_2^{\epsilon})
\right\|
\leq
\max\left\{
\frac{\hat{\epsilon}}{\sqrt{L_{\nabla P_1}}},
\frac{\hat{\epsilon}}{\sqrt{\sigma_y}}
\right\}
\end{equation*}
and   $\left((\hat{\vx}_1^{\epsilon}, \hat{\vy}_1^{\epsilon}, \hat{\vy}_2^{\epsilon}),(\hat{\vx}_2^{\epsilon}, \hat{\vz}_1^{\epsilon}, \hat{\vz}_2^{\epsilon})\right)$ is a $ \max\left\{\sqrt{2L_{\nabla P_1}}\hat{\epsilon},\frac{\epsilon}{2}\right\}$-primal-dual stationary point of \eqref{eq:new_minmax} in expectation.

Since $\hat{\epsilon}=\epsilon^{1.5}$, $\sigma_y=\Omega(\mu_{\vx_2})=\Omega(1)$ and $L_{\nabla P_1}=O(\rho)=O(\epsilon^{-1})$, we have
$$
\max\left\{\sqrt{2L_{\nabla P_1}}\hat{\epsilon},\frac{\epsilon}{2},\frac{\hat{\epsilon}}{\sqrt{L_{\nabla P_1}}},\frac{\hat{\epsilon}}{\sqrt{\sigma_y}}\right\}
=O(\epsilon)
$$
meaning that \eqref{eq:distance_xhatx} holds and $\left((\hat{\vx}_1^{\epsilon}, \hat{\vy}_1^{\epsilon}, \hat{\vy}_2^{\epsilon}),(\hat{\vx}_2^{\epsilon}, \hat{\vz}_1^{\epsilon}, \hat{\vz}_2^{\epsilon})\right)$ is an $O(\epsilon)$-primal-dual stationary point of \eqref{eq:new_minmax} in expectation.

Additionally, \eqref{eq:bound_phik} is also ensured by Theorem~\ref{thm:sapd+_thm_strong}, which implies \eqref{eq:rho_reciprocal_bound_stoch} by the same argument used in the end of the proof of Theorem~\ref{thm:minmax_stoch}.

Moreover, since $\sigma_y=\Omega(\mu_{\vx_2})=\Omega(1)$, $L_{\nabla P_1}=O(\rho)=O(\epsilon^{-1})$, $\delta_P^2=O(\rho^2)=O(\epsilon^{-2})$, we have 
$C_P=\tilde{O}(\epsilon^{-5})$, $K_P = O(\epsilon^{-3})$ and $N_P=\tilde O(\epsilon^{-8})$. Therefore, Algorithm~\ref{alg:first-order-new}
has a complexity of $\tilde O(\epsilon^{-8})$ for finding  a nearly lifted $O(\epsilon)$-KKT point of \eqref{eq:bilevel} in expectation.
\end{proof}

\begin{theorem}\label{thm:minmax_stoch_nox2}
Suppose Assumptions~\ref{assump:dom}, \ref{assump:stoc} and \ref{assump:dom_strong_nox2} hold, $(\vx_1^{(0)},\vy_1^{(0)}, \vy_2^{(0)})$ in Algorithm~\ref{alg:first-order-new} satisfies
\eqref{eq:initial_condition_stoch}
and $K$ in Algorithm~\ref{alg:first-order-new}  satisfies 
$
K= \left\lceil \left(\mathbb{E}\left[\max_{\vx_2}f(\vx_1^{(0)}, \vx_2, \vy_1^{(0)}, \vy_2^{(0)})\right]+1- f^*+\epsilon D_2/4\right)\epsilon^{-3}\right\rceil.
$
Also, suppose $\sigma_y$ in Line 5 of Algorithm~\ref{alg:first-order-new} is set to $\epsilon/(2D_2)+\min\{ \rho \mu_{\vy_1}, \rho \mu_{\vy_2} \}$. Algorithm~\ref{alg:first-order-new} terminates and outputs a nearly lifted $O(\epsilon)$-KKT point  of~\eqref{eq:bilevel} in expectation in stochastic oracle complexity $\tilde{O}(\epsilon^{-7})$.
\end{theorem}

\begin{proof}[Proof of Theorem~\ref{thm:minmax_stoch_nox2}]
Under Assumption~\ref{assump:dom_strong_nox2}, a lifted $\epsilon$-KKT point is still defined as in Definition~\ref{def:epsilon_KKT} except that the conditions related to $\vx_2$ are removed. Due to Assumption~\ref{assump:dom_strong_nox2}, 
$P_1(\vx_1,\vx_2, \vy_1,\vy_2, \vz_1, \vz_2)$ is $\min\{  \rho \mu_{\vy_1}, \rho \mu_{\vy_2} \}$-strongly concave in $(\vz_1, \vz_2)$. Then, the proof steps are almost the same as the proof of Theorem~\ref{thm:minmax_stoch_concavity}. The only difference is that \eqref{eq:minmax_prox} is an instance of \eqref{ea-prob} satisfying 
Assumption~\eqref{assump:barh} with \eqref{eq:mapPrhoL_strong_nox2} and satisfying Assumption~\eqref{assump:stoc_h} with \eqref{eq:delta_P}.

In this case, $\sigma_y=\Omega(\rho)=\Omega(\epsilon^{-1})$. Moreover, since $L_{\nabla P_1}=O(\rho)=O(\epsilon^{-1})$, $\delta_P^2=O(\rho^2)=O(\epsilon^{-2})$, we have 
$C_P=\tilde{O}(\epsilon^{-4})$, $K_P = O(\epsilon^{-3})$ and $N_P=\tilde O(\epsilon^{-7})$. Therefore, Algorithm~\ref{alg:first-order-new}
has a complexity of $\tilde O(\epsilon^{-7})$ for finding  a nearly lifted $O(\epsilon)$-KKT point of \eqref{eq:bilevel} in expectation.
\end{proof}

\section{Detailed comparison with related penalty methods}
\label{sec:penalty_comparison}

\begin{table}[t]
    \centering
    \footnotesize
    \renewcommand{\arraystretch}{1.25}
    \setlength{\tabcolsep}{4pt}
    \begin{tabularx}{\linewidth}{
        @{}p{2.0cm}p{1.6cm}
        >{\raggedright\arraybackslash}X
        >{\raggedright\arraybackslash}X@{}
    }
    \toprule
    \textbf{Work}
    & \textbf{Lower level}
    & \textbf{Key assumptions}
    & \textbf{Stationarity guarantee} \\
    \midrule

    Jiang et al.~\citep{jiang2025beyond}
    & Unconstrained min
    & PL condition on perturbed lower level;
      $(\delta,\alpha)$-flat upper objective
    & $\epsilon$-KKT of gradient-constrained reformulation \\

    \midrule

    Kwon et al.~\citep{kwon2024penalty}
    & Constrained min
    & Proximal error bound; solution-path regularity
      for hyperobjective stationarity
    & $\epsilon$-KKT of value-function reformulation;
      hyperobjective stationarity under additional regularity \\

    \midrule

    Kwon et al.~\citep{kwon2023fully}
    & Unconstrained min
    & Strongly convex lower level
    & $\epsilon$-stationarity of original hyperobjective \\

    \midrule

    \textbf{Our work}
    & \textbf{Min--max}
    & \textbf{Convex--concave lower level}
    & \textbf{Lifted $\epsilon$-KKT of value-function
      reformulation} \\

    \bottomrule
    \end{tabularx}
    \caption{Comparison with related penalty-based
    bilevel optimization methods.}
    \label{tab:penalty_comparison}
\end{table}

Table~\ref{tab:penalty_comparison} compares our method with
three closely related penalty-based approaches in terms of
their lower-level problem structures, key assumptions, and
stationarity guarantees.

\textbf{Oracle models and penalty parameters.}
Jiang et al.~\citep{jiang2025beyond} use deterministic
gradients with a penalty parameter of order
$\Omega(\epsilon^{-(2-\alpha)})$.
Kwon et al.~\citep{kwon2023fully} use stochastic gradients
and a polynomially increasing penalty parameter that
eventually reaches $\Omega(\epsilon^{-1})$.
Kwon et al.~\citep{kwon2024penalty} consider both
deterministic and stochastic oracles, with an improved
stochastic complexity bound under mean-squared smoothness.
Both their method and ours use a penalty parameter of order
$\Omega(\epsilon^{-1})$, while our approach supports both
deterministic and stochastic gradients.

In summary, unlike the above methods, which primarily
address lower-level minimization problems, our framework
directly handles convex-concave minimax lower-level
problems without requiring strong convexity or a
Polyak--Lojasiewicz condition, while providing
deterministic and stochastic first-order complexity
guarantees for approximate KKT points.

\section{Additional details on numerical experiments}\label{sec:imp_detail}
Synthetic experiments are run on Intel Xeon Gold 6438Y+, and deep DRO experiments are run on one H100 GPU.
\subsection{Constrained linear optimization}
\subsubsection{Instance generation}\label{sec:clo_instances}
We first randomly generate $c$ and $d$ with all the entries independently chosen from $\mathcal{N}(0,1)$. We then randomly generate $\widetilde A$ and $\widetilde B$ with all the entries independently chosen from $\mathcal{N}(0,0.01)$. In addition, we randomly generate $\hat y\in[-1,1]^m$ with all the entries independently chosen from $\mathcal{N}(0,0.1)$ and then projected to $[-1,1]^m$ and choose $\tilde d$ and $\tilde b$ based on the KKT conditions that make $\hat y$ the optimal solution of the lower level optimization of Problem~\eqref{eq:clo} with $x=0$. More specifically, $\tilde d$ and $\tilde b$ are chosen as follows. First notice that $\hat y$ already satisfies the box constraint. We can construct the Lagrangian of the lower-level problem \[
\mathcal L(y,\lambda)
=
\tilde d^\top y
+
\lambda^\top(\tilde B y-\tilde b),
\] Then we generate $\lambda$ randomly with $
\lambda\sim|\mathcal N(0,I)|$. We then set $\tilde b$ and $\tilde d$ such that 
\[
\nabla_y\mathcal L(\hat y,\lambda)
=
\tilde d+\tilde B^\top\lambda=0, \quad \tilde B\hat y-\tilde b=0\le 0.
\] 
This way, the KKT conditions for the lower level problem hold and thus $\hat y$ is optimal. 

\subsubsection{Hyperparameters and implementations}
For our method, we used a Lagrange bound of $B=200$ and set input $\epsilon=0.01$ in Algorithm~\ref{alg:first-order-new}. Hyperparameters of FOP and SMO, e.g., $\epsilon_k$, $\mu_k$ and $\rho_k$, are set based on the description of~\citet[Section 3]{lu2025solving}. It is worth mentioning that for the FOP baseline, they did not implement the pseudocode of ~\citet[Algorithm 2]{lu2024firstorder} but instead used a practical variant with a schedule for ($\epsilon_k,\mu_k,\rho_k$) described in their experiments section. We similarly implemented the variant for the FOP baseline. For our method and each instance, we tune the Lipschitz constant $L_{\nabla P_1}$ used to construct the step sizes by performing a grid search on $\{1,5,6,7,8,9,10,100\}$ based on the output performance. We find that $6$ is the best value for all instances. 
For the two baselines, since the Lipschitz constants of their subproblems change with the outer iteration due to the dynamic penalty parameters, we set their Lipschitz constant to be $c\cdot L_k$, where $L_k$ is the theoretical Lipschitz constant for the subproblem in their $k$th outer iteration, and $c$ is chosen from $\{0.01,0.05, 0.1, 0.2\}$ by a grid search based on the output performance. 

By Step 2 of Algorithm 1 in \citep{lu2025solving},  SMO needs to solve the following the subproblem up to an optimality gap $\epsilon_k$ in outer iteration $k$:
$$
\min_{\vy_1}\bar{f}(\vx_1^{(k)}, \vy_1)+\bar{f}_2(\vy_1) + \frac{1}{2\rho_k\mu_k}\left\|[\vz_2^{(k)}+\mu_k\bar{\vg}(\vx_1^{(k)}, \vy_1)]_+\right\|^2.
$$
The goal is to obtain a good initial solution $\vy_1^{(k)}$ for the  minimax problem in iteration $k$. Following \citep{lu2025solving}, we solve this problem using the optimal first-order method given in \citep{lu2025solving}[Algorithm A.2], which guarantees an $\epsilon_k$ optimality gap at termination. For the same purpose, FOP needs to solve the following subproblem up to an optimality gap $\epsilon_k$ in outer iteration $k$:
$$
\min_{\vy_1}\bar{f}(\vx_1^{(k)}, \vy_1)+\bar{f}_2(\vy_1) + \mu_k\left\|[\bar{\vg}(\vx_1^{(k)}, \vy_1)]_+\right\|^2.
$$
We also solve it by \citep{lu2025solving}[Algorithm A.2]. For the similar reason, we also need to obtain a good initial primal-dual solution $(\vy_1^{(0)},\vy_2^{(0)})$ but we need it only at the beginning of Algorithm~\ref{alg:first-order-new} instead of each outer iteration. We use the same $\vy_1^{(0)}$ as the one solved from the FOP's subproblem above with $k=0$ and set $\vy_2^{(0)}=\mu_0\nabla \bar{\vg}(\vx_1^{(0)}, \vy_1^{(0)})[\bar{\vg}(\vx_1^{(0)}, \vy_1^{(0)})]_+$, where $\mu_0$ is set as in FOP. 


All three methods use Algorithm~\ref{mmax-alg1} as a subroutine. We set the maximum number of outer iterations of Algorithm~\ref{mmax-alg1} to 200 and set the maximum number of inner iterations (Line 11) to 10000 for each method. Each of the three methods involves a precision-related parameter $\epsilon_k$. We terminate each method once its $\epsilon_k\leq 0.01$, 
$\|[\bar{\vg}(\vx_1^{(k)},\vy_1^{(k)})]_+\|\le 0.01$ and $\bar{f} (\vx_1^{(k)},\vy_1^{(k)}) - \bar{f}^*(\vx_1^{(k)}) \le 0.01$. 

\subsubsection{Additional results}

In Table \ref{tab:multiplier_box_gap},  we compare the bound for the dual variables, i.e., $B$, used in our experiments with its theoretical values in \eqref{eq:cons_y1} based on the lower-level Slater margin of the generated instances. To do so, we compute $G=-\min_{\|z\|_{\infty}\leq1}\max_{i=1}^l\{\tilde B_iz-\tilde b_i\}$
as an estimate of the Slater margin and present its value together with the theoretical multiplier bounds calculated as \eqref{eq:cons_y1} and 
the multiplier bound we used in experiments. The values in Table \ref{tab:multiplier_box_gap} are the averaged values over the five random instances.

\begin{table}[htbp]
\centering
\begin{tabular}{cccc}
\toprule
$(n,m,l)$ & \textbf{Slater Margin}
& \textbf{Theoretical $B$} & \textbf{$B$ Used} \\
\midrule
$(100,100,5)$   & 0.34 & 32.86  & 200 \\
$(200,200,10)$  & 0.49 & 52.60  & 200 \\
$(300,300,15)$  & 0.58 & 60.49  & 200 \\
$(400,400,20)$  & 0.70 & 93.31  & 200 \\
$(500,500,25)$  & 0.78 & 126.65 & 200 \\
\bottomrule
\end{tabular}
\vspace{1em}
\caption{Slater margin and theoretical $B$ for different instances.}
\label{tab:multiplier_box_gap}
\end{table}

In Table \ref{tab:theoretical_lip}, we calculate the theoretical Lipschitz constant $L_{\nabla P_1}$ 
derived in in Appendix \ref{sec:proof_thmcons}, i.e., 
$$ L_{\nabla P_1} = L_{\nabla f_1} + 2\rho(L_{\nabla \overline{f}_1} + B\sqrt{l}L_{\nabla \bar{\vg}} + 2L_{\bar{\vg}}) $$
for each instance. Each value in Table \ref{tab:theoretical_lip} is the average over the five random instances. We observe that the gap between the theoretical values and the one we used (i.e., 6) increases with the problem size.

\begin{table}[htbp]
\centering
\begin{tabular}{cc}
\toprule
$(n,m,l)$ & \textbf{Theoretical $L_{\nabla P_1}$} \\
\midrule
$(100,100,5)$   & 31.54 \\
$(200,200,10)$  & 45.12 \\
$(300,300,15)$  & 56.08 \\
$(400,400,20)$  & 64.42 \\
$(500,500,25)$  & 72.74 \\
\bottomrule
\end{tabular}
\vspace{1em}
\caption{Theoretical Lipschitz constant $L_{\nabla P_1}$ for different instances.}
\label{tab:theoretical_lip}
\end{table}

\subsection{Group DRO with hyperparameter tuning}
\label{sec:dro_hyper_tuning_details}

\subsubsection{Datasets}
We use the same split for both datasets as provided by~\citet{sagawa2020distributionally}. The primary task of the target label prediction is severely confounded by a spurious attribute (see Table~\ref{tab:dro_description}). In the CelebA dataset, standard ERM models tend to exploit the spurious correlation, effectively using gender as a proxy for hair color to minimize average loss. This shortcut learning leads to disastrous performance on minority group, i.e., blond male~\citep{sagawa2020distributionally}. To show the severity of this issue along the same strand, CUB images~\citep{wah2011cub} are overlaid on top of Places images~\citep{zhou2017places} with imbalanced proportions to create the Waterbirds dataset as done by~\citet{sagawa2020distributionally}, where the labels are the type of the bird (waterbirds vs.\ landbirds), confounded by the background (water vs.\ land).
Because the data is defined by the intersection of the target label and the spurious attribute, it naturally partitions into $G=4$ predefined groups. With deep learning paradigms, we use the proposed stochastic bilevel optimization from Section~\ref{sec:stochastic_bilevel}.
\begin{table}[h]
\centering
\begin{minipage}{0.4\linewidth}
\centering
\begin{tabular}{l cc}
\hline
\textbf{CelebA Groups} & $|D_{\textbf{train}}|$ & $|D_{\textbf{val}}|$ \\
\hline
Dark Hair, Female & 71,629 & 8,535 \\
Dark Hair, Male   & 66,874 & 8,276\\
Blond Hair, Female & 22,880 & 2,874\\
Blond Hair, Male & 1,387 & 182\\
\hline
\end{tabular}
\end{minipage}
\hfill
\begin{minipage}{0.55\linewidth}
\centering
\begin{tabular}{l cc}
\hline
\textbf{Waterbirds Groups} & $|D_{\textbf{train}}|$ & $|D_{\textbf{val}}|$ \\
\hline
Landbirds, land background   & 3,498&467 \\
Waterbirds, water background & 1,057&466 \\
Landbirds, water background  & 184&133 \\
Waterbirds, land background  & 56&133 \\
\hline
\end{tabular}
\end{minipage}
\vspace{1em}
\caption{Statistics for Training and Validation Sets. Group descriptions are \{label, confounder\}.} 
\label{tab:dro_description}
\end{table}

\subsubsection{Comparison to Group DRO}
Group DRO \citep{sagawa2020distributionally} developed a method to solve a minimax problem in the form of 
\begin{equation}\label{eq:group_dro}
\min_{\vx_1,\vy_1}\max_{\vy_2\in\mathcal Y}
\mathcal L_{\mathrm{train}}(\vx_1,\vy_1,\vy_2;C)
:=
\vy_2^\top\mathcal L^{\vx_1}(\vy_1)
+
\lambda(\|\vx_1\|^2+\|\vy_1\|^2)
+
C\sum_{g=1}^{G}\frac{(\vy_2)_g}{\sqrt{n_g}}.
\end{equation}
There are several differences between our formulation in \eqref{eq:dro_lower} and Group DRO \eqref{eq:group_dro}: the latter  (1) does not have the $\eta$-weighted regularization term for $\vy_2$, (2) jointly trains the entire model (both the feature extractor $\vx_1$ and the linear classifier $\vy_1$), (3) has the loss adjustment term $C/\sqrt{n_g}$ for each group $g$, and (4) uses only the training set in the optimization model. For our method, we perform tuning of the feature extractor $\vx_1$ through the loss on the validation set at the upper level, which is not exposed to the lower-level training problem just like the hold-out validation technique. Per SAPD algorithm \ref{alg:sapd}, we sample one batch per gradient with replacement, so in each outer most iteration of our method, we sample three batches. We used balanced sampling for the gradient of the upper-level validation loss and uniform sampling for the gradient of the lower-level training loss. 

\subsubsection{Training settings and sensitivity analysis}
We use ResNet-50~\cite{he2016resnet} as a feature extractor for representation learning on both datasets. To initialize the subproblem, the feature extractor ($\vx_1$) uses the train-from-scratch option and classifier ($\vy_1$) is initialized by sampling each coordinate from $\mathcal{N}(0,0.01)$.
In Table~\ref{tab:ablation}, we report the worst-group validation accuracy obtained by our methods when using different values of some tuning parameters in our algorithm. We find $T=1$, i.e., running one iteration in Algorithm~\ref{alg:sapd} to solve the proximal-point subproblem, leads to relatively good performance. We initialize $\eta$ with $1.0$ and note that gradient multiplication, while not used here, can be applied to accelerate learning.
For all experiments, we set $D_y=100$ and $\hat\epsilon=10.0$. For CelebA, we set the batch size for constructing stochastic gradients to be 256. For Waterbirds, we set the batch size to be 128. We evaluate and report the test performance using the solution from the iteration of Algorithm \ref{alg:first-order-new} that achieves the best validation performance.

For other unmentioned hyperparameters, we use the default value provided in the code by~\citet{sagawa2020distributionally}.

\begin{table}[htbp]
\centering
\begin{tabular}{crcc}
\toprule
\textbf{Hyperparameter} & \textbf{Value}
& \multicolumn{2}{c}{\textbf{Accuracy}} \\
& & CelebA & Waterbirds \\
\midrule

\multirow{3}{*}{$\rho$ in Algorithm~\ref{alg:first-order-new}}
 & 10 & $82.29 \pm 0.90$ & $83.33 \pm 3.61$ \\
 & \textbf{100} & $86.46 \pm 2.39$ & $81.25 \pm 5.41$ \\
 & 500 & $84.90 \pm 3.25$ & $39.59 \pm 10.97$ \\

\midrule
\multirow{3}{*}{$L_{\nabla P_1}$}
 & 500 & $84.90 \pm 3.25$ & $80.21 \pm 10.05$ \\
 & \textbf{1000} & $86.46 \pm 2.38$ & $81.25 \pm 5.41$ \\
 & 2000 & $87.50 \pm 4.13$ & $83.34 \pm 1.81$ \\

\midrule
\multirow{3}{*}{Total iteration $T$ in Algorithm~\ref{alg:sapd}}
 & \bf 1 & $86.46 \pm 2.39$ & $81.25 \pm 5.41$ \\
 & 3 & $34.90 \pm 5.49$ & $37.50 \pm 14.32$ \\
 & 5 & $31.77 \pm 10.41$ & $28.13 \pm 10.83$ \\

\bottomrule
\end{tabular}
\vspace{1em}
\caption{The worst-group validation accuracy using different tuning parameters. The best choice is bold.}
\label{tab:ablation}
\end{table}

\subsubsection{Additional results}
In Table~\ref{tab:dro_train_results}, we show the training accuracies by DRO and our method with one random seed. Our method performs worse on the training set, and we believe it is because the feature extractor and the hyperparameter are updated mainly using the information from the validation set. In Table \ref{tab:wall-clock-gdro}, we compare the runtimes of both methods. We want to emphasize that, although the runtime of Group DRO looks much longer in Table \ref{tab:wall-clock-gdro}, we cannot conclude that Group DRO is more time-consuming than our method. This is because the two methods run for different numbers of iterations during training. We terminate the training algorithm of Group DRO after 40 epochs, according to the original training setup implemented in their code. On the contrary, we terminate our algorithm after 100 outer iterations, which  use fewer than 400 epochs.

\begin{table}[h]
    \centering
    \begin{tabular}{cccccc}
    \toprule
    \textbf{Dataset} & \bf $\ell_2$ Penalty & \multicolumn{2}{c}{\bf Average} & \multicolumn{2}{c}{\bf Worst Group} \\
     &  & Group DRO & Ours & Group DRO & Ours \\
    \hline
    
    \multirow{2}{*}{\bf CelebA}
     & 1e-4 & 100.0 & 94.4 & 100.0 & 91.8 \\
     & 0.1 & 95.0 & 95.4 & 93.4 & 93.0 \\
    \hline
    
    \multirow{2}{*}{\bf Waterbirds}
     & 1e-4 & 100.0 & 96.1 & 100.0 & 92.6 \\
     & 1.0 & 99.1 & 95.2 & 97.5 & 90.3 \\
    \bottomrule
    \end{tabular}
    \vspace{1em}
    \caption{Average and worst-group accuracy for DRO and our method on the training set on one seed.}
    \label{tab:dro_train_results}
\end{table}

\begin{table}[ht]
\centering

\begin{tabular}{c c cc}
\toprule
\textbf{Dataset} & $\ell_2$ \textbf{Penalty}
& \textbf{Group DRO} & \textbf{Ours} \\
\midrule

\multirow{2}{*}{\textbf{CelebA}}
 & 1e-4 & 2h 37m 39s & \textbf{55m 19s} \\
 & 0.1  & 7h 38m 31s & \textbf{56m 26s} \\
\midrule

\multirow{2}{*}{\textbf{Waterbirds}}
 & 1e-4 & 1h 49m 01s & \textbf{28m 27s} \\
 & 1.0  & 1h 45m 21s & \textbf{28m 22s} \\
\bottomrule
\end{tabular}

\vspace{1em}
\caption{Wall-clock comparison between Group DRO and our method on CelebA and Waterbirds.}
\label{tab:wall-clock-gdro}
\end{table}

%% file: checklist.tex
\section*{NeurIPS Paper Checklist}

\begin{enumerate}

\item {\bf Claims}
    \item[] Question: Do the main claims made in the abstract and introduction accurately reflect the paper's contributions and scope?
    \item[] Answer: \answerYes{} 
    \item[] Justification: We give the summary of contribution in the introduction. The main theoretical contribution we claim is given in Theorem~\ref{thm:minmax}. The good computational results of our methods claimed in the abstract is justified by the numerical results in Section~\ref{sec:exp}.

\item {\bf Limitations}
    \item[] Question: Does the paper discuss the limitations of the work performed by the authors?
    \item[] Answer: \answerYes{} 
    \item[] Justification: We discuss the limitations of the work in Section~\ref{sec:limit}.

\item {\bf Theory assumptions and proofs}
    \item[] Question: For each theoretical result, does the paper provide the full set of assumptions and a complete (and correct) proof?
    \item[] Answer: \answerYes{} 
    \item[] Justification:  Each theoretical result has been rigorously proved with complete details. The proofs are provided in Sections~\ref{sec:lemma_connection}, ~\ref{sec:proof_thm1}, ~\ref{sec:proof_thmcons}, ~\ref{sec:SAPD}, ~\ref{sec:proof_thm2}, and ~\ref{sec:tech_proof}. The assumptions we need are clearly presented in Assumptions~\ref{assump:dom}, ~\ref{assume:constrained}, and~\ref{assump:stoc}.

    \item {\bf Experimental result reproducibility}
    \item[] Question: Does the paper fully disclose all the information needed to reproduce the main experimental results of the paper to the extent that it affects the main claims and/or conclusions of the paper (regardless of whether the code and data are provided or not)?
    \item[] Answer: \answerYes{} 
    \item[] Justification: We fully disclose all the details of our main experimental results in Sections~\ref{sec:exp}, and~\ref{sec:imp_detail}, which are sufficient to reproduce the results. Code will be released upon publication.

\item {\bf Open access to data and code}
    \item[] Question: Does the paper provide open access to the data and code, with sufficient instructions to faithfully reproduce the main experimental results, as described in supplemental material?
    \item[] Answer: \answerYes{} 
    \item[] Justification: In the submitted supplemental materials, we have included all the codes we developed for the numerical experiments in this paper, which are sufficient to reproduce our results. 

\item {\bf Experimental setting/details}
    \item[] Question: Does the paper specify all the training and test details (e.g., data splits, hyperparameters, how they were chosen, type of optimizer) necessary to understand the results?
    \item[] Answer: \answerYes{} 
    \item[] Justification: The details are included in the Implementation Details of the Appendix~\ref{sec:imp_detail}.

\item {\bf Experiment statistical significance}
    \item[] Question: Does the paper report error bars suitably and correctly defined or other appropriate information about the statistical significance of the experiments?
    \item[] Answer: \answerNo{} 
    \item[] Justification: When comparing with the baseline methods, we directly use the performance metrics reported in the corresponding papers. 
Since these papers do not report error bars or standard errors, we omit them as well for consistency.

\item {\bf Experiments compute resources}
    \item[] Question: For each experiment, does the paper provide sufficient information on the computer resources (type of compute workers, memory, time of execution) needed to reproduce the experiments?
    \item[] Answer: \answerYes{} 
    \item[] Justification: Provided in Appendix Implementation Details in Appendix~\ref{sec:imp_detail}.
    
\item {\bf Code of ethics}
    \item[] Question: Does the research conducted in the paper conform, in every respect, with the NeurIPS Code of Ethics \url{https://neurips.cc/public/EthicsGuidelines}?
    \item[] Answer: \answerYes{} 
    \item[] Justification: The research conducted for this paper complies with the NeurIPS Code of Ethics in all aspects. No human objects are involved, and we do not foresee any potential harms.

\item {\bf Broader impacts}
    \item[] Question: Does the paper discuss both potential positive societal impacts and negative societal impacts of the work performed?
    \item[] Answer: \answerNA{} 
    \item[] Justification: This is a theoretical paper. There's no any negative societal impact.
    
\item {\bf Safeguards}
    \item[] Question: Does the paper describe safeguards that have been put in place for responsible release of data or models that have a high risk for misuse (e.g., pre-trained language models, image generators, or scraped datasets)?
    \item[] Answer: \answerNA{} 
    \item[] Justification: This paper poses no such risks.

\item {\bf Licenses for existing assets}
    \item[] Question: Are the creators or original owners of assets (e.g., code, data, models), used in the paper, properly credited and are the license and terms of use explicitly mentioned and properly respected?
    \item[] Answer: \answerYes{} 
    \item[] Justification: All used assets are cited.

\item {\bf New assets}
    \item[] Question: Are new assets introduced in the paper well documented and is the documentation provided alongside the assets?
    \item[] Answer: \answerNA{} 
    \item[] Justification: This paper does not release new assets.

\item {\bf Crowdsourcing and research with human subjects}
    \item[] Question: For crowdsourcing experiments and research with human subjects, does the paper include the full text of instructions given to participants and screenshots, if applicable, as well as details about compensation (if any)? 
    \item[] Answer: \answerNA{} 
    \item[] Justification: This paper does not involve crowdsourcing nor research with human subjects.

\item {\bf Institutional review board (IRB) approvals or equivalent for research with human subjects}
    \item[] Question: Does the paper describe potential risks incurred by study participants, whether such risks were disclosed to the subjects, and whether Institutional Review Board (IRB) approvals (or an equivalent approval/review based on the requirements of your country or institution) were obtained?
    \item[] Answer: \answerNA{} 
    \item[] Justification: This paper does not involve crowdsourcing nor research with human subjects.

\item {\bf Declaration of LLM usage}
    \item[] Question: Does the paper describe the usage of LLMs if it is an important, original, or non-standard component of the core methods in this research? Note that if the LLM is used only for writing, editing, or formatting purposes and does \emph{not} impact the core methodology, scientific rigor, or originality of the research, declaration is not required.
    \item[] Answer: \answerNA{} 
    \item[] Justification: The core method development in this research does not involve LLMs as any important, original, or non-standard components.

\end{enumerate}

%% file: opt.bib
@inproceedings{zheng2026saddle,
  title     = {Bilevel Optimization over Saddle Points of
               Zero-Sum Markov Games},
  author    = {Zheng, Zihao and King, Irwin and Lu, Songtao},
  booktitle = {Proceedings of the 43rd International
               Conference on Machine Learning},
  series    = {Proceedings of Machine Learning Research},
  volume    = {306},
  pages     = {163179--163238},
  publisher = {PMLR},
  year      = {2026}
}

@article{zhang2026stability,
  title   = {On the Stability and Generalization of First-order Bilevel Minimax Optimization},
  author  = {Zhang, Xuelin and Yuan, Peipei},
  journal = {arXiv preprint arXiv:2604.20115},
  year    = {2026}
}

@inproceedings{gao2022value,
  title        = {Value function based difference-of-convex algorithm for bilevel hyperparameter selection problems},
  author       = {Gao, Lucy L and Ye, Jane and Yin, Haian and Zeng, Shangzhi and Zhang, Jin},
  booktitle    = {International conference on machine learning},
  pages        = {7164--7182},
  year         = {2022},
  organization = {PMLR}
}

@article{nie2025augmented,
  title   = {An Augmented Lagrangian Value Function Method for Lower-level Constrained Stochastic Bilevel Optimization},
  author  = {Nie, Hantao and Li, Jiaxiang and Wen, Zaiwen},
  journal = {arXiv preprint arXiv:2509.24249},
  year    = {2025}
}

@article{bertsekas1997nonlinear,
  title     = {Nonlinear programming},
  author    = {Bertsekas, Dimitri P},
  journal   = {Journal of the Operational Research Society},
  volume    = {48},
  number    = {3},
  pages     = {334--334},
  year      = {1997},
  publisher = {Taylor \& Francis}
}

@article{kornowski2024first,
  title   = {First-order methods for linearly constrained bilevel optimization},
  author  = {Kornowski, Guy and Padmanabhan, Swati and Wang, Kai and Zhang, Jimmy and Sra, Suvrit},
  journal = {Advances in neural information processing systems},
  volume  = {37}
}

@article{dvurechensky2026stochastic,
  title   = {Stochastic variance reduced extragradient methods for solving hierarchical variational inequalities},
  author  = {Dvurechensky, Pavel and Ebner, Andrea and Schnebel, Johannes Carl and Shtern, Shimrit and Staudigl, Mathias},
  journal = {arXiv preprint arXiv:2602.13510},
  year    = {2026}
}

@article{khalafi2025regularized,
  title   = {Regularized Operator Extrapolation Method For Stochastic Bilevel Variational Inequality Problems},
  author  = {Khalafi, Mohammad and Boob, Digvijay},
  journal = {arXiv preprint arXiv:2505.09778},
  year    = {2025}
}

@inproceedings{khanduri2023linearly,
  title        = {Linearly constrained bilevel optimization: A smoothed implicit gradient approach},
  author       = {Khanduri, Prashant and Tsaknakis, Ioannis and Zhang, Yihua and Liu, Jia and Liu, Sijia and Zhang, Jiawei and Hong, Mingyi},
  booktitle    = {International Conference on Machine Learning},
  pages        = {16291--16325},
  year         = {2023},
  organization = {PMLR}
}

@inproceedings{xu2023efficient,
  title     = {Efficient gradient approximation method for constrained bilevel optimization},
  author    = {Xu, Siyuan and Zhu, Minghui},
  booktitle = {Proceedings of the AAAI Conference on Artificial Intelligence},
  volume    = {37},
  number    = {10},
  pages     = {12509--12517},
  year      = {2023}
}

@inproceedings{yaoovercoming,
  title     = {Overcoming Lower-Level Constraints in Bilevel Optimization: A Novel Approach with Regularized Gap Functions},
  author    = {Yao, Wei and Yin, Haian and Zeng, Shangzhi and Zhang, Jin},
  booktitle = {The Thirteenth International Conference on Learning Representations},
  year      = {2025}
}

@inproceedings{yaoconstrained,
  title     = {Constrained Bi-Level Optimization: Proximal Lagrangian Value Function Approach and Hessian-free Algorithm},
  author    = {Yao, Wei and Yu, Chengming and Zeng, Shangzhi and Zhang, Jin},
  booktitle = {The Twelfth International Conference on Learning Representations},
  year      = {2024}
}

@article{jiang2024primal,
  title   = {A Primal-Dual-Assisted Penalty Approach to Bilevel Optimization with Coupled Constraints},
  author  = {Jiang, Liuyuan and Xiao, Quan and Tenorio, Victor M. and Real-Rojas, Fernando and Marques, Antonio G. and Chen, Tianyi},
  journal = {Advances in Neural Information Processing Systems},
  volume  = {37},
  pages   = {95026--95066},
  year    = {2024}
}

@inproceedings{jiang2025beyond,
  title     = {Beyond Value Functions: Single-Loop Bilevel Optimization under Flatness Conditions},
  author    = {Jiang, Liuyuan and Xiao, Quan and Chen, Lisha and Chen, Tianyi},
  booktitle = {Advances in Neural Information Processing Systems},
  volume    = {38},
  year      = {2025}
}

@article{chen2022fast,
  title   = {A fast and convergent proximal algorithm for regularized nonconvex and nonsmooth bi-level optimization},
  author  = {Chen, Ziyi and Kailkhura, Bhavya and Zhou, Yi},
  journal = {arXiv preprint arXiv:2203.16615},
  year    = {2022}
}

@inproceedings{tsaknakis2022implicit,
  title        = {An implicit gradient-type method for linearly constrained bilevel problems},
  author       = {Tsaknakis, Ioannis and Khanduri, Prashant and Hong, Mingyi},
  booktitle    = {ICASSP 2022-2022 IEEE International Conference on Acoustics, Speech and Signal Processing (ICASSP)},
  pages        = {5438--5442},
  year         = {2022},
  organization = {IEEE}
}

@article{arbel2022non,
  title   = {Non-convex bilevel games with critical point selection maps},
  author  = {Arbel, Michael and Mairal, Julien},
  journal = {Advances in Neural Information Processing Systems},
  volume  = {35},
  pages   = {8013--8026},
  year    = {2022}
}

@inproceedings{ouattara2018duality,
  title        = {Duality approach to bilevel programs with a convex lower level},
  author       = {Ouattara, Aur{\'e}lien and Aswani, Anil},
  booktitle    = {2018 Annual American Control Conference (ACC)},
  pages        = {1388--1395},
  year         = {2018},
  organization = {IEEE}
}

@article{huang2023momentum,
  title   = {On momentum-based gradient methods for bilevel optimization with nonconvex lower-level},
  author  = {Huang, Feihu},
  journal = {arXiv preprint arXiv:2303.03944},
  year    = {2023}
}

@article{huang2022enhanced,
  title   = {Enhanced bilevel optimization via bregman distance},
  author  = {Huang, Feihu and Li, Junyi and Gao, Shangqian and Huang, Heng},
  journal = {Advances in Neural Information Processing Systems},
  volume  = {35},
  pages   = {28928--28939},
  year    = {2022}
}

@article{huang2023adaptive,
  title   = {Adaptive mirror descent bilevel optimization},
  author  = {Huang, Feihu},
  journal = {arXiv preprint arXiv:2311.04520},
  year    = {2023}
}

@inproceedings{huang2024optimal,
  title     = {Optimal Hessian/Jacobian-free nonconvex-PL bilevel optimization},
  author    = {Huang, Feihu},
  booktitle = {Proceedings of the 41st International Conference on Machine Learning},
  pages     = {19598--19621},
  year      = {2024}
}

@article{liu2022general,
  title     = {A general descent aggregation framework for gradient-based bi-level optimization},
  author    = {Liu, Risheng and Mu, Pan and Yuan, Xiaoming and Zeng, Shangzhi and Zhang, Jin},
  journal   = {IEEE Transactions on Pattern Analysis and Machine Intelligence},
  volume    = {45},
  number    = {1},
  pages     = {38--57},
  year      = {2022},
  publisher = {IEEE}
}

@inproceedings{liu2021value,
  title        = {A value-function-based interior-point method for non-convex bi-level optimization},
  author       = {Liu, Risheng and Liu, Xuan and Yuan, Xiaoming and Zeng, Shangzhi and Zhang, Jin},
  booktitle    = {International conference on machine learning},
  pages        = {6882--6892},
  year         = {2021},
  organization = {PMLR}
}

@article{liu2023value,
  title     = {Value-function-based sequential minimization for bi-level optimization},
  author    = {Liu, Risheng and Liu, Xuan and Zeng, Shangzhi and Zhang, Jin and Zhang, Yixuan},
  journal   = {IEEE Transactions on Pattern Analysis and Machine Intelligence},
  volume    = {45},
  number    = {12},
  pages     = {15930--15948},
  year      = {2023},
  publisher = {IEEE}
}

@article{liu2021towards,
  title   = {Towards gradient-based bilevel optimization with non-convex followers and beyond},
  author  = {Liu, Risheng and Liu, Yaohua and Zeng, Shangzhi and Zhang, Jin},
  journal = {Advances in Neural Information Processing Systems},
  volume  = {34},
  pages   = {8662--8675},
  year    = {2021}
}

@inproceedings{liu2023averaged,
  title        = {Averaged method of multipliers for bi-level optimization without lower-level strong convexity},
  author       = {Liu, Risheng and Liu, Yaohua and Yao, Wei and Zeng, Shangzhi and Zhang, Jin},
  booktitle    = {International Conference on Machine Learning},
  pages        = {21839--21866},
  year         = {2023},
  organization = {PMLR}
}

@inproceedings{kwon2024penalty,
  title     = {On penalty methods for nonconvex bilevel optimization and first-order stochastic approximation},
  author    = {Kwon, Jeongyeol and Kwon, Dohyun and Wright, Stephen and Nowak, Robert},
  booktitle = {International Conference on Learning Representations},
  volume    = {2024},
  pages     = {47473--47518},
  year      = {2024}
}

@article{xiao2022alternating,
  title   = {Alternating implicit projected sgd and its efficient variants for equality-constrained bilevel optimization},
  author  = {Xiao, Quan and Shen, Han and Yin, Wotao and Chen, Tianyi},
  journal = {arXiv preprint arXiv:2211.07096},
  year    = {2022}
}

@article{xiao2023alternating,
  title   = {An alternating optimization method for bilevel problems under the Polyak-{\L}ojasiewicz condition},
  author  = {Xiao, Quan and Lu, Songtao and Chen, Tianyi},
  journal = {Advances in Neural Information Processing Systems},
  volume  = {36},
  pages   = {63847--63873},
  year    = {2023}
}

@inproceedings{chen2024finding,
  title        = {On finding small hyper-gradients in bilevel optimization: Hardness results and improved analysis},
  author       = {Chen, Lesi and Xu, Jing and Zhang, Jingzhao},
  booktitle    = {The Thirty Seventh Annual Conference on Learning Theory},
  pages        = {947--980},
  year         = {2024},
  organization = {PMLR}
}

@article{yang2021provably,
  title   = {Provably faster algorithms for bilevel optimization},
  author  = {Yang, Junjie and Ji, Kaiyi and Liang, Yingbin},
  journal = {Advances in Neural Information Processing Systems},
  volume  = {34},
  pages   = {13670--13682},
  year    = {2021}
}

@article{doi:10.1137/S1052623403425629,
  author  = {Nemirovski, Arkadi},
  title   = {Prox-Method with Rate of Convergence O(1/t) for Variational Inequalities with Lipschitz Continuous Monotone Operators and Smooth Convex-Concave Saddle Point Problems},
  journal = {SIAM Journal on Optimization},
  volume  = {15},
  number  = {1},
  pages   = {229-251},
  year    = {2004}
}

@article{JMLR:v26:23-1104,
  author  = {Lesi Chen and Yaohua Ma and Jingzhao Zhang},
  title   = {Near-Optimal Nonconvex-Strongly-Convex Bilevel Optimization with Fully First-Order Oracles},
  journal = {Journal of Machine Learning Research},
  year    = {2025},
  volume  = {26},
  number  = {109},
  pages   = {1--56},
  url     = {http://jmlr.org/papers/v26/23-1104.html}
}

@inproceedings{kwon2023fully,
  title        = {A fully first-order method for stochastic bilevel optimization},
  author       = {Kwon, Jeongyeol and Kwon, Dohyun and Wright, Stephen and Nowak, Robert D},
  booktitle    = {International Conference on Machine Learning},
  pages        = {18083--18113},
  year         = {2023},
  organization = {PMLR}
}

@article{shen2025penalty,
  title     = {On penalty-based bilevel gradient descent method},
  author    = {Shen, Han and Xiao, Quan and Chen, Tianyi},
  journal   = {Mathematical Programming},
  pages     = {1--51},
  year      = {2025},
  publisher = {Springer}
}

@inproceedings{shen2023penalty,
  title        = {On penalty-based bilevel gradient descent method},
  author       = {Shen, Han and Chen, Tianyi},
  booktitle    = {International conference on machine learning},
  pages        = {30992--31015},
  year         = {2023},
  organization = {PMLR}
}

@article{zhang2024robust,
  title     = {Robust accelerated primal-dual methods for computing saddle points},
  author    = {Zhang, Xuan and Aybat, Necdet Serhat and G{\"u}rb{\"u}zbalaban, Mert},
  journal   = {SIAM Journal on Optimization},
  volume    = {34},
  number    = {1},
  pages     = {1097--1130},
  year      = {2024},
  publisher = {SIAM}
}

@article{zhang2022sapd+,
  title   = {Sapd+: An accelerated stochastic method for nonconvex-concave minimax problems},
  author  = {Zhang, Xuan and Aybat, Necdet Serhat and Gurbuzbalaban, Mert},
  journal = {Advances in Neural Information Processing Systems},
  volume  = {35},
  pages   = {21668--21681},
  year    = {2022}
}

@inproceedings{finn2017model,
  title        = {Model-agnostic meta-learning for fast adaptation of deep networks},
  author       = {Finn, Chelsea and Abbeel, Pieter and Levine, Sergey},
  booktitle    = {International conference on machine learning},
  pages        = {1126--1135},
  year         = {2017},
  organization = {PMLR}
}

@article{collins2020task,
  title   = {Task-robust model-agnostic meta-learning},
  author  = {Collins, Liam and Mokhtari, Aryan and Shakkottai, Sanjay},
  journal = {Advances in Neural Information Processing Systems},
  volume  = {33},
  pages   = {18860--18871},
  year    = {2020}
}

@article{rajeswaran2019meta,
  title   = {Meta-learning with implicit gradients},
  author  = {Rajeswaran, Aravind and Finn, Chelsea and Kakade, Sham M and Levine, Sergey},
  journal = {Advances in neural information processing systems},
  volume  = {32},
  year    = {2019}
}

@article{duchi2021learning,
  title     = {Learning models with uniform performance via distributionally robust optimization},
  author    = {Duchi, John C and Namkoong, Hongseok},
  journal   = {The Annals of Statistics},
  volume    = {49},
  number    = {3},
  pages     = {1378--1406},
  year      = {2021},
  publisher = {Institute of Mathematical Statistics}
}

@article{goodfellow2014generative,
  title   = {Generative adversarial nets},
  author  = {Goodfellow, Ian J and Pouget-Abadie, Jean and Mirza, Mehdi and Xu, Bing and Warde-Farley, David and Ozair, Sherjil and Courville, Aaron and Bengio, Yoshua},
  journal = {Advances in neural information processing systems},
  volume  = {27},
  year    = {2014}
}

@inproceedings{madry2018towards,
  title     = {Towards Deep Learning Models Resistant to Adversarial Attacks},
  author    = {Madry, Aleksander and Makelov, Aleksandar and Schmidt, Ludwig and Tsipras, Dimitris and Vladu, Adrian},
  booktitle = {International Conference on Learning Representations},
  year      = {2018}
}

@article{ben2013robust,
  title     = {Robust solutions of optimization problems affected by uncertain probabilities},
  author    = {Ben-Tal, Aharon and Den Hertog, Dick and De Waegenaere, Anja and Melenberg, Bertrand and Rennen, Gijs},
  journal   = {Management Science},
  volume    = {59},
  number    = {2},
  pages     = {341--357},
  year      = {2013},
  publisher = {INFORMS}
}

@article{li2024learning,
  title   = {Learning reward and policy jointly from demonstration and preference improves alignment},
  author  = {Li, Chenliang and Zeng, Siliang and Liao, Zeyi and Li, Jiaxiang and Kang, Dongyeop and Garcia, Alfredo and Hong, Mingyi},
  journal = {arXiv preprint arXiv:2406.06874},
  year    = {2024}
}

@article{lu2025solving,
  author  = {Lu, Zhaosong and Mei, Sanyou},
  title   = {Solving Bilevel Optimization via Sequential Minimax Optimization},
  journal = {Mathematics of Operations Research},
  year    = {2026},
  doi     = {10.1287/moor.2024.0521},
  url     = {https://doi.org/10.1287/moor.2024.0521},
  note    = {Published online January 6, 2026}
}

@article{lu2024firstorder,
  author  = {Zhaosong Lu and Sanyou Mei},
  title   = {First-Order Penalty Methods for Bilevel Optimization},
  journal = {SIAM Journal on Optimization},
  volume  = {34},
  number  = {2},
  pages   = {1937--1969},
  year    = {2024},
  doi     = {10.1137/23M1566753},
  url     = {https://doi.org/10.1137/23M1566753}
}

@article{allende2013solving,
  title     = {Solving bilevel programs with the KKT-approach},
  author    = {Allende, Gemayqzel Bouza and Still, Georg},
  journal   = {Mathematical programming},
  volume    = {138},
  number    = {1},
  pages     = {309--332},
  year      = {2013},
  publisher = {Springer}
}

@inproceedings{bennett2008bilevel,
  author    = {Bennett, Kristin P. and Kunapuli, Gautam and Hu, Jing and Pang, Jong-Shi},
  title     = {Bilevel optimization and machine learning},
  booktitle = {IEEE World Congress on Computational Intelligence},
  pages     = {25--47},
  publisher = {Springer},
  year      = {2008}
}

@inproceedings{bertinetto2018meta,
  author    = {Bertinetto, Luca and Henriques, Joao F. and Torr, Philip and Vedaldi, Andrea},
  title     = {Meta-learning with differentiable closed-form solvers},
  booktitle = {International Conference on Learning Representations},
  year      = {2018}
}

@inproceedings{chen2022single,
  author    = {Chen, Tianyi and Sun, Yuejiao and Xiao, Quan and Yin, Wotao},
  title     = {A single-timescale method for stochastic bilevel optimization},
  booktitle = {International Conference on Artificial Intelligence and Statistics},
  pages     = {2466--2488},
  year      = {2022}
}

@article{chen2021closing,
  title   = {Closing the gap: Tighter analysis of alternating stochastic gradient methods for bilevel problems},
  author  = {Chen, Tianyi and Sun, Yuejiao and Yin, Wotao},
  journal = {Advances in Neural Information Processing Systems},
  volume  = {34},
  pages   = {25294--25307},
  year    = {2021}
}

@inproceedings{franceschi2018bilevel,
  author    = {Franceschi, Luca and Frasconi, Paolo and Salzo, Saverio and Grazzi, Riccardo and Pontil, Massimiliano},
  title     = {Bilevel programming for hyperparameter optimization and meta-learning},
  booktitle = {International Conference on Machine Learning},
  pages     = {1568--1577},
  year      = {2018}
}

@article{ghadimi2018approximation,
  author  = {Ghadimi, Saeed and Wang, Mengdi},
  title   = {Approximation methods for bilevel programming},
  journal = {arXiv preprint arXiv:1802.02246},
  year    = {2018}
}

@inproceedings{grazzi2020iteration,
  author    = {Grazzi, Riccardo and Franceschi, Luca and Pontil, Massimiliano and Salzo, Saverio},
  title     = {On the iteration complexity of hypergradient computation},
  booktitle = {International Conference on Machine Learning},
  pages     = {3748--3758},
  year      = {2020}
}

@article{guo2021randomized,
  author  = {Guo, Zhishuai and Hu, Quanqi and Zhang, Lijun and Yang, Tianbao},
  title   = {Randomized stochastic variance-reduced methods for multi-task stochastic bilevel optimization},
  journal = {arXiv preprint arXiv:2105.02266},
  year    = {2021}
}

@article{yang2024bilevel,
  title   = {Bilevel reinforcement learning via the development of hyper-gradient without lower-level convexity},
  author  = {Yang, Yan and Gao, Bin and Yuan, Ya-xiang},
  journal = {arXiv preprint arXiv:2405.19697},
  year    = {2024}
}

@article{hong2023two,
  author  = {Hong, Mingyi and Wai, Hoi-To and Wang, Zhaoran and Yang, Zhuoran},
  title   = {A two-timescale stochastic algorithm framework for bilevel optimization},
  journal = {SIAM Journal on Optimization},
  volume  = {33},
  number  = {1},
  pages   = {147--180},
  year    = {2023}
}

@article{hu2023improved,
  title     = {An improved unconstrained approach for bilevel optimization},
  author    = {Hu, Xiaoyin and Xiao, Nachuan and Liu, Xin and Toh, Kim-Chuan},
  journal   = {SIAM Journal on Optimization},
  volume    = {33},
  number    = {4},
  pages     = {2801--2829},
  year      = {2023},
  publisher = {SIAM}
}

@inproceedings{ji2021bilevel,
  author    = {Ji, Kaiyi and Yang, Junjie and Liang, Yingbin},
  title     = {Bilevel optimization: Convergence analysis and enhanced design},
  booktitle = {ICML},
  pages     = {4882--4892},
  year      = {2021}
}

@inproceedings{khanduri2021near,
  author    = {Khanduri, Prashant and Zeng, Siliang and Hong, Mingyi and Wai, Hoi-To and Wang, Zhaoran and Yang, Zhuoran},
  title     = {A near-optimal algorithm for stochastic bilevel optimization via double-momentum},
  booktitle = {NeurIPS},
  year      = {2021}
}

@inproceedings{kovalev2022first,
  author    = {Kovalev, Dmitry and Gasnikov, Alexander},
  title     = {The first optimal algorithm for smooth minimax optimization},
  booktitle = {NeurIPS},
  year      = {2022}
}

@inproceedings{li2022fully,
  title     = {A fully single loop algorithm for bilevel optimization without hessian inverse},
  author    = {Li, Junyi and Gu, Bin and Huang, Heng},
  booktitle = {Proceedings of the AAAI Conference on Artificial Intelligence},
  volume    = {36},
  number    = {7},
  pages     = {7426--7434},
  year      = {2022}
}

@article{li2023novel,
  title   = {A novel approach for bilevel programs based on Wolfe duality},
  author  = {Li, Yuwei and Lin, Gui-Hua and Zhang, Jin and Zhu, Xide},
  journal = {arXiv preprint arXiv:2302.06838},
  year    = {2023}
}

@inproceedings{liu2022bome,
  author    = {Liu, Bo and Ye, Mao and Wright, Stephen and Stone, Peter and Liu, Qiang},
  title     = {BOME! bilevel optimization made easy},
  booktitle = {NeurIPS},
  year      = {2022}
}

@inproceedings{liu2018darts,
  author    = {Liu, Hanxiao and Simonyan, Karen and Yang, Yiming},
  title     = {DARTS: Differentiable architecture search},
  booktitle = {ICLR},
  year      = {2018}
}

@article{lu2023first,
  author  = {Zhaosong Lu and Sanyou Mei},
  title   = {A First-Order Augmented Lagrangian Method for Constrained Minimax Optimization},
  journal = {Mathematical Programming},
  volume  = {213},
  number  = {1},
  pages   = {1063--1104},
  year    = {2025},
  month   = sep,
  doi     = {10.1007/s10107-024-02163-3},
  url     = {https://doi.org/10.1007/s10107-024-02163-3}
}

@inproceedings{pedregosa2016hyperparameter,
  author    = {Pedregosa, Fabian},
  title     = {Hyperparameter optimization with approximate gradient},
  booktitle = {ICML},
  year      = {2016}
}

@article{sow2022primal,
  title   = {A primal-dual approach to bilevel optimization with multiple inner minima},
  author  = {Sow, Daouda and Ji, Kaiyi and Guan, Ziwei and Liang, Yingbin},
  journal = {arXiv preprint arXiv:2203.01123},
  year    = {2022}
}

@article{ye2023difference,
  title     = {Difference of convex algorithms for bilevel programs with applications in hyperparameter selection},
  author    = {Ye, Jane J. and Yuan, Xiaoming and Zeng, Shangzhi and Zhang, Jin},
  journal   = {Mathematical Programming},
  volume    = {198},
  number    = {2},
  pages     = {1583--1616},
  year      = {2023},
  publisher = {Springer}
}

@inproceedings{sagawa2020distributionally,
  author    = {Shiori Sagawa and Pang Wei Koh and Tatsunori B. Hashimoto and Percy Liang},
  title     = {Distributionally Robust Neural Networks for Group Shifts: {On} the Importance of Regularization for Worst-Case Generalization},
  booktitle = {International Conference on Learning Representations (ICLR)},
  year      = {2020}
}

@inproceedings{ren2018learning,
  title        = {Learning to Reweight Examples for Robust Deep Learning},
  author       = {Ren, Mengye and Zeng, Wenyuan and Yang, Bin and Urtasun, Raquel},
  booktitle    = {International Conference on Machine Learning (ICML)},
  pages        = {4334--4343},
  year         = {2018},
  organization = {PMLR}
}

@inproceedings{kirichenko2023last,
  title     = {Last Layer Re-Training is Sufficient for Robustness to Spurious Correlations},
  author    = {Polina Kirichenko and Pavel Izmailov and Andrew Gordon Wilson},
  booktitle = {International Conference on Learning Representations (ICLR)},
  year      = {2023}
}

@inproceedings{kang2019decoupling,
  title     = {Decoupling Representation and Classifier for Long-Tailed Recognition},
  author    = {Kang, Bingyi and Xie, Saining and Rohrbach, Marcus and Yan, Zhicheng and Gordo, Albert and Feng, Jiashi and Kalantidis, Yannis},
  booktitle = {International Conference on Learning Representations (ICLR)},
  year      = {2020}
}

@inproceedings{liu2015faceattributes,
  title     = {Deep Learning Face Attributes in the Wild},
  author    = {Liu, Ziwei and Luo, Ping and Wang, Xiaogang and Tang, Xiaoou},
  booktitle = {Proceedings of the IEEE International Conference on Computer Vision (ICCV)},
  pages     = {3730--3738},
  year      = {2015}
}

@article{zhou2017places,
  author  = {Zhou, Bolei and Lapedriza, Àgata and Khosla, Aditya and Oliva, Aude and Torralba, Antonio},
  year    = {2017},
  month   = {07},
  pages   = {1-1},
  title   = {Places: A 10 Million Image Database for Scene Recognition},
  volume  = {PP},
  journal = {IEEE Transactions on Pattern Analysis and Machine Intelligence},
  doi     = {10.1109/TPAMI.2017.2723009}
}

@techreport{wah2011cub,
  title       = {The Caltech-UCSD Birds-200-2011 Dataset},
  author      = {Wah, C. and Branson, S. and Welinder, P. and Perona, P. and Belongie, S.},
  year        = {2011},
  institution = {California Institute of Technology},
  number      = {CNS-TR-2011-001}
}

@inproceedings{he2016resnet,
  title     = {Deep residual learning for image recognition},
  author    = {He, Kaiming and Zhang, Xiangyu and Ren, Shaoqing and Sun, Jian},
  booktitle = {Proceedings of the IEEE conference on computer vision and pattern recognition},
  pages     = {770--778},
  year      = {2016}
}
